\documentclass[12pt]{article}
\usepackage{color,amsmath,amsfonts,amssymb,epsfig,latexsym}
\usepackage{graphicx}
\usepackage[latin1]{inputenc}
\usepackage{amssymb}
\usepackage{bm}
\usepackage{color}

\usepackage{pgf,tikz}
\usepackage{mathrsfs}
\usetikzlibrary{arrows}
\usepackage{leftidx}

\usepackage{longtable}

\usepackage{blindtext, subfig}
\usepackage{dblfloatfix} 

\newcommand{\dt}{{\rm d} t }

\newcommand{\R}{\mbox{\F R}}

\renewcommand{\d}{{\rm d} }

\def\R{{\mathbb R}}
\def\N{{\mathbb N}}
\def\T{{\mathbb T}}
\def\Q{{\mathbb Q}}
\def\P{{\mathbb P}}
\def\J{{\mathbb J}}
\def\I{{\mathbb I}}
\def\D{{\mathbb D}}
\def\divg{\mathop{\rm div}\nolimits}

\def\bpsi{\boldsymbol\psi}

\def\bphi{\boldsymbol\varphi}

\def\bn{{\bf n}}
\def\bv{{\bf v}}

\def\bu{{\bf u}}
\def\bU{{\bf U}}
\def\bq{{\bf q}}
\def\bx{{\bf x}}
\def\by{{\bf y}}
\def\bX{{\bf X}}
\def\bY{{\bf Y}}
\def\ba{{\bf a}}
\def\bA{{\bf A}}
\def\bB{{\bf B}}

\def\by{{\bf y}}
\def\bn{{\bf n}}
\def\bN{{\bf N}}

\def\bpsi{\boldsymbol\psi}

\def\bomega{\boldsymbol\omega}
\def\bOmega{\boldsymbol\Omega}

\def\bgamma{\boldsymbol\gamma}

\def\bphi{\boldsymbol\varphi}

\def\bPsi{\boldsymbol\Psi}
\newtheorem{teo}{Theorem}[section]
\newtheorem{definition}{Definition}[section]
\newenvironment{proof}[1][Proof]{\medskip\noindent\textit{#1. }\upshape}{\medskip}
\newtheorem{lemma}{Lemma}[section]
\newtheorem{corollary}{Corollary}[section]
\newtheorem{remark}{Remark}[section]
\newtheorem{proposition}{Proposition}[section]
\numberwithin{equation}{section}
\newcommand\quotient[2]{#1\,/\,#2}

\begin{document}
\date{}
\title{On the regularity of weak solutions to the fluid-rigid body interaction problem}
\author{{\large Boris Muha\textsuperscript{1}\thanks{The research of B.M. leading to these results has been supported by Croatian Science Foundation under the project IP-2018-01-3706}, \v S\'arka Ne\v casov\'a\textsuperscript{2}\thanks{The research of \v S.N. leading to these results has received funding from
	the Czech Sciences Foundation (GA\v CR),
			22-01591S.  Moreover,  \v S. N.   has been supported by  Praemium Academiae of \v S. Ne\v casov\' a. CAS is supported by RVO:67985840.
			}, Ana Rado\v sevi\'c\textsuperscript{2,3}\thanks{The research of A.R. leading to these results has been supported by Croatian Science Foundation under the project IP-2019-04-1140. Moreover, The research of A.R. leading to these results has received funding from
			the Czech Sciences Foundation (GA\v CR) 22-01591S,
			and  by  Praemium Academiae of \v S. Ne\v casov\' a.  }}\\
	{\small \textsuperscript{1} Department of Mathematics}\\
	{\small Faculty of Science}\\
	{\small University of Zagreb, Croatia}\\
	{\small borism@math.hr}\\
	{\small $^2$ Institute of Mathematics, }\\
	{\small \v Zitn\'a 25, 115 67 Praha 1, Czech Republic }\\
	{\small matus@math.cas.cz}\\
	{\small \textsuperscript{3} Department of Mathematics}\\
	{\small Faculty of Economics and Business}\\
	{\small University of Zagreb, Croatia}\\
	{\small aradosevic@efzg.hr}
}
\maketitle

We study a 3D fluid-rigid body interaction problem. The fluid flow is governed by 3D incompressible Navier-Stokes equations, while the motion of the rigid body is described by a system of ordinary differential equations describing conservation of linear and angular momentum. Our aim is to prove that any weak solution satisfying certain regularity conditions is smooth. This is a generalization of the classical result for the $3D$ incompressible Navier-Stokes equations, which says that a weak solution that additionally satisfy Prodi - Serrin $L^r-L^s$ condition is smooth. We show that in the case of fluid - rigid body the Prodi - Serrin conditions imply $W^{2,p}$ and $W^{1,p}$ regularity for the fluid velocity and fluid pressure, respectively. Moreover, we show that solutions are $C^{\infty}$ if additionally we assume that the rigid body acceleration is bounded almost anywhere in time variable.

\section{Introduction}

\subsection{Fluid - rigid body system}

Let $\Omega\subset\R^3$ be a smooth bounded domain, and $S_0\subset\Omega$ be smooth such that $d(\partial\Omega,\overline{S_0})>0$. $S_0$ represents part of the domain occupied by the rigid body at the initial state. $\Omega_F=\Omega\setminus\overline{S_0}$ is the fluid domain at the initial state which we will use as the reference domain. The unknowns of the system are fluid velocity $\bu:[0,T]\times\Omega_F(t)\to\R^3$, fluid pressure $p:[0,T]\times\Omega_F(t)\to\R$, position of the center of mass of the rigid body $\bq:[0,T]\to\R^3$ and angular velocity of the rigid body $\bomega:[0,T]\to\R^3$. Here we used the following abuse of notation which is standard in analysis of moving boundary problems:
\begin{equation}\label{FluidDomainNonCylindircal}
(0,T)\times\Omega_F(t)=\bigcup_{t\in (0,T)}\{t\}\times\Omega_F(t),
\end{equation}
where $\Omega_F(t)=\Omega\setminus\overline{S(t)}$ is the fluid domain at time $t$, and $S(t)$ is a part of the domain occupied by the rigid body at time $t$ and is defined by $\bq$ and $\bomega$ in the following way. Let $\P$ be a skew-symmetric matrix such that $\P(t)\bx=\bomega(t)\times \bx,\; \bx\in \R^{3}.$ Then rotation of the rigid body $\Q:[0,T]\to SO(3)$ is defined by relation 
\begin{equation}\label{compatibility_Q}
    \frac{d\Q}{dt}\Q^{T}=\P.
\end{equation}
The domain $S(t)$ is defined by an orientation preserving isometry
\begin{equation}	\label{is}
{\bB}(t,\by)=\bq(t)+\Q(t)(\by-\bq(0)),\qquad \by\in S_{0},\; t\in [0,T],
\end{equation}
as the set
\begin{equation}	\label{set}
S(t)=\{\bx\in \R^{3}:\ \,\bx=\bB(t,\by),\quad \by\in S_{0}\}
=\bB(t,S_{0}),\qquad t\in [0,T].
\end{equation}

The Eulerian velocity of the rigid body is given by:
\begin{equation}\label{RigidVelocity}
\bu_S(t,\bx):=
\partial_{t}\bB(t,\bB^{-1}(t,\bx))=\ba(t)+\P(t)(\bx-\bq(t))\qquad \text{for all}\quad
\bx\in S(t),
\end{equation}
where $\ba=\frac{\d}{\dt}\bq$ is the translation velocity of the rigid body.

The equations modelling dynamics of the fluid - rigid body system read as follows:
\begin{equation}\label{FSINoslip}
\begin{array}{l}
\left.
\begin{array}{l}
\partial_{t}\bu+(\bu\cdot \nabla )\bu
= \divg \left( {\T}(\bu,p)\right) , \\
\divg\bu = 0
\end{array}
\right\} \;\mathrm{in}\;(0,T)\times\Omega_{F}(t),
\\
\left.
\begin{array}{l}
\frac{\d^{2}}{\dt^{2}}\bq = -\int_{\partial S(t)}{\T}(\bu,p) \bn\,\d\bgamma(\bx), \\
\frac{\d}{\dt}(\J\bomega) = -\int_{\partial S(t)}(\bx-\bq(t))\times {\T}(\bu,p)\bn\,\d\bgamma(\bx)
\end{array}
\right\} \;\mathrm{in}\;(0,T),
\\
\bu =\frac{\d}{\dt}\bq+\bomega\times (\bx-\bq),\quad \mathrm{on}\;(0,T)\times\partial S(t),
\\
\bu = 0 \quad \mathrm{on}\, \partial \Omega,
\\
\bu(0,.)=\bu_{0}\qquad \mathrm{in}\;\Omega,\quad \bq(0)=\bq_{0},\quad \frac{\d}{\dt}\bq(0)=\ba_{0},\quad
\bomega(0)=\bomega_{0}.
\end{array}
\end{equation}
Here $\T(\bu,p) = -p\I + 2\D\bu$ is the fluid Cauchy stress tensor, where $\D\bu = \frac{1}{2}\left(\nabla\bu+\nabla\bu^T\right)$ is symmetric part of the gradient, and $\J$ is the inertial tensor defined as follows:
\begin{equation*}
\J=\int_{S(t)}(|\bx-\bq(t)|^{2}\I-(\bx-\bq(t))\otimes (\bx-\bq(t)))\,\d\bx . 
\end{equation*}
Notice that for simplicity we normalized all physical constants since their concrete values do not influence our analysis. 
\begin{remark} [About the notation]
	Throughout the paper we will denote by $\bu$ both the fluid velocity defined on $\Omega_F$ and the global velocity defined on $\Omega$. The global velocity is obtained by extending the fluid velocity by setting $\bu=\bu_S$ on $S(t)$. To avoid confusion we will always write the domain of definition.
\end{remark}

\subsection{Statement of the results}

The goal of the paper is to study the regularity of weak solution to fluid - rigid body problem \eqref{FSINoslip}. Definition and existence of finite energy weak solutions (i.e. of Leray-Hopf type) are well-known (see e.g. \cite{ConcaRigid00}). Here for the convenience of reader, we recall the definition of weak solution:
First we define a function space
\begin{equation}\label{FluidFS}
V(t)=\{\bv\in H^1_0(\Omega):\divg\bv=0,\; \D\bv=0\;{\rm in}\; S(t)\},
\end{equation}
and a weak solution is given by the following definition

\begin{definition}[\cite{BorisSarkaAna2020weak}]\label{definition}
	The couple $\left(\bu, \bB\right)$
	is a weak solution to the system \eqref{FSINoslip} if the following conditions
	are satisfied:
	\begin{enumerate}
	\item The function $\bB(t,\cdot ):\R^{3}\rightarrow \R^{3}$ \ is an orientation preserving isometry given by the formula \eqref{is}, which defines a time-dependent set $S(t)=\bB(t,S)$.
	The isometry $\bB$ is compatible with $\bu=\bu_S$ on $S(t)$ in the following sense: the rigid part of velocity $\bu$, denoted by $\bu_S$, satisfies condition \eqref{RigidVelocity}, and 
	$\bq,\; \Q$ are absolutely continuous on $\left[ 0,T\right]$ and satisfy \eqref{compatibility_Q} with $\ba=\frac{\d}{\dt}\bq$.
	\item The function $\bu\in L^{2}(0,T;V(t))\cap L^{\infty
	}(0,T;L^2(\Omega ))$ satisfies the integral equality
	\begin{multline} \label{weak}
	\int_{0}^{T}\int_{\Omega}\{
	\bu \cdot \partial_{t}\bphi
	+ (\bu\otimes \bu) :\D\bphi
	- 2\,\D\bu:\D\bphi\,\}\, \d\bx\d t
	- \int_{\Omega }\bu(T)\bphi(T)\,\d\bx
	= - \int_{\Omega }\bu_{0}\bphi(0)\,\d\bx,
	\end{multline}
	which holds for any test function $\bphi\in H^1(0,T;V(t))$.
	\item The  energy inequality
	\begin{equation*}	\label{EnergyInequality}
	\frac{1}{2}\|\bu(t)\|_{L^2(\Omega)}^2
	+ 2\int_{0}^{t}
	\int_{\Omega_{F}(\tau)}\,|\D \bu|^{2}
	\,\d\bx\,\d\tau
	\leq
	\frac{1}{2}\|\bu_0\|_{L^2(\Omega)}^2.
	\end{equation*}
	holds for almost every $t\in(0,T)$.
	\end{enumerate}
\end{definition}

Now we state the main result of the paper.

\begin{teo}\label{mainResult}	
Let $(\bu,\bB)$ be a weak solution to the system \eqref{FSINoslip}. Assume that $d(S(t),\partial\Omega)>\delta$, for some constant $\delta>0$. 
If
$\frac{\d}{\dt}\ba,\frac{\d}{\dt}\bomega\in L^{\infty}(0,T)$,
and $\bu$ satisfies Prodi-Serrin condition
\begin{equation}\label{LrLs}
\bu\in L^r(0,T;L^s(\Omega_F(t)))\quad
\text{ for some } s,r \text{ such that }\quad
\frac{3}{s}+\frac{2}{r}=1,\, s\in(3,+\infty)
\end{equation}
then
\begin{equation*}
\bu\in C^{\infty}((0,T]\times\Omega_F(t)),\;\ba,\bomega\in C^{\infty}((0,T]).
\end{equation*}
\end{teo}

It is a classical result in the theory of Navier-Stokes equations that any weak solution to $3D$ Navier-Stokes equations satisfying condition \eqref{LrLs} is smooth (see e.g. \cite{escauriaza2003img} where also critical case $s=\infty$ and $r=3$ was solved, and references within). Our result is generalization of this classical result to the fluid - rigid body system \eqref{FSINoslip}. Notice that we have two extra assumptions. The first assumption says that rigid body do not  touch the boundary of $\Omega$, i.e. the fluid domain does not degenerate. It is well known (e.g. \cite{Star04}) that this condition is necessary since the domain degeneration leads to non-smooth solutions. The second condition, i.e. boundedness of the rigid body acceleration, is somewhat unexpected, and we will later elaborate more on technical reasons that give rise to that condition. In absence of this condition we can show that solutions are strong:

\begin{teo}\label{mainResult2}	
Let $(\bu,\bB)$ be a weak solution to the system \eqref{FSINoslip}. Assume that $d(S(t),\partial\Omega)>\delta$, for some constant $\delta>0$. 
If $\bu$ satisfies Prodi-Serrin condition \eqref{LrLs}
then
\begin{align}\label{LpStrong}
&\bu \in L^{p}(\varepsilon,T;W^{2,p}(\Omega_F(t)))\cap
W^{1,p}(\varepsilon,T;L^{p}(\Omega_F(t))),\, \ba,\bomega\in W^{1,p}(\varepsilon,T)
\end{align}
for all $\varepsilon>0$ and for all $1\leq p<\infty$.
\end{teo}

\subsection{Literature review}

The global regularity of the incompressible Navier-Stokes equations in dimension $2$ is a well-known result which was first proved by  by Leray \cite{leray1934essai} and Ladyzhenskaya \cite{ladyzhenskaia1959solution}, but in dimension $3$ it is a famous open problem. However, there are regularity results for weak solutions that additionally satisfy Prodi-Serrin condition proved by Leray \cite{leray1934mouvement} 
for $\Omega=\R^3$, including the case $s=\infty$, and by Fabes, Lewis and Riviere \cite{fabes1977singular, fabes1977boundary}, and Sohr \cite{sohr1983regularitatstheorie} for domains with a bounded boundary. These regularity results were extended to the case $s=3$ by von Wahl \cite{vonWhal1986} on the bounded domain, and by Giga \cite{giga1986solutions} on the domain with a bounded boundary.
There are also plenty of conditional regularity results with other types of conditions, e.g. on gradient of the fluid velocity or the pressure (see e.g. Remarks 5.6 and 5.9 in \cite{GaldiNS00}).{\footnote{Let us mention also the conditional regularity of the type that one component of the velocity field is more regular, see \cite{NeP}.
}}

In the case of the fluid - rigid body system theory is much less developed. $2D$ case is studied in \cite{bravin2018weak,bravin20202d,GlassSueur} where existence and uniqueness of global weak solution is proved provided that rigid body does not touch the boundary. Moreover, they show that these solutions are strong away from $t=0$. In the three dimensional case there are results of local-in-time strong solutions or global-in time solutions for small initial data \cite{cumsille2008well, DE2, GGH13, MaityTucsnak18, Takahashi03, dintelmann2009strong}. Moreover, global-in-time existence (or existence up to the time of contact) of Leray - Hopf type weak solution were studied in \cite{ConcaRigid00, DE, Feireisl03, GunzRigid00, feireisl2002motion}. We also mention the existence results in the case of slip boundary conditions \cite{gerard2014existence, ChemetovSarka, bravin2018weak, bravin20202d, wang2014strong}. 
Uniqueness of weak solutions is still an open problem, but results of weak-strong uniqueness type were proved in both slip and no-slip case (\cite{chemetov2017weak, BorisSarkaAna2020weak}) which state that the weak solution satisfying additional condition on the fluid velocity is unique in the larger class of weak solutions.

Our regularity result stated in Theorem \ref{mainResult} is motivated by the classical regularity result for the incompressible  Navier-Stokes case \cite[Theorem 5.2]{GaldiNS00}). To the best of our knowledge, this is the first result on the regularity of weak solution for the fluid-rigid body interaction problem. Inspired by works \cite{GGH13,Takahashi03,TakahashiTucsnak04} we use fixed point theorem in combination with the maximal regularity result for the Stokes problem. The proof strategy in more details is outline in the next Section.

\section{Proof strategy}

Here we follow the classical approach to linearize problem \eqref{FSINoslip} around a solution that satisfies the Prodi-Serrin condition, and to analyse the regularity properties of the solution to the linear system. Since the linearized problem has a unique solution and a solution to the nonlinear problem is the solution to the linearized problem, proving the regularity for the linearized problem is enough. This approach has also been used to prove a conditional regularity result for the incompressible Navier-Stokes equations (see e.g. \cite[Section 5]{GaldiNS00}). However, adapting this strategy to the fluid - rigid body system is very technically challenging as outlined below. The main steps of the proof are:
\begin{itemize}
    \item \textbf{Linearization.} Let $(\widetilde{\bu},\widetilde{p},\widetilde{\bq},\widetilde{\bomega})$ be a weak solution to the problem \eqref{FSINoslip} which satisfies Prodi-Serrin condition. We linearize around that solution and obtain a linear fluid - rigid body problem where the movement of the rigid body is prescribed by $\widetilde{\bq}$ and $\widetilde{\bomega}$. We show that the obtained linear problem has a unique weak solution which we denote by $({\bu},{p},{\ba},{\bomega})$.
    \item \textbf{Transformation to the fixed reference domain.} Since the linear problem is posed in the moving domain it is convenient to transform it a cylindrical domain, i.e. to the reference domain that does not depend on time. We use a change of variable that preserves the divergence free condition and is rigid near the rigid body. This change of variables was introduced in \cite{inoue1977existence} and by now is a standard tool in analysis of the fluid - rigid body system. Solution to the linearized problem on the cylindrical domain is denoted by $(\bU,P,\bA,\bOmega)$.
    \item \textbf{Strong solution.} We show that if $\widetilde{\bu}$ satisfied the Prodi-Serrin condition then $(\bU,P,\bA,\bOmega)$ is the strong solution, i.e. equations \eqref{FSINoslip} are satisfied in the $L^p$ sense. The main technical tool is the fixed point theorem and a maximal regularity result for the fluid - rigid body operator. This finishes the proof of Theorem \ref{mainResult2}.
    \item \textbf{Higher regularity.} In this step the goal is to bootstrap the argument from the previous step to get the higher regularity estimates. Therefore, first we need to prove regularity of the time derivatives. This is achieved by analysing the system obtained by formally differentiating in time the linearized system from the previous steps. The main issue here is to prove that the solution to the system obtained by taking the time derivative is exactly the time derivative of the solution. This is a nontrivial step because we do not have any a priori estimates for time derivatives and our solution is obtained by the fixed point procedure and thus it is not possible to directly justify the formal estimates. In this step we need an additional regularity condition on the rigid body acceleration $\frac{\d^2}{\dt^2}\widetilde{\bq},\frac{\d}{\dt}\widetilde{\bomega}\in L^{\infty}(0,T)$.
\end{itemize}

The paper is organized as follows. In Section \ref{Section:Linear}
we introduce the linearized problem and show that it admits a unique weak solution which is equal to the solution of \eqref{FSINoslip}. The second step, transformation to the fixed reference domain is done in Section \ref{Section:Transformed}. There we also state Proposition \ref{strong} (which corresponds to third step of the proof), and Propositions  \ref{time_derivatives} and \ref{spatial_derivatives} which correspond to the last step of the proof. The proofs of these Propositions will be presented in Sections \ref{Section:StrongSolution}, \ref{Section:time_derivatives} and \ref{Section:spatial_derivatives}, respectively. The technical core of the paper are Sections \ref{Section:StrongSolution} and \ref{Section:time_derivatives}. The additional condition on the boundedness of the rigid body acceleration is needed in Section \ref{Section:time_derivatives}. Here we want to point out that even though formal estimates do not require this condition, this condition is needed in rigorous justification of the estimates. Namely, the standard methods for construction of the solutions, such as Galerkin method or the regularization method, do not seem to work in this context because of the presence of the moving boundary. Finally, few technical proofs are relegated to the Appendix. Since the proof involves a lot of notation, for the convenience of the reader we have summarized all notations used in the paper in table at the end of Appendix.

\subsection{Linear problem}
\label{Section:Linear}
Let $(\widetilde{\bu},\widetilde{p},\widetilde{\bq},\widetilde{\bomega})$ be a weak solution to the problem \eqref{FSINoslip} in a sense of Definition \ref{definition} which additionally satisfies the Prodi-Serrin condition. Let the rigid body domain $S(t)$ be defined by $(\widetilde{\bq},\widetilde{\bomega})$ through formulas \eqref{is} and \eqref{set}  as well as the fluid domain $\Omega_{F}(t) = \Omega\setminus\overline{S(t)}$. We define the following  linear problem:
\\
Find $({\bu},{p},{\ba},{\bomega})$ such that
\begin{equation}	\label{Linear}
	\begin{array}{l}
	\left.
	\begin{array}{l}
	\partial_{t}{\bu}+({\widetilde{\bu}}\cdot \nabla ){\bu}
	= \divg \left( {\T}({\bu},{p})\right) , \\
	\divg{\bu} = 0
	\end{array}
	\right\} \;\mathrm{in}\;(0,T)\times{\Omega_{F}(t)},
	\\
	\left.
	\begin{array}{l}
	\frac{\d}{\dt}{\ba} = -\int_{\partial{S(t)}}{\T}({\bu},{p}) \bn\,\d\bgamma(\bx), \\
	\frac{\d}{\dt}(\J{\bomega}) = -\int_{\partial {S(t)}}(\bx-{\widetilde{\bq}}(t))\times {\T}({\bu},{p})\bn\,\d\bgamma(\bx)
	\end{array}
	\right\} \;\mathrm{in}\;(0,T),
	\\
	{\bu} ={\ba}+{\bomega}\times (\bx-{\widetilde{\bq}}),\quad \mathrm{on}\;(0,T)\times\partial {S(t)},
	\\
	{\bu} = 0 \quad \mathrm{on}\, \partial \Omega,
	\\
	{\bu}(0,.)=\bu_{0}\qquad \mathrm{in}\;\Omega,\quad {\ba}(0)=\ba_{0},\quad
	{\bomega}(0)=\bomega_{0}.
	\end{array}
\end{equation}
Note that the problem is linear since the motion of the domain is a priori given and is not computed via $(\ba,\bomega)$.
\begin{definition}\label{LinearWeak}
A weak solution for \eqref{Linear} is a function ${\bu}\in  L^{2}(0,T;V(t))\cap L^{\infty
}(0,T;L^2(\Omega))$ which satisfies
\begin{equation}	\label{weak2}
	\begin{split}
	\int_{0}^{t}\int_{\Omega}
	{\bu} \cdot \partial_{t}\bphi
	\, \d\bx\d \tau
	&+ \int_{0}^{t}\int_{\Omega_F(t)}\{
	({\widetilde{\bu}}\otimes {\bu}) :\nabla\bphi^T
	- 2\,\D{\bu}:\D\bphi\,\}\, \d\bx\d \tau\\
	&- \int_{\Omega }{\bu}(t)\cdot\bphi(t)\,\d\bx
	= - \int_{\Omega }\bu_{0}\cdot\bphi(0)\,\d\bx
	\,- \int_{0}^{t}\widetilde{\ba}\times{\ba}\cdot\bphi_{\bomega},
	\end{split}
\end{equation}
for all $\bphi\in H^1(0,T;V(t))$, where 
$$
\bphi(t,\bx) = \bphi_{a}(t) + \bphi_{\bomega}(t)\times(\bx-\widetilde{\bq}(t)),
\quad  \bx\in S(t).
$$
\end{definition}

First, we show the uniqueness result for the linearized problem \eqref{Linear}. 
\begin{lemma}	\label{uniqueness_1}
	Let $\widetilde{\bu}$ be a weak solution to the \eqref{FSINoslip} which additionally satisfies the Prodi-Serrin condition,
	and let ${\bu}$ be a weak solution for \eqref{Linear}. Then $\bu=\widetilde{\bu}$ almost everywhere in $(0,T)\times\Omega$.
\end{lemma}	
\begin{proof}
Let us denote $\bar{\bu}=\widetilde{\bu}-\bu$, $\bar{\ba}=\widetilde{\ba}-\ba$, $\bar{\bomega}=\widetilde{\bomega}-\bomega$.
By subtracting the equations
\begin{equation*}
\begin{split}
\int_{0}^{t}\int_{\Omega}
\widetilde{\bu} \cdot \partial_{t}\bpsi
\, \d\bx\d \tau
+ \int_{0}^{t}\int_{\Omega_F(t)}\{
(\widetilde{\bu}\otimes \widetilde{\bu})& :\nabla\bpsi^T
- 2\,\D\widetilde{\bu}:\D\bpsi\,\}\, \d\bx\d \tau\\
&- \int_{\Omega }\widetilde{\bu}(t)\bpsi(t)\,\d\bx
= - \int_{\Omega }\bu_{0}\bpsi(0)\,\d\bx,
\end{split}
\end{equation*}
\begin{equation*}
\begin{split}
\int_{0}^{t}\int_{\Omega}
{\bu} \cdot \partial_{t}\bpsi
\, \d\bx\d \tau
&+ \int_{0}^{t}\int_{\Omega_F(t)}\{
(\widetilde{\bu}\otimes {\bu}) :\nabla\bpsi^T
- 2\,\D{\bu}:\D\bpsi\,\}\, \d\bx\d \tau\\
&- \int_{\Omega }{\bu}(t)\bpsi(t)\,\d\bx
= - \int_{\Omega }\bu_{0}\bpsi(0)\,\d\bx
+ \int_{0}^{t}\widetilde{\ba}\times{\ba}\cdot\bpsi_{\bomega}\,\d\tau,
\end{split}
\end{equation*}
we get
\begin{equation*}
\begin{split}
\int_{0}^{t}\int_{\Omega}
\bar{\bu} \cdot \partial_{t}\bphi
\, \d\bx\d \tau
&+ \int_{0}^{t}\int_{\Omega_F(t)}\{
(\widetilde{\bu}\otimes\bar{\bu}) :\nabla\bphi^T
- 2\,\D\bar{\bu}:\D\bphi\,\}\, \d\bx\d \tau
- \int_{\Omega }\bar{\bu}(t)\bphi(t)\,\d\bx \\
& = \int_{0}^{t}
\underbrace{\widetilde{\ba}\times\bar{\ba}\cdot\bphi_{\bomega}}
_{=\int_{S(\tau)}\widetilde{\ba}\times\bar{\bu}\cdot\bphi_{\bomega}}\d\tau.
\end{split}
\end{equation*}
We substitute $\bphi = \bar{\bu}^h$ and let $ h \to 0 $. Here $\bu^{h}$ is a regularization of $\bu$ as defined in \cite[Section 2.3]{BorisSarkaAna2020weak}. We use \cite[Lemma 2.4]{BorisSarkaAna2020weak} to pass to the limit and get:
\begin{equation*}
\frac{1}{2}\|\bar{\bu}(t)\|_{L^2(\Omega)}^2 
+ 2\int_{0}^{t}\int_{\Omega_F(\tau)}|\D\bar{\bu}|^2\,\d\bx\d \tau
+ \int_{0}^{t}\int_{\Omega_F(\tau)} \widetilde{\bu}\cdot\nabla\bar{\bu}\cdot\bar{\bu}\,\d\bx\d \tau
= \int_{0}^{t}\int_{S(\tau)}
(\widetilde{\bu}-\widetilde{\ba})\times{\bar{\bu}}\cdot\bar{\bomega}\,\d\bx\d \tau .
\end{equation*}
{To estimate the third term we use the Prodi-Serrin condition as in \cite{BorisSarkaAna2020weak}, and in the end we get the inequality of the form}
\begin{equation*}
\|\bar{\bu}(t)\|_{L^2(\Omega)}^2 
\leq \int_{0}^{t} f(\tau)\|\bar{\bu}(\tau)\|_{L^2(\Omega)}^2\,\d\tau,
\end{equation*}
where $f\in L^1(0,t)$,
so Gronwall lemma implies that $\bar{\bu}=0$.
\end{proof}

Therefore, by Lemma \ref{uniqueness_1} it follows that $\widetilde{\bu}$ is a unique weak solution to the linearized problem \eqref{Linear}. Therefore, in the rest of the paper we will consider linear problem \eqref{Linear} and prove that the solution to the linear problem is regular.

\subsection{The transformed problem}
\label{Section:Transformed}

In order to transform the problem to the cylindrical domain we use a change of variables inspired by Inoue and Wakimoto \cite{inoue1977existence}, i.e. we define the mapping $(\bu,p,\ba,\bomega)\mapsto(\bU,P,\bA,\bOmega)$
with
\begin{equation} \label{ChangeOfVariables}
\left\{
\begin{array}{l}
{\bU}(t,\by)=\nabla\bY(t,\bX(t,\by)){\bu}(t,\bX(t,\by)),
\\
{P}(t,\by)={p}(t,\bX(t,\by)),
\\
{\bA}(t)=\Q^{T}(t){\ba}(t),
\\
{\bOmega}(t)=\Q^{T}(t){\bomega}(t),
\end{array}
\right.
\end{equation}
where $\bX(t)$ is a volume-preserving diffeomorphism from initial to the physical domain described in \cite{BorisSarkaAna2020weak}, Appendix A.1, and $\bY(t)$ is its inverse. By construction, $\bX$, $\bY$ belong to $W^{1,\infty}(0,T;C_c^{\infty}(\Omega))$ and depend on the domain of given solution, i.e. translation velocity $\widetilde{\ba}$ and angular velocity $\widetilde{\bomega}$.
In the following, $(\bU,P,\bA,\bOmega)$ and $(\widetilde{\bU},\widetilde{P},\widetilde{\bA},\widetilde{\bOmega})$ denote  the transformations of solutions $(\bu,p,\ba,\bomega)$ and $(\widetilde{\bu},\widetilde{p},\widetilde{\ba},\widetilde{\bomega})$ to the  cylindrical domain by mapping \eqref{ChangeOfVariables}. 
Notice that lowercase letters refer to the solutions defined on the physical moving domain and uppercase letters to the solutions defined on the fixed reference domain.
Therefore $(\widetilde{\bU},\widetilde{P},\widetilde{\bA},\widetilde{\bOmega})$ is the solution to the following system (which is equivalent to \eqref{Linear}):
\begin{equation}	\label{FixedDomain_2}
	\begin{array}{l}
	\left.
	\begin{array}{l}
	\partial _{t}{\bU}-\triangle {\bU}+\nabla {P} = {F}(\bU,P), \\
	\divg{\bU} = 0
	\end{array}
	\right\} \;\mathrm{in}\;(0,T)\times\Omega_{F},
	\\
	\left.
	\begin{array}{l}
	\frac{\d}{\dt}{\bA}
	=-\int_{\partial S_{0}}\mathcal{T}({\bU},{P})\bN\,\d\gamma(\by)
	+ G(\bA)
	\\
	\frac{\d}{\dt}(\mathcal{J}{\bOmega})
	=
	-\int_{\partial S_{0}}
	\by\times {\mathcal{T}}({\bU},{P})\bN \,\d\gamma(\by)
	+ H(\bOmega)
	\end{array}
	\right\} \;\mathrm{in}\;(0,T),
	\\
	{\bU} = {\bOmega}\times\by + {\bA}\qquad \mathrm{on}\;(0,T)\times\partial S_0,
	\\
	{\bU} = 0\qquad \mathrm{on}\;(0,T)\times\partial\Omega,
	\\
	{\bU}(0,.)=\bu_0\qquad \mathrm{in}\;\Omega,\quad {\bA}(0)=\ba_0,\quad
	{\bOmega}(0)=\bomega_0
	\end{array}
\end{equation}
where
\begin{equation}	\label{rightF}
F(\bU,P)
:= (\mathcal{L}-\triangle ){\bU}
-\mathcal{M}{\bU}
-{\mathcal{N}}{\bU}
-(\mathcal{G}-\nabla ){P},
\end{equation}
\begin{equation}	\label{rihgtGH}
G(\bA) = -\widetilde{\bOmega}\times {\bA},\quad
H(\bOmega) = -\widetilde{\bOmega}\times (\mathcal{J}{\bOmega}),
\end{equation}
\begin{equation*}
\mathcal{T}({\bU},{P}) = -{P}\,\I + 2\D\bU,
\quad
\mathcal{J}(t)=\Q^T(t)\J(t)\Q(t).
\end{equation*}
The operator $\mathcal{L}$ is the transformed Laplace operator and it is given by
\begin{align}  \label{OpL}
(\mathcal{L}\bU)_{i}& =\sum_{j,k=1}^{n}\partial
_{j}(g^{jk}\partial_k \bU_{i})+2\sum_{j,k,l=1}^{n}g^{kl}\Gamma
_{jk}^{i}\partial _{l}\bU_{j}  
+\sum_{j,k,l=1}^{n}\Big(
\partial _{k}(g^{kl}\Gamma
_{jl}^{i})+\sum_{m=1}^{n}g^{kl}\Gamma _{jl}^{m}\Gamma _{km}^{i}
\Big)
\bU_{j},
\end{align}
the convection term is transformed into
\begin{equation} \label{OpN}
({\mathcal{N}}({\bU}))_i
= \left(({\widetilde{\bU}}\cdot \nabla ){\bU}\right)_i 
+ \sum_{j,k=1}^n \Gamma ^i_{jk}{\widetilde{\bU}_j}{\bU}_k,
\end{equation}
the transformation of time derivative and gradient is given by
\begin{equation}	\label{OpM}
(\mathcal{M} \bU)_{i} = \sum _{j=1}^n \dot{\mathbf{Y}}_j \partial _j
\bU_i + \sum _{j,k=1}^n \Big(\Gamma _{jk}^i \dot{\mathbf{Y}}_k +
(\partial _k \mathbf{Y}_i)(\partial _j \dot{\mathbf{X}}_k)\Big)\bU_j,
\end{equation}
and the gradient of pressure is transformed as follows:
\begin{equation}  \label{OpP}
(\mathcal{G}P)_{i}=\sum_{j=1}^{n}g^{ij}\partial _{j}P
=
(\nabla\bY\nabla\bY^T\nabla P)_{i}.
\end{equation}
Here we have denoted the metric covariant tensor
\begin{equation}	\label{covariant}
g_{ij}=X_{k,i}X_{k,j},\qquad
X_{k,i}=\frac{\partial \bX_{k}}{\partial \by_{i}},
\end{equation}
the metric covariant tensor
\begin{equation}
g^{ij}=Y_{i,k}Y_{j,k}\qquad Y_{i,k}=\frac{\partial \bY_{i}}{\partial \bx_{k}},
\end{equation}
and the Christoffel symbol (of the second kind)
\begin{equation}	\label{Christoffel}
\Gamma _{ij}^{k}=\frac{1}{2}g^{kl}(g_{il,j}+g_{jl,i}-g_{ij,l}),\qquad
g_{il,j}=\frac{\partial {g_{il}}}{\partial \by_{j}}.
\end{equation}	
Note that operators $\mathcal{L},\mathcal{M},\mathcal{N}$ and $\mathcal{G}$ are linear and depend on transformation $\bX$, i.e. on functions $\widetilde{\bA}$ and $\widetilde{\bOmega}$.

The first step of the proof is to show that the linear problem \eqref{FixedDomain_2} admits a unique strong solution, which  by Lemma \ref{uniqueness_1} means that the solution to the original nonlinear problem is also a strong solution. Therefore Theorem \ref{mainResult2} follows from the following result:

\begin{proposition}\label{strong}
	Let $(\widetilde{\bU},\widetilde{\bA},\widetilde{\bOmega})\in  (L^{2}(0,T;V(0))\cap L^{\infty
	}(0,T;L^2(\Omega)))\times L^{\infty}(0,T)\times L^{\infty}(0,T)$ be a weak solution to the problem \eqref{FSINoslip} that satisfies the Prodi-Serrin condition. Then there exists a unique solution $(\bU,P,\bA,\bOmega)$ of \eqref{FixedDomain_2} in the sense of Definition \ref{LinearWeak} satisfying the following regularity properties
	\begin{align*}
&\bU \in L^{p}(\varepsilon,T;W^{2,p}(\Omega_F))\cap
W^{1,p}(\varepsilon,T;L^{p}(\Omega_F)),
\\
&P \in L^{p}(\varepsilon,T;\quotient{W^{1,p}(\Omega_F)}{\R}),
\\
&\bA,\bOmega\in
W^{1,p}(\varepsilon,T),
\end{align*}
for all $\varepsilon>0$ and for all $1\leq p<\infty$.
Moreover, by Lemma \ref{uniqueness_1},
$(\widetilde{\bU}, \widetilde{P},\widetilde{\bA},\widetilde{\bOmega})=({\bU}, P, {\bA},{\bOmega})$ and thus satisfies the same regularity properties. In particular, $(\widetilde{\bU},\widetilde{P},\widetilde{\bA},\widetilde{\bOmega})$ is a strong solution to problem \eqref{FSINoslip}.
\end{proposition} 

We relegate the proof to Section \ref{Section:StrongSolution}. Next, we state two Propositions which provide higher regularity of the solution and thus finish the proof of Theorem \ref{mainResult}. The proofs of these Propositions are given in Sections \ref{Section:time_derivatives} and \ref{Section:spatial_derivatives}.

\begin{proposition} \label{time_derivatives}
	Let $\widetilde{\bU}$ be a weak solution satisfying the assumption of Lemma \ref{strong} and 
{$\frac{\d}{\dt}\widetilde{\bA}, \frac{\d}{\dt}\widetilde{\bOmega}\in L^{\infty}(0,T)$}. Then
	\begin{align*}
&\partial_t^l\bU, \partial_t^l\widetilde{\bU}\in L^{p}(\varepsilon,T;W^{2,p}(\Omega_F))\cap
W^{1,p}(\varepsilon,T;L^{p}(\Omega_F)),
\\
&\partial_t^l P, \partial_t^l\widetilde{P}\in L^{p}(\varepsilon,T;\quotient{W^{1,p}(\Omega_F)}{\R}),
\\
&\frac{\d^{l}}{\dt^{l}}\bA,\frac{\d^{l}}{\dt^{l}}\bOmega, 
\frac{\d^{l}}{\dt^{l}}\widetilde{\bA}, \frac{\d^{l}}{\dt^{l}}\widetilde{\bOmega}\in
W^{1,p}(\varepsilon,T),
\end{align*}
for all $\varepsilon>0, l\geq 0$ and for all $1\leq p<\infty$.
\end{proposition}

\begin{proposition}\label{spatial_derivatives}
	Let $\widetilde{\bU}$ be a weak solution satisfying the assumption of Proposition \ref{time_derivatives}. Then
	$$
	\partial_t^l\widetilde{\bU}\in L^2(\varepsilon,T;H^{k}(\Omega_F)),\, \partial_t^l\widetilde{P}\in L^2(\varepsilon,T;\quotient{H^{k-1}(\Omega_F)}{\R}),\quad \forall l\geq 0, k\geq 2.
	$$
\end{proposition}

\section{Strong Solution}
\label{Section:StrongSolution}
This Section deals with   the proof of Proposition \ref{strong}.
Since $\bX,\bY\in W^{1,\infty}(0,T;C_c^{\infty}(\Omega))$,
the transformation \eqref{ChangeOfVariables} preserves integrability of functions, so we have
$$
(\widetilde{\bU},\widetilde{\bA},\widetilde{\bOmega})\in  (L^{2}(0,T;V(0))\cap L^{\infty
}(0,T;L^2(\Omega)))\times L^{\infty}(0,T)\times L^{\infty}(0,T)
$$
and $\widetilde{\bU}$ satisfies Prodi-Serrin condition. Since we are interested in the regularity excluding $t=0$ we multiply \eqref{FixedDomain_2} by $t$ and define
\begin{equation*}
(\bU^{*}, P^{*},\bA^{*},\bOmega^{*})= (t\bU, tP,t\bA,t\bOmega),
\end{equation*}
which satisfy the following problem on a cylindrical domain with vanishing initial conditions
\begin{equation}	\label{FixedDomain_stokes}
	\begin{array}{l}
	\left.
	\begin{array}{l}
	\partial_{t}{\bU}^{*}-\triangle {\bU}^{*}+\nabla {P}^{*} = F^{*}, \\
	\divg{\bU}^{*} = 0
	\end{array}
	\right\} \;\mathrm{in}\;(0,T)\times\Omega_{F},
	\\
	\left.
	\begin{array}{l}
	\frac{\d}{\dt}{\bA}^{*}
	=-\int_{\partial S_{0}}\mathcal{T}({\bU}^{*},{P}^{*})\bN\,\d\gamma(\by)
	+ G^{*}
	\\
	\frac{\d}{\dt}(\mathcal{J}{\bOmega}^{*})
	=
	-\int_{\partial S_{0}}
	\by\times {\mathcal{T}}({\bU}^{*},{P}^{*})\bN \,\d\gamma(\by)
	+ H^{*}
	\end{array}
	\right\} \;\mathrm{in}\;(0,T),
	\\
	{\bU}^{*} = {\bA}^{*} + {\bOmega}^{*}\times\by
	\qquad \mathrm{on}\;(0,T)\times\partial S_0,
	\\
	{\bU}^{*} = 0\qquad \mathrm{on}\;(0,T)\times\partial\Omega,
	\\
	{\bU}^{*}(0,.)=0\qquad \mathrm{in}\;\Omega,\quad {\bA}^{*}(0)=0,\quad
	{\bOmega}^{*}(0)=0,
	\end{array}
\end{equation}
where
\begin{equation} \label{right1}
F^{*} =  F(\bU^{*},P^{*}) +\bU,\quad
G^{*} = {G}(\bA^{*}) + {\bA},\quad
H^{*} = {H}(\bOmega^{*}) + \mathcal{J}{\bOmega}.
\end{equation}

It is sufficient to show that there is a unique strong solution to the problem
\eqref{FixedDomain_stokes}-\eqref{right1}, with
\begin{equation}	\label{integrability}
(\bU,\bA,\bOmega)\in (L^{2}(0,T;V(0))\cap L^{\infty
}(0,T;L^2(\Omega_F)))
\times L^{2}(0,T)\times L^{2}(0,T)
\end{equation}
are given functions.
Then the problem \eqref{FixedDomain_2} also has a unique strong solution on the interval $(\varepsilon,T)$, for all  $\varepsilon>0$. Finally, by change of variables, i.e. returning to the physical domain, follows that $(\widetilde{\bu},\widetilde{\ba},\widetilde{\bomega})$ is the strong solution of \eqref{FSINoslip}.

First we prove the following result:
\begin{proposition}	\label{strong2}
	Let 
	\begin{equation*} \label{right_side}
	F^{*} =  F(\bU^{*},P^{*}) + \mathcal{R},
	\quad
	G^{*} = {G}(\bA^{*}) + \mathcal{R}_{\ba},
	\quad
	H^{*} = H(\bOmega^{*}) + \mathcal{R}_{\bomega},
	\end{equation*}
	where operators $F$, $G$ and $H$ are defined by \eqref{rightF}-\eqref{rihgtGH}.
	\begin{itemize}
	\item[a)] Let $$
	(\widetilde{\bU},\widetilde{\bA},\widetilde{\bOmega})\in  (L^{2}(0,T;V(0))\cap L^{\infty
	}(0,T;L^2(\Omega)))\times L^{\infty}(0,T)\times L^{\infty}(0,T)
	$$ 
	satisfy Prodi-Serrin condition \eqref{LrLs} and
	\begin{equation}	\label{R1}
	\mathcal{R}\in L^2(0,T;L^2(\Omega_F)),\quad
	\mathcal{R}_{\ba},\mathcal{R}_{\bomega}\in{L^2(0,T)}.
	\end{equation}
	Then there is a unique solution $(\bU^{*},P^{*},\bA^{*},\bOmega^{*})$ for \eqref{FixedDomain_stokes} satisfying
	\begin{equation*}
	\begin{split}
	&\bU^{*}\in H^1(0,T;L^2(\Omega_F))\cap L^2(0,T;H^{2}(\Omega_F)),
	\\
	&P^{*}\in L^2(0,T;\quotient{H^{1}(\Omega_F)}{\R}),
	\\
	&\bA^{*},\bOmega^{*}\in H^{1}(0,T).
	\end{split}
	\end{equation*}
	\item[b)] If in addition we assume  
	$$
	(\widetilde{\bU},\widetilde{\bA},\widetilde{\bOmega})\in  (L^{2m}(0,T;W^{2,2m}(\Omega_F))\cap W^{1,2m}(0,T;L^{2m}(\Omega_F)))\times W^{1,2m}(0,T)\times W^{1,2m}(0,T)
	$$ 
	and
	\begin{equation}	\label{R2}
	\mathcal{R}\in L^{2m}(0,T;L^{2m}(\Omega_F)),\quad
	\mathcal{R}_{\ba},\mathcal{R}_{\bomega}\in{L^{2m}(0,T)},
	\end{equation}
	for $m\in\N$.
	Then
	\begin{equation*}
	\begin{split}
	&\bU^{*}\in W^{1,2m}(0,T;L^{2m}(\Omega_F))\cap L^{2m}(0,T;W^{2,2m}(\Omega_F)),
	\\
	&P^{*}\in L^{2m}(0,T;\quotient{W^{1,2m}(\Omega_F)}{\R}),
	\\
	&\bA^{*},\bOmega^{*}\in W^{1,2m}(0,T).
	\end{split}
	\end{equation*}
	\end{itemize}
\end{proposition}
	
Especially, for
\begin{equation*}
\mathcal{R} = \bU,\quad
\mathcal{R}_{\ba} = {\bA},\quad
\mathcal{R}_{\ba} = \mathcal{J}{\bOmega},
\end{equation*}
we obtain the right hand side \eqref{right1} which satisfies \eqref{R1}.

We will prove Proposition \ref{strong2} by using the fixed point theorem and the following maximal regularity result 
\begin{teo}[\cite{GGH13}, Theorem 4.1]	\label{regularity_fixedPoint}
	Let $\bOmega$ be a domain with boundary of class $C^{2,1}$ and $p,q\in (1,\infty)$. Let $F^{*}\in L^{p}(0,T;L^{q}(\Omega_F))$, and $G^{*}, H^{*}\in L^{p}(0,T)$. Then for every $T > 0$, problem \eqref{FixedDomain_stokes} admits a unique solution
	\begin{align*}
	&{\bU}^{*}\in X_{p,q}^T :=W^{1,p}(0,T;L^q(\Omega_{F}))\cap L^p(0,T;W^{2,q}(\Omega_F)),\\
	&{P}^{*}\in Y_{p,q}^T := L^p(0,T;\quotient{{W}^{1,q}(\Omega_{F})}{\R}),\\
	& {\bA}^{*},{\bOmega}^{*}\in W^{1,p}(0,T),
	\end{align*}
	which satisfies the estimate
	\begin{equation*}
	\|{\bU}^{*}\|_{X_{p, q}^{T}}
	+ \|{P}^{*}\|_{Y_{p,q}^{T}}
	+ \|{\bA}^{*}\|_{W^{1,p}\left(0,T\right)}
	+\|{\bOmega}^{*}\|_{W^{1,p}\left(0,T\right)}
	\leq C\left(\|F^{*}\|_{L^p(0,T;L^q(\Omega_F))}
	+\|G^{*}\|_{L^p(0,T)}
	+\|H^{*}\|_{L^p(0,T)}\right),
	\end{equation*}
	where the constant $C$ depends only on the geometry of the rigid body and on $T$.
\end{teo}
\begin{remark}
	From the proof of the above Theorem it can be seen that the constant $C$ is non-decreasing with respect to $T$.
\end{remark}
For fixed $R>0$, which we will choose later, we define a set
\begin{multline*}
\mathcal{K}_{R} := \big\lbrace\, (\widehat{\bU},\widehat{P},\widehat{\bA},\widehat{\bOmega})\in X_{p,q}^T\times Y_{p,q}^T\times W^{1,p}(0,T)\times W^{1,p}(0,T)
\,:\, \\
\|\widehat{\bU}\|_{X_{p, q}^{T}}
+ \|\widehat{P}\|_{Y_{p,q}^{T}}
+ \|\widehat{\bA}\|_{W^{1,p}\left(0,T\right)}
+\|\widehat{\bOmega}\|_{W^{1,p}\left(0,T\right)}
\leq R
\,\big\rbrace,
\end{multline*}
and a function
$$
\mathcal{S}:\, (\widehat{\bU}, \widehat{P}, \widehat{\bA}, \widehat{\bOmega})\mapsto ({\bU}^{*}, {P}^{*}, {\bA}^{*}, {\bOmega}^{*})
$$
where $({\bU}^{*}, {P}^{*}, {\bA}^{*}, {\bOmega}^{*})$ is solution to the problem \eqref{FixedDomain_stokes} with the right hand side which depends on $(\widehat{\bU}, \widehat{P}, \widehat{\bA}, \widehat{\bOmega})$, i.e. we consider problem \eqref{FixedDomain_stokes} with
\begin{equation}	\label{rightFstar}
F^{*} 
= F(\widehat{\bU},\widehat{P}) + \mathcal{R}
= (\mathcal{L}-\triangle )\widehat{\bU}-\mathcal{M}\widehat{\bU}
-{\mathcal{N}}\widehat{\bU}
-(\mathcal{G}-\nabla )\widehat{P}
+ \mathcal{R}
,
\end{equation}
\begin{equation}	\label{rightGHstar}
	G^{*}
	= G(\widehat{\bA}) + \mathcal{R}_{\ba}
	= -\widetilde{\bOmega}\times \widehat{\bA} + \mathcal{R}_{\ba},\quad
	H^{*} 
	= H(\widehat{\bOmega}) + \mathcal{R}_{\bomega}
	= -\widetilde{\bOmega}\times (\mathcal{J}\widehat{\bOmega}) + \mathcal{R}_{\bomega}.
\end{equation}
We will prove that $S$ is a contraction and will use the Banach's fixed point theorem. More precisely, we will show that: 
\begin{itemize}
	\item $\mathcal{S}$ is well defined on $\mathcal{K}_R$ and $\mathcal{S}(\mathcal{K}_R)\subset \mathcal{K}_R$,
	\item $\mathcal{S}$ is a contraction,
\end{itemize}
which yields a unique fixed point of $\mathcal{S}$, i.e. a unique solution $({\bU}^{*}, {P}^{*}, {\bA}^{*}, {\bOmega}^{*})\in\mathcal{K}_R$ to problem \eqref{FixedDomain_stokes}-\eqref{right1}.

\subsection{Estimates on the right hand side}
In order to use Theorem \ref{regularity_fixedPoint} we have to estimate the right hand side \eqref{rightFstar}-\eqref{rightGHstar}.

First, we state an auxiliary Lemma which directly follows 
form the basic properties of the transformations $\bX$ and $\bY$. Since these estimates follow by direct calculations in the standard way (see e.g. \cite[Section 6.2]{Takahashi03}) we omit the proof.
\begin{lemma}	\label{estimates_1}
	If $\widetilde{\bA},\widetilde{\bOmega}\in W^{l,p}(0,T)$, for $l\geq 0$ and $1< p\leq\infty$, then
	\begin{equation*}
	\Vert \bX\Vert _{W^{l,\infty}(0,T;W^{3,\infty }(\Omega))}
	+\Vert
	\bY\Vert_{W^{l,\infty}(0,T;W^{3,\infty }(\Omega))}
	\leq C,
	\end{equation*}
	\begin{equation*}
	\Vert \bX\Vert _{W^{l+1,p}(0,T;W^{3,\infty }(\Omega))}
	+\Vert\bY\Vert_{W^{l+1,p}(0,T;W^{3,\infty }(\Omega))}
	\leq C,
	\end{equation*}
	\begin{equation*}
	\Vert g_{ij}\Vert_{W^{l,\infty}(0,T;W^{1,\infty }(\Omega))} 
	+\Vert
	g^{ij}\Vert _{W^{l,\infty}(0,T;W^{1,\infty }(\Omega))} 
	+ \Vert\Gamma
	_{ij}^{k}\Vert_{W^{l,\infty}(0,T;L^{\infty }(\Omega))}
	\leq C,
	\end{equation*}
	\begin{equation*}
	\|{g_{j k}}-\delta_{j k}\|_{L^{\infty }(0,T;L^{\infty }(\Omega)}+\|{g^{j k}} -\delta_{j k}\|_{L^{\infty }(0,T;L^{\infty }(\Omega)}
	\leq C\left(T^{\frac{1}{p'}}+T^{\frac{2}{p'}}\right),
	\end{equation*}	
	for all $t\in[0,T]$,	
	where 
	$
	\frac{1}{p}+\frac{1}{p'}=1
	$
	and constant $C$ depends on
	$
	K_p = \|\widetilde{\bA}\|_{W^{l,p}(0,T)}+\|\widetilde{\bOmega}\|_{W^{l,p}(0,T)}
	$
	nondecreasingly.
	
\end{lemma}

For the convective term we have the following result.
\begin{lemma}	\label{estimates_2}
Assume that 
$
\widehat{\bU}\in X_{2,2}^T
$
and 
$
\widetilde{\bU}\in 
L^{r}(0,T;L^{s}(\Omega_F)),
$
such that $\frac{3}{s}+\frac{2}{r}=1$, $s\in(3,\infty)$.
Then 
\begin{align*}
\|({\widetilde{\bU}}\cdot \nabla )\widehat{\bU}\|_{L^2(0,T;L^2(\Omega_F))}
&
\leq
C\|{\widetilde{\bU}}\|_{L^r(0,T;L^s(\Omega_F))}
\|\widehat{\bU}\|_{X_{2,2}^T}.
\end{align*}
\end{lemma} 
\begin{proof}
By using the following embeddings
\begin{multline*}
L^2(0,T;H^2(\Omega_F))\cap H^1(0,T;L^2({\Omega_F}))\hookrightarrow H^{\frac{1}{r}}(0,T;H^{2\left(1-\frac{1}{r}\right)}({\Omega_F}))
\\
= H^{\frac{1}{r}}(0,T;H^{\frac{3}{s}+1}({\Omega_F}))
\hookrightarrow
L^{r'}(0,T;W^{1,s'}({\Omega_F}))
\end{multline*}
for 
$
\frac{1}{r} + \frac{1}{r'} = \frac{1}{2},  \frac{1}{s} + \frac{1}{s'} = \frac{1}{2},
$
we conclude that
$
\nabla\widehat{\bU}\in 
L^{r'}(0,T;L^{s'}(\Omega_F)),
$
and therefore we have
\begin{align*}
\|({\widetilde{\bU}}\cdot \nabla )\widehat{\bU}\|_{L^2(0,T;L^2(\Omega_F))}^2
&= \int_{0}^{T}\int_{\Omega_F}
|({\widetilde{\bU}}\cdot \nabla )\widehat{\bU}|^2\,\d\by\d\tau
\leq
\int_{0}^{T}
\|{\widetilde{\bU}}\|_{L^s(\Omega_F)}^2\| \nabla \widehat{\bU}\|_{L^{s'}(\Omega_F)}^2\,\d\tau
\\
&\leq\left(\int_{0}^{T}
\|{\widetilde{\bU}}\|_{L^s(\Omega_F)}^r\,\d\tau\right)^{\frac{2}{r}}
\left(\int_{0}^{T}\| \nabla \widehat{\bU}\|_{L^{s'}(\Omega_F)}^{r'}\,\d\tau\right)^{\frac{2}{r'}}
\\
&=
\|{\widetilde{\bU}}\|_{L^r(0,T;L^s(\Omega_F))}^{2}
\| \nabla \widehat{\bU}\|_{L^{r'}(0,T;L^{s'}(\Omega_F))}^{2}.
\end{align*}
\end{proof}

Now, we can show the following Lemma to estimate the right-hand side given by \eqref{rightFstar}.
\begin{lemma}	\label{estimates_3}
Assume that
$
\widetilde{\bA},\widetilde{\bOmega}\in L^{p}(0,T)
$, for all $1\leq p<\infty$
and 
$
\widetilde{\bU}\in 
L^{r}(0,T;L^{s}(\Omega_F)),
$
such that $\frac{3}{s}+\frac{2}{r}=1$, $s\in(3,\infty)$.
Then for
$\widehat{\bU}\in X_{2,2}^T
, \widehat{P}\in Y_{2,2}^T
$
the following estimates hold:
\begin{align*}
\|\mathcal{N}\widehat{\bU} \|_{L^2(0,T;L^2(\Omega_F))}
&\leq C
\|{\widetilde{\bU}}\|_{L^r(0,T;L^s(\Omega_F))}
\| \widehat{\bU}\|_{X_{2,2}^T},
\\
\|(\mathcal{L}-\triangle )\widehat{\bU}\|
_{L^2(0,T;L^2(\Omega_F))}
&\leq CT^{\frac{1}{4}}\|\widehat{\bU}\|_{X_{2,2}^T},
\\
\|\mathcal{M}\widehat{\bU}\|_{L^2(0,T;L^2(\Omega_F))}
&\leq CT^{\frac{1}{4}}
\|\widehat{\bU}\|_{X_{2,2}^T},
\\
\|(\mathcal{G}-\nabla )\widehat{P}\|_{L^2(0,T;L^2(\Omega_F))}
&\leq CT^{\frac{1}{2}}\|\widehat{P}\|_{Y_{2,2}^T},
\end{align*}
for $T\leq 1$, where constant $C>0$ depends on $T$ nondecreasingly.
\end{lemma}
\begin{proof}
Lemma \ref{estimates_1} and \ref{estimates_2} imply the estimate for the convective term.
\begin{align*}
\|\mathcal{N}\widehat{\bU} \|_{L^2(0,T;L^2(\Omega_F))}^2
&= \int_{0}^{T}\int_{\Omega_F }
\left(
|({\widetilde{\bU}}\cdot \nabla )\widehat{\bU}|^2
+ \sum_{i,j,k=1}^n |\Gamma ^i_{jk} {\widetilde{\bU}_j}\widehat{\bU}_k|^2
\right)\,\d\by\d\tau\\
&\leq 
\|({\widetilde{\bU}}\cdot \nabla )\widehat{\bU}\|_{L^2(0,T;L^2(\Omega_F))}^2
+ 
\left(
\sup_{i,j,k} \|\Gamma ^i_{jk}\|_{\infty,\infty}^2
\right)
\sum_{i,j,k=1}^n\int_{0}^{T}\int_{\Omega_F } 
|{\widetilde{\bU}_j}\widehat{\bU}_k|^2\,\d\by\d\tau
\\
&\leq C(1+T)
\|{\widetilde{\bU}}\|_{L^r(0,T;L^s(\Omega_F))}^{2}
\| \widehat{\bU}\|_{X_{2,2}^T}^{2}.
\end{align*}
The second estimate comes from Lemma \ref{estimates_1} 
\begin{align*}
\|(\mathcal{L}-\triangle )\widehat{\bU}\|&
_{L^2(0,T;L^2(\Omega_F))}
\leq  C \left(\sup _{j, k}\left\|g^{j k}-\delta_{j k}\right\|_{L^{\infty}(0,T;L^{\infty}(\Omega))}\|\Delta \widehat{\bU}\|_{L^2(0,T;L^2(\Omega_F))}\right.
\\
&+\left(\left\|\partial_{j} g^{j k}\right\|_{L^{\infty}(0,T;L^{\infty}(\Omega))}+\left\|g^{k i}\right\|_{L^{\infty}(0,T;L^{\infty}(\Omega))}\left\|\Gamma_{j k}^{i}\right\|_{L^{\infty}(0,T;L^{\infty}(\Omega))}\right)
T^{\frac{1}{4}}\|\nabla \widehat{\bU}\|_{L^4(0,T;L^2(\Omega_F))}
\\
&+\left(\left\|\partial_{k} g^{k l}\right\|_{L^{\infty}(0,T;L^{\infty}(\Omega))}\left\|\Gamma_{k l}^{i}\right\|_{L^{\infty}(0,T;L^{\infty}(\Omega))}+\left\|g^{k l}\right\|_{L^{\infty}(0,T;L^{\infty}(\Omega))}\left\|\partial_{k} \Gamma_{k l}^{i}\right\|_{L^{\infty}(0,T;L^{\infty}(\Omega))}\right.
\\
&\left.\left.+\left\|g^{k l}\right\|_{L^{\infty}(0,T;L^{\infty}(\Omega))}\left\|\Gamma_{j l}^{m}\right\|_{L^{\infty}(0,T;L^{\infty}(\Omega))}\left\|\Gamma_{k m}^{i}\right\|_{L^{\infty}(0,T;L^{\infty}(\Omega))}\right)
T^{\frac{1}{4}}\| \widehat{\bU}\|_{L^4(0,T;L^2(\Omega_F))}
\right) 
\\
&\leq C
\left(
T^{\frac{1}{4}} + T
\right)
\|\widehat{\bU}\|_{X_{2,2}^T},
\end{align*}
\begin{align*}
\|\mathcal{M}\widehat{\bU}\|&_{L^{2}(0,T;L^{2}(\Omega_F))}
=\left\| 
\sum _{j=1}^n \dot{\mathbf{Y}}_j \partial _j
\widehat{\bU}_i 
+ \sum _{j,k=1}^n \Big(\Gamma _{jk}^i \dot{\mathbf{Y}}_k +
(\partial _k \mathbf{Y}_i)(\partial _j \dot{\mathbf{X}}_k)\Big)\widehat{\bU}_j
\right\|_{L^{2}(0,T;L^{2}(\Omega_F))}
\\
&\leq
T^{\frac{1}{4}}\|{\dot{\mathbf{Y}}}\|_{L^{\infty}(0,T;L^{\infty}(\Omega))} \|\nabla\widehat{\bU}\|_{L^{4}(0,T;L^{2}(\Omega_F))}\\
&+ (\sup_{i,j,k}\|\Gamma_{j k}^{i}\|_{L^{\infty}(0,T;L^{\infty}(\Omega))}
\|\dot{\mathbf{Y}}\|_{L^{\infty}(0,T;L^{\infty}(\Omega))}
\\
&\qquad\qquad\qquad
+ \|\nabla{\mathbf{Y}}\|_{L^{\infty}(0,T;L^{\infty}(\Omega))}
\|\nabla\dot{\mathbf{X}}\|_{L^{\infty}(0,T;L^{\infty}(\Omega))}
)
T^{\frac{1}{4}}
\|\widehat{\bU}\|_{L^{4}(0,T;L^{2}(\Omega_F))}
\\
&
\leq
CT^{\frac{1}{4}}
\|\widehat{\bU}\|_{X_{2,2}^T},
\end{align*}
since
$
L^2(0,T;H^2(\Omega_{F}))\cap H^1(0,T;L^2(\Omega_{F}))
\hookrightarrow H^{\frac{1}{2}}(0,T;H^{1}(\Omega_{F}))
\hookrightarrow L^{4}(0,T;H^{1}(\Omega_{F}))
$
and $\widetilde{\bA},\widetilde{\bOmega}\in L^{\infty}(0,T)$.

Finally, for the pressure term we get
\begin{align*}
\|(\mathcal{G}-\nabla )\widehat{P}\|_{L^{2}(0,T;L^{2}(\Omega_F))}
&\leq
C\sup _{ j, k}\left\|g^{j k}-\delta_{j k}\right\|_{L^{\infty}(0,T;L^{\infty}(\Omega_F))}
\|\nabla\widehat{P}\|_{L^{2}(0,T;L^{2}(\Omega_F))}\\
&\leq C\left(T^{\frac{1}{2}}+T\right)\|\nabla\widehat{P}\|_{L^{2}(0,T;L^{2}(\Omega_F))}.
\end{align*}
\end{proof}

\begin{corollary}	\label{estimates_4}
	Assume that
	$
	\widetilde{\bA},\widetilde{\bOmega}\in L^{p}(0,T)$, for all $1\leq p<\infty$
	and 
	$
	\widetilde{\bU}\in X_{2m,2m}^T,
	$
	for some $m\in\N$.
	Then for
	$\widehat{\bU}\in X_{2(m+1),2(m+1)}^T
	, \widehat{P}\in Y_{2(m+1),2(m+1)}^T
	$
	the following estimates hold:
	\begin{align*}
	\|\mathcal{N}\widehat{\bU} \|_{L^{2(m+1)}(0,T;L^{2(m+1)}(\Omega_F))}
	&\leq C
	\|{\widetilde{\bU}}\|_{X_{2m,2m}^T}
	\| \widehat{\bU}\|_{X_{2(m+1),2(m+1)}^T},
	\\
	\|(\mathcal{L}-\triangle )\widehat{\bU}\|
	_{L^{2(m+1)}(0,T;L^{2(m+1)}(\Omega_F))}
	&\leq CT^{\frac{1}{2(m+1)}}
        \|\widehat{\bU}\|_{X_{2(m+1),2(m+1)}^T}
	\\
	\|\mathcal{M}\widehat{\bU}\|_{L^{2(m+1)}(0,T;L^{2(m+1)}(\Omega_F))}
	&\leq C
	T^{\frac{1}{2(m+1)}}
	\|\widehat{\bU}\|_{X_{2(m+1),2(m+1)}^T},
	\\
	\|(\mathcal{G}-\nabla )\widehat{P}\|_{L^{2(m+1)}(0,T;L^{2(m+1)}(\Omega_F))}
	&\leq CT^{\frac{1}{2}}\|\widehat{P}\|_{Y_{2(m+1),2(m+1)}^T},
	\end{align*}
	for $T\leq 1$, where a constant $C>0$ depends on $T$ nondecreasingly.
\end{corollary}
\begin{proof}
	Let $m\in\N$. By using the interpolation result
	in \cite[{Theorem 5.2}]{Amann2000} for $s_0=1, s_1=0$, $p_0=p_1=2m$ we conclude that
	\begin{equation*}
	X_{2m,2m}^T=L^{2m}
	(0,T;W^{2,2m}(\Omega_F))\cap W^{1,2m}
	(0,T;L^{2m}(\Omega_F))\hookrightarrow W^{s,2m}(0,T;W^{2\theta,2m}(\Omega_F)),
	\end{equation*}		
	for all $\theta\in(0,1)$ and $s<1-\theta$.
	Now for 
	$
	\theta=\frac{3}{4m(m+1)},
	$
	we can choose $\frac{1}{2m}<s<1-\theta$ since 
	\begin{equation*}
	1-\theta-\frac{1}{2m}
	=\frac{4m^2+2m-5}{4m(m+1)}
	\geq \frac{1}{4m(m+1)}>0,
	\end{equation*}
	and get
	\begin{align*}
	&\widetilde{\bU}\in X_{2m,2m}^T\hookrightarrow 
	\hookrightarrow W^{s,2m}(0,T;W^{2\theta,2m}(\Omega_F))\hookrightarrow L^{\infty}(0,T;L^{2(m+1)}(\Omega_F)),
	\\
	&\widehat{\bU}\in X_{2(m+1),2(m+1)}^T
	=L^{2(m+1)}
	(0,T;W^{2,2(m+1)}(\Omega_F))\cap W^{1,2(m+1)}
	(0,T;L^{2(m+1)}(\Omega_F))
	\\
	&\qquad\qquad\qquad\qquad\qquad
	\hookrightarrow
	L^{2(m+1)}
	(0,T;W^{1,\infty}(\Omega_F)),
	\end{align*}
	Therefore, we have
	\begin{equation*}
	\|({\widetilde{\bU}}\cdot \nabla )\widehat{\bU}\|_{L^{2(m+1)}(0,T;L^{2(m+1)}(\Omega_F))}
	\leq
	\|{\widetilde{\bU}}\|_{L^{\infty}(0,T;L^{2(m+1)}(\Omega_F))}
	\| \nabla \widehat{\bU}\|_{L^{2(m+1)}(0,T;L^{\infty}(\Omega_F))}.
	\end{equation*}
	and by Lemma \ref{estimates_1} we get
	\begin{equation*}
	\|\mathcal{N}\widehat{\bU} \|_{L^{2(m+1)}(0,T;L^{2(m+1)}(\Omega_F))}
	\leq C
	\|{\widetilde{\bU}}\|_{X_{2m,2m}^T}
	\| \widehat{\bU}\|_{X_{2(m+1),2(m+1)}^T}.
	\end{equation*}
	Next, we have
	\begin{align*}
	&\|\mathcal{M}\widehat{\bU}\|_{L^{2(m+1)}(0,T;L^{2(m+1)}(\Omega_F))}
	=\left\| 
	\sum _{j=1}^n \dot{\mathbf{Y}}_j \partial _j
	\widehat{\bU}_i 
	+ \sum _{j,k=1}^n \Big(\Gamma _{jk}^i \dot{\mathbf{Y}}_k +
	(\partial _k \mathbf{Y}_i)(\partial _j \dot{\mathbf{X}}_k)\Big)\widehat{\bU}_j
	\right\|_{L^{2(m+1)}(0,T;L^{2(m+1)}(\Omega_F))}
	\\
	&\quad\leq
        T^{\frac{1}{2(m+1)}}\|\dot{\bY}\|_{L^{\infty}(0,T;L^{\infty}(\Omega_F))}
	\|\nabla\widehat{\bU}\|_{L^{\infty}(0,T;L^{2(m+1)}(\Omega_F))}
	\\
	&\quad\qquad+ \left(\sup_{i,j,k}\|\Gamma_{j k}^{i}\|_{L^{\infty}L^{\infty}}\|\dot{\mathbf{Y}}\|_{L^{\infty}L^{\infty}}
	+ \|\nabla{\mathbf{Y}}\|_{L^{\infty}L^{\infty}}
	\|\nabla\dot{\mathbf{X}}\|_{L^{\infty}L^{\infty}}\right)
	T^{\frac{1}{2(m+1)}}\|\widehat{\bU}\|_{L^{\infty}(0,T;L^{2(m+1)}(\Omega_F))}
	\\
	&\quad\leq C
	T^{\frac{1}{2(m+1)}}
	\|\widehat{\bU}\|_{X_{2(m+1),2(m+1)}^T}.
	\end{align*}
	The last inequality follows by embedding 
	$$
	X_{2(m+1),2(m+1)}^T
	\hookrightarrow W^{\frac{3}{8},2(m+1)}(0,T;W^{1,2(m+1)}(\Omega_F))
	\hookrightarrow
	L^{\infty}(0,T;W^{1,2(m+1)}(\Omega_F)).
	$$ 
	The other terms can be estimated in a similar way as before. In the same way, we can get estimates for arbitrary $m\in\N$.
\end{proof}

\subsection{Proof of Proposition \ref{strong2}}

Now, we can finish the proof of Proposition \ref{strong2}. First we are going to show that $\mathcal{S}(\mathcal{K}_R)\subset \mathcal{K}_R$ for suitably chosen $R$ and $T$.
For $(\widehat{\bU},\widehat{P},\widehat{\bA},\widehat{\bOmega})\in\mathcal{K}_R$, Lemma \ref{estimates_3} implies that
\begin{align*}
&\|F^{*}\|_{L^2(0,T;L^2(\Omega_F))}
=\left\|(\mathcal{L}-\triangle )\widehat{\bU}-\mathcal{M}\widehat{\bU}
-{{\mathcal{N}}\widehat{\bU}}-(\mathcal{G}-\nabla )\widehat{P}
+ \mathcal{R}\right\|_{L^2(0,T;L^2(\Omega_F))}\\
&\leq
\|(\mathcal{L}-\triangle )\widehat{\bU}\|
_{L^2L^2}
+\|\mathcal{M}\widehat{\bU}\|_{L^2(0,T;L^2(\Omega_F))}
+\|{{\mathcal{N}}\widehat{\bU}}\|_{L^2(0,T;L^2(\Omega_F))}
+\|(\mathcal{G}-\nabla )\widehat{P}\|_{L^2(0,T;L^2(\Omega_F))}\\
&\quad +\|\mathcal{R}\|_{L^2(0,T;L^2(\Omega_F))}
\\
&\leq C(T^{\frac{1}{4}}
+
\|{\widetilde{\bU}}\|_{L^r(0,T;L^s(\Omega_F))}
)
\left(
\|\widehat{\bU}\|_{X_{2,2}^T}
+
\|\widehat{P}\|_{Y_{2,2}^T}
\right)
+\|\mathcal{R}\|_{L^2(0,T;L^2(\Omega_F))},
\end{align*}
and since
\begin{align*}
\|G^{*}\|_{L^2(0,T)}
&\leq 
\|\widetilde{\bOmega}\times \widehat{\bA}\|_{L^2(0,T)}
+ \|\mathcal{R}_{\ba}\|_{L^2(0,T)}\\
&\leq
\|\widetilde{\bOmega}\|_{L^2(0,T)}\|\widehat{\bA}\|_{L^{\infty}(0,T)}
+ \|\mathcal{R}_{\ba}\|_{L^2(0,T)}
\leq
\|\widetilde{\bOmega}\|_{L^2(0,T)}\|\widehat{\bA}\|_{H^{1}(0,T)}
+ \|\mathcal{R}_{\ba}\|_{L^2(0,T)}
\\
\|H^{*}\|_{L^2(0,T)}
&\leq
\|\widetilde{\bOmega}\times (\mathcal{J}\widehat{\bOmega})\|_{L^2(0,T)}
+ \|\mathcal{R}_{\bomega}\|_{L^2(0,T)}
\leq
C\left(
\|\widetilde{\bOmega}\|_{L^2(0,T)}\|\widehat{\bOmega}\|_{H^{1}(0,T)}
+ \|\mathcal{R}_{\bomega}\|_{L^2(0,T)}
\right)
\end{align*}
we obtain
\begin{align*}
\|F^{*}\|_{L^2(0,T;L^2(\Omega_F))}
&+ \|G^{*}\|_{L^2(0,T)}
+\|H^{*}\|_{L^2(0,T)}
\\
&\leq
C(
T^{\frac{1}{4}}
+ \|\widetilde{\bU}\|_{L^r(0,T;L^s(\Omega_F))}
+ \|\widetilde{\bA}\|_{L^2(0,T)}
+ \|\widetilde{\bOmega}\|_{L^2(0,T)}
)
\\
&\qquad\qquad\qquad\qquad
\left(
\|\widehat{\bU}\|_{X_{2,2}^T}
+
\|\widehat{P}\|_{Y_{2,2}^T}
+ \|\widehat{\bA}\|_{H^1(0,T)}
+ \|\widehat{\bOmega}\|_{H^1(0,T)}
\right)
\\
&\qquad+ \|\mathcal{R}\|_{L^2(0,T;L^2(\Omega_F))} + \|\mathcal{R}_{\ba}\|_{L^2(0,T)} +\|\mathcal{R}_{\bomega}\|_{L^2(0,T)}
\\
&\leq C(
T^{\frac{1}{4}}
+ \|\widetilde{\bU}\|_{L^r(0,T;L^s(\Omega_F))}
+ \|\widetilde{\bA}\|_{L^2(0,T)}
+ \|\widetilde{\bOmega}\|_{L^2(0,T)}
)R
+ \frac{1}{2C_0} R
\end{align*}
for $R>0$ large enough, where $C_0>0$ is the constant form Theorem \ref{regularity_fixedPoint}. Also, we can choose $T=T_0>0$ small enough, i.e. such that
$$
C(
T^{\frac{1}{4}}
+ \|\widetilde{\bU}\|_{L^r(0,T;L^s(\Omega_F))}
+ \|\widetilde{\bA}\|_{L^2(0,T)}
+ \|\widetilde{\bOmega}\|_{L^2(0,T)}
) < \frac{1}{2C_0}
$$
and, therefore we get
$$
C_0\left(\|F^{*}\|_{L^2(0,T;L^2(\Omega_F))}
+ \|G^{*}\|_{L^2(0,T)}
+\|H^{*}\|_{L^2(0,T)}\right)
\leq R.
$$
By Theorem \ref{regularity_fixedPoint} we conclude that $\mathcal{S}$ is well defined on $\mathcal{K}_R$, and $\mathcal{S}(\mathcal{K}_R)\subset \mathcal{K}_R$. It remains to show that $\mathcal{S}$ is a contraction.

For 
$(\widehat{\bU},\widehat{P},\widehat{\bA},\widehat{\bOmega})\in\mathcal{K}_R
$
and 
$$
\widetilde{\bU}\in L^{2}(0,T;V(t))\cap L^{\infty}(0,T;L^2(\Omega_F ))\cap L^{r}(0,T;L^{s}(\Omega_F)),
\quad
\widetilde{\bOmega}\in L^{2}(0,T),
$$
we have
\begin{align*}
\|F^{*}(\widehat{\bU}_1,\widehat{P}_1) &- F^{*}(\widehat{\bU}_2, \widehat{P}_2)\|_{L^2(0,T;L^2(\Omega_F))}
\leq C
(
T^{\frac{1}{4}}
+
\|{\widetilde{\bU}}\|_{L^r(0,T;L^s(\Omega_F))}
)
\left(
\|\widehat{\bU}_1-\widehat{\bU}_2\|_{X_{2,2}^{T}}
+
\|\widehat{P}_1-\widehat{P}_2\|_{Y_{2,2}^{T}}
\right)
\end{align*}
\begin{align*}
\|G^{*}(\widehat{\bA}_1)-G^{*}(\widehat{\bA}_2)\|_{L^2}
&\leq 
\|\widetilde{\bOmega}\|_{L^2}\|\widehat{\bA}_1-\widehat{\bA}_2\|_{H^{1}}
\\
\|H^{*}(\widehat{\bOmega}_1)-H^{*}(\widehat{\bOmega}_2)\|_{L^2}
&\leq c\|\widetilde{\bOmega}\|_{L^2}\|\widehat{\bOmega}_1-\widehat{\bOmega}_2\|_{H^{1}}
\end{align*}
Hence,
\begin{align*}
\|\mathcal{S}(\widehat{\bU}_1,\widehat{P}_1,\widehat{\bA}_1,\widehat{\bOmega}_1) 
&- \mathcal{S}(\widehat{\bU}_2,\widehat{P}_2,\widehat{\bA}_2,\widehat{\bOmega}_2)\|
\\
&\leq C_0C
(
T^{\frac{1}{4}}
+
\|{\widetilde{\bU}}\|_{L^r(0,T;L^s(\Omega_F))}
+ \|\widetilde{\bA}\|_{L^2(0,T)}
+ \|\widetilde{\bOmega}\|_{L^2(0,T)}
)
\\
&\qquad
\left(
\|\widehat{\bU}_1-\widehat{\bU}_2\|_{X_{2,2}^T}
+
\|\widehat{P}_1-\widehat{P}_2\|_{Y_{2,2}^T}
+
\|\widehat{\bA}_1-\widehat{\bA}_2\|_{H^1(0,T)}
+
\|\widehat{\bOmega}_1-\widehat{\bOmega}_2\|_{H^1(0,T)}
\right)
\end{align*}
and, again, for $T=T_0>0$ small enough, we get
\begin{align*}
\|\mathcal{S}(\widehat{\bU}_1,\widehat{P}_1,\widehat{\bA}_1,\widehat{\bOmega}_1) 
&- \mathcal{S}(\widehat{\bU}_2,\widehat{P}_2,\widehat{\bA}_2,\widehat{\bOmega}_2)\|
\\
&\leq \mu
\| 
(\widehat{\bU}_1,\widehat{P}_1,\widehat{\bA}_1,\widehat{\bOmega}_1) 
- (\widehat{\bU}_2,\widehat{P}_2,\widehat{\bA}_2,\widehat{\bOmega}_2)
\|
\end{align*}
for some $0<\mu<1$, so $\mathcal{S}$ i contraction. Banach fixed point theorem implies that $\mathcal{S}$ has a unique fixed point $({\bU}^{*},{P}^{*},{\bA}^{*},{\bOmega}^{*})\in\mathcal{K}_R$, which is a unique solution to the problem \eqref{FixedDomain_stokes}-\eqref{right1}. Therefore, we have shown part a) of Proposition \ref{strong2}.

\begin{remark}
The choice of the time $T_0$ does not depend on the solution itself, but only on the norms 
$
\|{\widetilde{\bU}}\|_{L^r(0,T_0;L^s(\Omega_F))}, \|{\widetilde{\bA}}\|_{L^2(0,T_0)}
$
and
$
\|{\widetilde{\bOmega}}\|_{L^2(0,T_0)}.
$	
Since the norms $
\|{\widetilde{\bU}}\|_{L^r(0,T;L^s(\Omega_F))}, \|{\widetilde{\bA}}\|_{L^2(0,T)}
$
and
$
\|{\widetilde{\bOmega}}\|_{L^2(0,T)}
$
are given, we can split the interval $\left[0,T\right]$ into smaller ones $\left[T_{i-1},T_i\right]$, such that
$$
\|{\widetilde{\bU}}\|_{L^r(T_{i-1},T_{i};L^s(\Omega_F))}
+ \|{\widetilde{\bOmega}}\|_{L^2(0,T)}
+ \|{\widetilde{\bOmega}}\|_{L^2(0,T_0)}
< \varepsilon
$$
repeat the procedure at each interval.
\end{remark}

For part b) of Proposition \ref{strong2} we can proceed as in part a) by using Corollary \ref{estimates_4} instead of Lemma \ref{estimates_3}.

Notice that Theorem \ref{mainResult2} follows directly from our construction and Proposition \ref{strong2}. Namely, Proposition \ref{strong2} a), change of coordinates and Lemma \ref{uniqueness_1} implies
\begin{align*}
&\bU, \widetilde{\bU}\in L^{2}(\varepsilon,T;H^{2}(\Omega_F))\cap
H^{1}(\varepsilon,T;L^{2}(\Omega_F)),
\\
&P, \widetilde{P}\in L^{2}(\varepsilon,T;{\quotient{H^{1}(\Omega_F)}{\R}}),
\\
&\bA,\bOmega, \widetilde{\bA}, \widetilde{\bOmega}\in
H^{1}(\varepsilon,T),
\end{align*}
for all $\varepsilon>0$.
By induction and  using Proposition \ref{strong2} b) we get property \eqref{LpStrong}.

\section{Higher time derivatives estimates}
\label{Section:time_derivatives}
To summarize, in the previous Section we proved that:
\begin{equation}	\label{strong_p}
\begin{split}
&\widetilde{\bU},{\bU}\in 
W^{1,p}(\varepsilon,T;L^p(\Omega_F))\cap L^{p}(\varepsilon,T;W^{2,p}(\Omega_F)),
\\
&\widetilde{P},P
\in L^p(\varepsilon,T;{\quotient{W^{1,p}(\Omega_F)}{\R}}),
\\
&\widetilde{\bA},{\bA}, \widetilde{\bOmega},{\bOmega}\in W^{1,p}(\varepsilon,T),
\end{split}
\end{equation}
for all $\varepsilon>0$ and all $1\leq p<\infty$.
Now, we want to show inductively that all time derivatives of the solution has the following regularity properties:
\begin{equation}	\label{time1}
\begin{split}
&\partial_t^{l}\widetilde{\bU},\partial_t^{l}{\bU}\in W^{1,p}(\varepsilon,T;L^p(\Omega_F))\cap L^{p}(\varepsilon,T;W^{2,p}(\Omega_F)),
\\
&\partial_t^{l}\widetilde{P},\partial_t^{l}P
\in L^p(\varepsilon,T;{\quotient{W^{1,p}(\Omega_F)}{\R}}),
\\
&\frac{\d^{l}}{\dt^{l}}\widetilde{\bA},\frac{\d^{l}}{\dt^{l}}{\bA}, \frac{\d^{l}}{\dt^{l}}\widetilde{\bOmega},\frac{\d^{l}}{\dt^{l}}{\bOmega}\in W^{1,p}(\varepsilon,T),
\end{split}
\end{equation}
for all $l\geq 1$ and all $\varepsilon>0$.

In this Section, for simplicity of presentation, we prove \eqref{time1} for $l=1$ and $p=2$, while the general case is done in Appendix, see Section \ref{Section:time_derivatives_induction}.
We consider the problem \eqref{FixedDomain_stokes} with right hand side
\begin{equation}	\label{right2}
\begin{split}
&F^{*} = F_1^{*} = F(\bU^{*},P^{*}) + tF_1(\bU,P) + \partial_t\bU,
\\
&G^{*} = G_1^{*} = G(\bA^{*}) + tG_1(\bA)  + \frac{\d}{\dt}\bA,
\\
&H^{*} = H_1^{*} = H(\bOmega^{*}) + tH_1(\bOmega) + \mathcal{J}\left(\frac{\d}{\dt}\bOmega\right),
\end{split}
\end{equation}
where
\begin{equation}\label{F1}
F_1(\bU,P)=\mathcal{L}_1(\bU)-\mathcal{M}_1(\bU)-\mathcal{N}_1(\bU)-\mathcal{G}_1(P).
\end{equation}
\begin{equation}\label{G1H1}
G_1 =
-\frac{\d}{\dt}\widetilde{\bOmega}\times \bA,
\quad
H_1 =
-\frac{\d}{\dt}\widetilde{\bOmega}\times (\mathcal{J}\bOmega),
\end{equation}
and $\mathcal{L}_1$, $\mathcal{M}_1$, $\mathcal{N}_1$ and $\mathcal{G}_1$ denote operators obtained by taking time derivative of the coefficients in operators $\mathcal{L}$, $\mathcal{M}$, $\mathcal{N}$ and $\mathcal{G}$
\begin{align}
(\mathcal{N}_1(\bU))_i 
&= \left(({\partial_{t}\widetilde{\bU}}\cdot \nabla ){\bU}\right)_i 
- \sum_{j,k=1}^n \partial_{t}(\Gamma ^i_{jk} {\widetilde{\bU}_j}){\bU}_k
\label{f2}\\
(\mathcal{M}_1(\bU))_i =&
\sum _{j=1}^n \partial_{t}(\dot{\mathbf{Y}}_j )\partial _j
\mathbf{U}_i 
+ 
\partial_{t}\Big(\Gamma _{jk}^i \dot{\mathbf{Y}}_k 
+(\partial _k \mathbf{Y}_i)(\partial _j \dot{\mathbf{X}}_k)\Big)\mathbf{U}_j
\label{f4}\\
(\mathcal{L}_1(\bU))_i&=
\sum_{j,k=1}^{n}\partial
_{j}(\partial_{t} g^{jk}\partial_k \mathbf{U}_{i})
+ 2\sum_{j,k,m=1}^{n}
\partial_{t}(g^{km}\Gamma_{jk}^{i})\partial _{m}\mathbf{U}_{j}  
\notag \\
& \qquad\qquad\qquad
+\sum_{j,k,m=1}^{n}
\partial_{t}\Big(
\partial_{k}(g^{km}\Gamma_{jm}^{i})
+\sum_{m=1}^{n}g^{km}\Gamma _{jm}^{m}\Gamma_{km}^{i}\Big)
\mathbf{U}_{j}
\label{f5}\\
\mathcal{G}_1(P)=&
\sum_{j=1}^{n}\partial_{t}g^{ij}\partial_{j}P.
\label{f6}
\end{align}

\begin{remark} \label{remark_t1}
This problem is obtained by formal differentiation of \eqref{FixedDomain_2} w.r.t. time variable, multiplication by $t$ (to cut-off initial condition), and setting
$$
\bU^{*} = \bU_1^{*} = t\partial_t \bU,
\quad
P^{*} = P_1^{*} = t\partial_t P,
\quad
\bA^{*} = \bA_1^{*} = t\frac{\d}{\dt}\bA,
\quad
\bOmega^{*} = \bOmega_1^{*} = t\frac{\d}{\dt}\bOmega.
$$
To have vanishing initial data, the time derivative $\partial_t\bU$ should not explode too fast at $t=0$. Since
	$$
	\bU\in
	L^2(\varepsilon,T;H^2(\Omega_F))
	\cap
	H^1(\varepsilon,T;L^2(\Omega_F)),
	\quad
	\forall \varepsilon>0,
	$$
	it is better to multiply by $t-\varepsilon'$ and look at the solution for $ t>\varepsilon'$. Here, for $\varepsilon>0$ arbitrary small, we can chose $\varepsilon<\varepsilon'<2\varepsilon$ such that $\partial_t^{l}\bU(\varepsilon')\in L^2(\Omega_F)$. After the translation $t\mapsto t+\varepsilon'$, we obtain \eqref{FixedDomain_stokes} with right hand side \eqref{right2}. Therefore, in the following we can replace $\varepsilon$ with 0 in the assumption \eqref{strong_p}.
\end{remark}
In what follows we show that described problem has a unique solution $(\bU^{*},P^{*},\bA^{*},\bOmega^{*})$ such that
\begin{equation*}
\begin{split}
&\bU^{*}\in H^1(0,T;L^2(\Omega_F))\cap L^2(0,T;H^{2}(\Omega_F)),
\\
&P^{*}\in L^2(0,T;{\quotient{H^{1}(\Omega_F)}{\R}}),
\\
&\bA^{*},\bOmega^{*}\in H^1(0,T),
\end{split}
\end{equation*}
and that it equals 
$
\left(
t\partial_t\bU,t\partial_t P,t\frac{\d}{\dt}\bA,t\frac{\d}{\dt}\bOmega
\right).
$
Then it follows that
\begin{equation*}
\begin{split}
&\partial_t\bU\in H^{1}(\varepsilon,T;L^2(\Omega_F))\cap L^{2}(\varepsilon,T;H^{2}(\Omega_F)),
\\
&\partial_t P\in L^{2}(\varepsilon,T;{\quotient{H^{1}(\Omega_F)}{\R}}),
\\
&\frac{\d}{\dt}\bA, \frac{\d}{\dt}\bOmega\in H^{1}(\varepsilon,T),
\end{split}
\end{equation*}
for all $\varepsilon>0$.
We use  Proposition \ref{strong2} with
\begin{equation*}
\begin{split}
&\mathcal{R} = tF_1(\bU,P)+\partial_t\bU,
\quad\mathcal{R}_{\ba} = tG_1(\bA)+\frac{\d}{\dt}\bA,
\quad\mathcal{R}_{\bomega} = tH_1(\bOmega)+\frac{\d}{\dt}\bOmega.
\end{split}
\end{equation*}
Therefore, it is sufficient to show that
\begin{equation*}
\begin{split}
&\mathcal{R}\in L^2(0,T;L^2(\Omega_{F}))
\quad\mathcal{R}_{\ba}\in L^2(0,T),
\quad\mathcal{R}_{\bomega} \in L^2(0,T).
\end{split}
\end{equation*}

The critical term  is \eqref{f2} which comes from the convective term. By Theorem  \ref{mainResult2} and Remark \ref{remark_t1} we have
\begin{align*}
\widetilde{\bU},\bU\in L^{4}(0,T;W^{2,4}(\Omega_F))\cap
W^{1,4}(0,T;L^{4}(\Omega_F)).
\end{align*}
Therefore
\begin{align*}
\|({\partial_t\widetilde{\bU}}\cdot \nabla ){\bU}\|_{L^{2}(0,T;L^{2}(\Omega_F))}^2
&= \int_{0}^{T}\int_{\Omega_F}
|\underbrace{{\partial_t\widetilde{\bU}}}_{\in L_t^4L_x^4}|^2 
\,|\underbrace{\nabla {\bU}}_{\in L_t^{4}W_x^{1,4}}|^2\,\d\by\d t
\leq
\int_{0}^{T}
\|{\partial_t\widetilde{\bU}}\|_{L^4(\Omega_F)}^2\| \nabla {\bU}\|_{L^{4}(\Omega_F)}^2\,\d t\\
&\leq
\|{\partial_t\widetilde{\bU}}\|_{L^{4}(0,T;L^{4}(\Omega_F))}^2
\|\nabla\bU\|_{L^{4}(0,T;L^{4}(\Omega_F))}^2
\leq
C
\|{\widetilde{\bU}}\|_{W^{1,4}(0,T;L^{4}(\Omega_F))}^2
\|\bU\|_{L^{4}(0,T;W^{2,4}(\Omega_F))}^2
\end{align*}
All other terms can be estimated as in Section \ref{Section:StrongSolution}, so by Proposition \ref{strong2} we conclude that there exists a unique strong solution $(\bU_1^{*},P_1^{*},\bA_1^{*},\bOmega_1^{*})$ satisfying
\begin{equation}\label{UStarRegularity}
\begin{split}
&\bU_1^{*}\in L^{2}(0,T;H^{2}(\Omega_F))\cap
H^{1}(0,T;L^{2}(\Omega_F)),
\\
& P_1^{*}\in L^{2}(0,T;{\quotient{H^{1}(\Omega_F)}{\R}}),
\\
&\bA_1^{*},\bOmega_1^{*}\in
H^{1}(0,T).
\end{split}
\end{equation}
It remains to prove that the obtained solution equals 
$
\left(
t\partial_t\bU,t\partial_t P,t\frac{\d}{\dt}\bA,t\frac{\d}{\dt}\bOmega
\right),
$
which will imply the statement of Proposition \ref{time_derivatives} for $l=1$ and $p=2$.

\subsection{$(\bU_1^{*},P_1^{*},\bA_1^{*},\bOmega_1^{*})=\left(t\partial_t\bU,t\partial_t P,t\frac{\d}{\dt}\bA,t\frac{\d}{\dt}\bOmega\right)$} 
\label{Section:Uniqueness}

So far we have shown that there is a unique strong solution $(\bU,P,\bA,\bOmega)$ of problem \eqref{FixedDomain_2} satisfying \eqref{strong_p}, and a unique strong solution
$(\bU_1^{*},P_1^{*},\bA_1^{*},\bOmega_1^{*})$
of \eqref{FixedDomain_stokes} with right hand side \eqref{right2} satisfying \eqref{UStarRegularity}.
In order to complete the proof of Proposition \ref{time_derivatives}, it is necessary to show that
$$
(\bU_1^{*},P_1^{*},\bA_1^{*},\bOmega_1^{*}) = \left(
t\partial_t\bU,t\partial_t P,t\frac{\d}{\dt}\bA,t\frac{\d}{\dt}\bOmega
\right).
$$
Notice that while the above equality formally holds, it is delicate to prove it because we do not have any information on $\partial_t P$.
We consider problem \eqref{FixedDomain_2} in the form
\begin{equation}	\label{equation1}
\begin{array}{l}
\left.
\begin{array}{l}
\partial _{t}{\bU} = \mathcal{F}(\bU) - \mathcal{G}(P), \\
\divg{\bU} = 0
\end{array}
\right\} \;\mathrm{in}\;(0,T)\times\Omega_{F},
\\
\left.
\begin{array}{l}
\frac{\d}{\dt}\bA
=-\int_{\partial S_{0}}\mathcal{T}({\bU},{P})\bN\,\d\gamma(\by)
+ G(\bA)
\\
\frac{\d}{\dt}(\mathcal{J}{\bOmega})
=
-\int_{\partial S_{0}}
\by\times {\mathcal{T}}({\bU},{P})\bN \,\d\gamma(\by)
+ H(\bOmega)
\end{array}
\right\} \;\mathrm{in}\;(0,T),
\\
{\bU} = {\bOmega}\times\by + {\bA}\qquad \mathrm{on}\;(0,T)\times\partial S_0,
\\
{\bU} = 0\qquad \mathrm{on}\;(0,T)\times\partial\Omega,
\\
{\bU}(0,.)=\bu_0\qquad \mathrm{in}\;\Omega,\quad {\bA}(0)=\ba_0,\quad
{\bOmega}(0)=\bomega_0
\end{array}
\end{equation}
where
\begin{equation}	\label{rightF2}
\begin{split}
&\mathcal{F}(\bU)
= \mathcal{L}{\bU}
-\mathcal{M}{\bU}
-{\mathcal{N}}{\bU},
\\
G(\bA) &= -\widetilde{\bOmega}\times {\bA},\quad
H(\bOmega) = -\widetilde{\bOmega}\times (\mathcal{J}{\bOmega}),
\end{split}
\end{equation}
and the problem \eqref{FixedDomain_stokes} with the right hand side \eqref{right2}
\begin{equation}	\label{equation2}
\begin{array}{l}
\left.
\begin{array}{l}
\partial_{t}{\bU}_1^{*} = \mathcal{F}(\bU_1^{*}) - \mathcal{G}(P_1^{*}) + tF_1(\bU,P) + \partial_t\bU , \\
\divg{\bU}_1^{*} = 0
\end{array}
\right\} \;\mathrm{in}\;(0,T)\times\Omega_{F},
\\
\left.
\begin{array}{l}
\frac{\d}{\dt}{\bA}_1^{*}
=-\int_{\partial S_{0}}\mathcal{T}({\bU}_1^{*},{P}_1^{*})\bN\,\d\gamma(\by)
+ G(\bA_1^{*})+ tG_1(\bA) + \frac{\d}{\dt}\bA
\\
\frac{\d}{\dt}(\mathcal{J}{\bOmega}_1^{*})
=
-\int_{\partial S_{0}}
\by\times {\mathcal{T}}({\bU}_1^{*},{P}_1^{*})\bN \,\d\gamma(\by)
+ H(\bOmega_1^{*})+ tH_1(\bOmega)
+ \mathcal{J}\left(\frac{\d}{\dt}\bOmega\right)
\end{array}
\right\} \;\mathrm{in}\;(0,T),
\\
{\bU}_1^{*} = {\bA}_1^{*} + {\bOmega}_1^{*}\times\by
\qquad \mathrm{on}\;(0,T)\times\partial S_0,
\\
{\bU}_1^{*} = 0\qquad \mathrm{on}\;(0,T)\times\partial\Omega,
\\
{\bU}_1^{*}(0,.)=0\qquad \mathrm{in}\;\Omega,\quad {\bA}_1^{*}(0)=0,\quad
{\bOmega}_1^{*}(0)=0.
\end{array}
\end{equation}
Operators $\mathcal{L}$, $\mathcal{M}$, $\mathcal{N}$ and $\mathcal{G}$ are defined by formulas \eqref{OpL}-\eqref{OpP}, operators $F_1, G_1$ and $H_1$ by \eqref{F1}-\eqref{G1H1}.

In order to compare solutions of \eqref{equation1} and \eqref{equation2} we would like to differentiate \eqref{equation1} with respect to time, but the right-hand side is not regular enough, since $\partial_t\bU\in L^2(\varepsilon,T;L^2(\Omega_F))$ and the pressure $P$ is not regular enough in time variable. This means that we have to use some regular approximations of the solution for \eqref{equation1}.
The idea is to use Galerkin's method. First we will multiply the equation $\eqref{equation1}_1$ by
\begin{equation*}
\mathbb{G}=(g_{ij})=\nabla\bX^T \nabla\bX
\end{equation*}
and obtain pressure term $\nabla P$ on the right-hand side in the equation
\begin{equation}\label{equation1_1}
\mathbb{G}\partial_t \bU=\mathbb{G}\mathcal{F}(\bU)-\nabla P\qquad\mathrm{in}\;(0,T)\times\Omega_{F},
\end{equation} 
since 
$
\mathbb{G}\mathcal{G}(P)=\nabla P.
$
Then we can get rid of the pressure by using a divergence-free test function, write down an approximative problem and to show that it has a unique solution which is a good approximation for $\bU$.
That allows us to differentiate the approximative problem with respect to time and get estimates for $\partial_t \bU$. Finally, we will show that $\bU_1^{*} = t\partial_t \bU$ by using the equation for approximative problem and the equation for $\bU_1^{*}$. The point is that in this way we will avoid the term with $\partial_t P$.	

We are going to use following function spaces
\begin{itemize}
	\item 
	$
	\mathbb{H}(\Omega_{F}) = \left\lbrace (\bv,\ba,\bomega)\in L^2(\Omega)\times\R^3\times\R^3 \,:\, \divg\bv=0,\, \bv\cdot\bn|_{\partial\Omega} = 0,\, \bv|_{S_0}(\by)=\ba + \bomega\times\by \right\rbrace 
	$
	\item 
	$
	\mathbb{V}(\Omega_{F}) = \left\lbrace (\bv,\ba,\bomega)\in H_0^1(\Omega)\times\R^3\times\R^3 \,:\, \divg\bv=0,\, \bv|_{S_0}(\by)=\ba + \bomega\times\by \right\rbrace 
	$
\end{itemize}
Let
$
\left\lbrace \,
\left(
{\bPsi}_i,
{\bPsi}^{\ba}_i,
{\bPsi}^{\bomega}_i
\right),
i\in\N
\,\right\rbrace 
$
be an orthonormal basis for $\mathbb{H}(\Omega_{F})$
with scalar product
\begin{equation}	\label{SkalProd}
\left( (\bphi,\bphi_{\ba},\bphi_{\bomega}),(\bpsi,\bpsi_{\ba},\bpsi_{\bomega})\right) 
= \int_{\Omega_F}\bpsi\cdot\bphi\,\d\by
+ \bphi_{\ba}\cdot\bpsi_{\ba} + \mathcal{J}\bphi_{\bomega}\cdot\bpsi_{\bomega}.
\end{equation}
We define finite-dimensional space
\begin{align*}
\mathbb{V}_m 
=span\left\lbrace\,
\left(
{\bPsi}_1,
{\bPsi}^{\ba}_1,
{\bPsi}^{\bomega}_1
\right),...,
\left(
{\bPsi}_m,
{\bPsi}^{\ba}_m,
{\bPsi}^{\bomega}_m
\right)\,
\right\rbrace 
\end{align*}
For $m\in\N$ we observe the following approximate problem
\begin{equation}\label{approx}
\begin{split}
	&\int_{\Omega_F}\mathbb{G}\partial_{t}{\bU^m}
	\cdot{\bPsi}_i
	\, \d\by
	+ \frac{\d}{\dt}\bA^m\cdot{\bPsi}_i^{\ba}
	+ \frac{\d}{\dt}(\mathcal{J}{\bOmega^m})\cdot{\bPsi}_i^{\bomega}
	\\
	&\qquad+ \left\langle \mathbb{G}\mathcal{F}(\bU^m),{\bPsi}_i\right\rangle
	- G(\bA^m)\cdot{\bPsi}_i^{\ba}
	- H(\bOmega^m)\cdot{\bPsi}_i^{\bomega}
	= 0
	\end{split}
\end{equation}
for $i=1,...,m$, where
\begin{equation}	\label{weakF}
	\left\langle \mathbb{G}\mathcal{F}(\bU^m),{\bPsi}_i\right\rangle 
	=
	\left\langle \mathbb{G}\mathcal{L}\bU^m,{\bPsi}_i\right\rangle 
	+ \int_{\Omega_F}\mathbb{G}(\mathcal{M}\bU^m + \mathcal{N}\bU^m)
	\cdot{\bPsi}_i
	\, \d\by,
\end{equation} 
\begin{equation}\label{weakL}
		\begin{split}
		\left\langle \mathbb{G}\mathcal{L}\bU,\bpsi\right\rangle 
		&= \int_{\Omega_{F}} \Big(
		2\D\bU\cdot\D\bpsi
		+ (g_{ik}g^{jl}-\delta_{ik}\delta_{jl})\partial_j\bU_i\partial_l\bpsi_k
		\\
		&\qquad
		+\left(\Gamma_{kl}^m g_{im}g^{jl} + \Gamma_{ik}^j\right)\partial_j\bU_i\bpsi_k
		+\left(\Gamma_{ij}^m g_{im}g^{jl} + \Gamma_{ij}^l\delta_{jk}\right)\bU_i\partial_l\bpsi_k
		\\
		&\qquad
		+\left(\Gamma_{ij}^m\Gamma_{kl}^p g_{mp}g^{jl} + \Gamma_{ij}^l\Gamma_{kl}^j\right)\bU_i\bpsi_k
		\Big)\d\by
		\end{split}
\end{equation}
and
\begin{equation*}
\left(
\bU^m(t,\by),
\bA^m(t),
\bOmega^m(t)
\right)
=  \sum_{j=1}^{m} c_{jm}(t)
\left(
{\bPsi}_j(\by),
{\bPsi}^{\ba}_j,
{\bPsi}^{\bomega}_j
\right),
\qquad {c_{jm}\in H^{1}(0,T)}.
\end{equation*}
Equation \eqref{approx} is obtained by summing the equation \eqref{equation1_1} multiplied by the test function ${\bPsi}_i$ and integrated over $\Omega_{F}$ with the equations $\eqref{equation1}_3$ and $\eqref{equation1}_4$ multiplied by the test functions ${\bPsi}_i^{\ba}$ and ${\bPsi}_i^{\bomega}$, respectively.
Then from \eqref{weakL} using integration by parts we obtain
\begin{align*}
\left\langle \mathbb{G}\mathcal{L}(\bU),\bpsi\right\rangle
&=-\int_{\Omega_F}\mathbb{G}\mathcal{L}\bU\cdot\bpsi
\, \d\by 
+ \int_{\partial S_0}\mathcal{T}(\bU,P)\bN\cdot\bpsi\, \d\gamma(\by)
+ \int_{\Omega_F}\nabla P\cdot\bpsi
\, \d\by,
\end{align*}
and therefore,
\begin{equation}\label{weakF_2}
\left\langle \mathbb{G}\mathcal{F}(\bU),\bpsi\right\rangle
=-\int_{\Omega_F}\mathbb{G}\mathcal{F}(\bU)\cdot\bpsi
\, \d\by 
+ \int_{\partial S_0}\mathcal{T}(\bU,P)\bN\cdot\bpsi\, \d\gamma(\by)
+ \int_{\Omega_F}\nabla P\cdot\bpsi
\, \d\by,
\end{equation}
for regular enough and divergence-free functions $\bU, \bpsi$.	

Together with the initial conditions
\begin{equation*}
\left(
\bU^m(0,\by),
\bA^m(0),
\bOmega^m(0)
\right)
=  \sum_{j=1}^{m} c_{jm}^0
\left(
{\bPsi}_j(\by),
{\bPsi}^{\ba}_j,
{\bPsi}^{\bomega}_j
\right),
\qquad c_{jm}^0=\left(
\left(
\bu_0,
\ba_0,
\bomega_0
\right),
\left(
{\bPsi}_j(\by),
{\bPsi}^{\ba}_j,
{\bPsi}^{\bomega}_j
\right)
\right). 
\end{equation*}
there exists a unique solution $(\bU^m,\bA^m,\bOmega^m)$ for \eqref{approx} on some interval $[0,T_m]$, $T_m\leq T$ ($c_{jm}\in H^1(0,T_k)$). 

To show that $(\bU^m,\bA^m,\bOmega^m)$ converges to the solution $(\bU,\bA,\bOmega)$, we have to derive the energy estimates. 
We multiply \eqref{approx} by $c_{jm}$, and sum over $j$ from $1$ to $m$, and if we go back to physical domain $\Omega_F(t)$ with 
\begin{equation*}
	\bu^m(t,\bx) = \nabla \bX(t,\bY(t,\bx))\bU^m(t,\bY(t,\bx)),
	\quad
	\ba^m(t) = \Q(t)\bA^m(t),
	\quad
	\bomega^m(t) = \Q(t)\bOmega^m(t)
\end{equation*}
we will obtain
\begin{equation*}
	\begin{split}
	&\int_{\Omega_F(t)}\partial_{t}{\bu^m}
	\cdot\bu^m
	\, \d\bx
	+ \frac{\d}{\dt}(\J{\bomega^m})\cdot\bomega^m
	+ \frac{\d}{\dt}\ba^m\cdot\ba^m
	+\int_{\Omega_{F}(t)}2|\D\bu^m|^2\,\d\bx
	+ \int_{\Omega_F(t)} \widetilde{\bu}\cdot\nabla\bu^m\cdot\bu^m
	\, \d\bx
	= 0.
	\end{split}
\end{equation*}
Since
\begin{align*}
\int_{\Omega_F(t)}\partial_{t}{\bu^m}
\cdot\bu^m
\, \d\bx
+ \frac{\d}{\dt}(\J{\bomega^m})\cdot\bomega^m
+ \frac{\d}{\dt}\ba^m\cdot\ba^m
&= \int_{\Omega_F(t)}\frac{1}{2}\partial_{t}|\bu^m|^2
\, \d\bx
+ \frac{\d}{\dt}(\J{\bomega^m})\cdot\bomega^m
+ \frac{\d}{\dt}\ba^m\cdot\ba^m
\\
&= \frac{1}{2}\frac{\d}{\dt}\|\bu^m(t)\|_{L^2(\Omega)}^2
- \frac{1}{2} \int_{\Omega_{F}(t)}\widetilde{\bu}\cdot\nabla|\bu^m|^2\,\d\bx
\end{align*}
it follows that
\begin{equation*}
	\frac{1}{2}\frac{\d}{\dt}\|\bu^m(t)\|_{L^2(\Omega)}^2
	+\int_{\Omega_{F}(t)}2|\D\bu^m|^2\,\d\bx
	= 0.
\end{equation*}
By integrating the equality on $(0,t)$ we obtain the estimate
\begin{equation*}
\frac{1}{2}\|\bu^m(t)\|_{L^2(\Omega)}^2
+ 2\int_{0}^{t}\|\D\bu^m\|_{L^2(\Omega_{F}(\tau))}^2\,\d\tau
= \frac{1}{2}\|\bu^m(0)\|_{L^2(\Omega)}^2
\leq \frac{1}{2}\|\bu_0\|_{L^2(\Omega)}^2,
\quad\forall t\in [0,T_m]
\end{equation*}
from which we conclude that $|c_{jm}(t)|<\|\bu_0\|_{L^2(\Omega)}$, so the inequality holds for all $t\in[0,T]$.
Hence, $\bu^m$ is bounded in $L^{\infty}(0,T;L^2(\Omega))\cap L^2(0,T;H^1(\Omega))$ which implies that $\bu^m\to\bar{\bu}$  weakly in $L^2(0,T;H^1(\Omega))$ and weakly-* in $L^{\infty}(0,T;L^2(\Omega))$. 

Let us show that $\bar{\bu}$ is a weak solution for \eqref{FSINoslip}. 
We take test function
\begin{equation*}
\left(
\bpsi(t,\by),
\bpsi_{\ba}(t),
\bpsi_{\bomega}
\right)
= h(t)
\left(
{\bPsi}_j(\by),
{\bPsi}^{\ba}_j,
{\bPsi}^{\bomega}_j
\right)
\qquad h\in C_c^{\infty}([0,T)),
\end{equation*}
multiply \eqref{approx} by $h(t)$ and write the equation in the physical domain $\Omega_F(t)$ 
\begin{equation*}
\begin{split}
&\int_{\Omega_F(t)}\partial_{t}{\bu^m}
\cdot\bphi
\, \d\bx
+ \frac{\d}{\dt}(\J{\bomega^m})\cdot\bphi_{\bomega}
+ \frac{\d}{\dt}\ba^m\cdot\bphi_{\ba}
+ \int_{\Omega_F(t)}2\D{\bu^m}
\cdot\D\bphi
\, \d\bx
+ \int_{\Omega_F(t)} \widetilde{\bu}\cdot\nabla\bu^m\cdot\bphi
\, \d\bx
= 0
\end{split}
\end{equation*}
with
\begin{equation*}
\bphi(t,\bx) = \nabla \bX(t,\bY(t,\bx))\bpsi(t,\bY(t,\bx)),
\quad
\bphi_{\ba}(t) = \Q(t)\bpsi_{\ba}(t),
\quad
\bphi_{\bomega}(t) = \Q(t)\bpsi_{\bomega}(t).
\end{equation*}
By integrating on $(0,T)$ and using the integration by parts we obtain
\begin{equation*}
\begin{split}
&\int_{0}^{T}\int_{\Omega_F(t)}{\bu^m}
\cdot\partial_{t}\bphi
\, \d\bx\d t
+\int_{0}^{T} \left(
\J{\bomega^m}\cdot\frac{\d}{\dt}\bphi_{\bomega}
+ \ba^m\cdot\frac{\d}{\dt}\bphi_{\ba}
\right)
\,\d t
\\
&- \int_{0}^{T}\int_{\Omega_F(t)}2\D{\bu^m}
\cdot\D\bphi
\, \d\bx\d t
+ \int_{0}^{T}\int_{\Omega_F(t)} (\widetilde{\bu}\otimes\bu^m)\cdot\nabla\bphi^T
\, \d\bx\d t
\\
&= -\int_{\Omega_F}\bu^m(0)\cdot\bphi(0)\, \d\bx
-\mathcal{J}{\bomega^m}(0)\cdot\bphi_{\bomega}(0)
-\ba^m(0)\cdot\bphi_{\ba}(0).
\end{split}
\end{equation*}
Now, we let $m\to\infty$
\begin{equation} \label{weak_linear}
\begin{split}
&\int_{0}^{T}\int_{\Omega_F(t)}{\bar{\bu}}
\cdot\partial_{t}\bphi
\, \d\bx\d t
+\int_{0}^{T} \left(
\J{\bar{\bomega}}\cdot\frac{\d}{\dt}\bphi_{\bomega}
+ \bar{\ba}\cdot\frac{\d}{\dt}\bphi_{\ba}
\right)
\,\d t
\\
&- \int_{0}^{T}\int_{\Omega_F(t)}2\D{\bar{\bu}}
\cdot\D\bphi
\, \d\bx\d t
+ \int_{0}^{T}\int_{\Omega_F(t)} (\widetilde{\bu}\otimes\bar{\bu})\cdot\nabla\bphi^T
\, \d\bx\d t
\\
&= -\int_{\Omega_F}\bar{\bu}(0)\cdot\bphi(0)\, \d\bx
-\mathcal{J}{\bomega}_0\cdot\bphi_{\bomega}(0)
-{\ba}_0\cdot\bphi_{\ba}(0)
\end{split}
\end{equation}
and go to the cylindrical domain
\begin{equation*}
	\begin{split}
	&\int_{0}^{T}\int_{\Omega_F}\bar{\bU}
	\cdot\partial_{t}\left(\mathbb{G}\bpsi\right)
	\, \d\by\d t
	+\int_{0}^{T}\left(
	\bar{\bA}\cdot\frac{\d}{\dt}\bpsi_{\ba}
	+ \mathcal{J}\bar{\bOmega}\cdot\frac{\d}{\dt}\bpsi_{\bomega}
	\right)\d t
	\\
	&\qquad- \int_{0}^{T}\left\langle \mathbb{G}\mathcal{F}(\bar{\bU}),\bpsi
	\right\rangle\d t
	\,+\int_{0}^{T}\left(
	G(\bar{\bA})\cdot\bpsi_{\ba}
	+H(\bar{\bOmega})\cdot\bpsi_{\bomega}
	\right)\d t
	\\
	&=-\int_{\Omega_F}
	\bu_0
	\cdot\bpsi(0)
	\, \d\by
	- \ba_0\cdot\bpsi_{\ba}(0)
	- \mathcal{J}{\bomega_0}\cdot\bpsi_{\bomega}(0)
	\end{split}
\end{equation*}
with
\begin{equation*}
\bar{\bu}(t,\bx) = \nabla \bX(t,\bY(t,\bx))\bar{\bU}(t,\bY(t,\bx)),
\quad
\bar{\ba}(t) = \Q(t)\bar{\bA}(t),
\quad
\bar{\bomega}(t) = \Q(t)\bar{\bOmega}(t).
\end{equation*}
By taking $\bpsi=h_i(t){\bPsi}_i$ and summing up over $i$ from $1$ to $m$ the equality holds for all test functions in $\mathbb{V}_m$ which is dense in $\mathbb{V}(\Omega_F)$. It is not difficult to show that the equation \eqref{weak_linear} is equivalent to \eqref{weak2} from Definition \ref{Linear} and therefore we call the function $\bar{\bU}$ a weak solution for \eqref{equation1}. Finally, the uniqueness of weak solution implies $\bar{\bU}=\bU\in L^2(0,T;H^2(\Omega))\cap H^1(0,T;L^2(\Omega))$.
Then, since $\bar{\bU}=\bar{\bA}+\bar{\bOmega}\times\by$ and ${\bU}={\bA}+{\bOmega}\times\by$ on $\overline{S_0}$, we conclude that $\bar{\bA}=\bA$ and $\bar{\bOmega}=\bOmega$ belong to $H^{1}(0,T)$.

\subsubsection{Estimates for time derivatives}
\label{Section:TimeDerivatives_estimates}

In this step we would like to show that
$$
\partial_t\widetilde{\bU},\partial_t{\bU}\in L^{\infty}(\varepsilon,T;L^2(\Omega))\cap L^{2}(\varepsilon,T;H^1(\Omega)).
$$ 
By \eqref{approx} solution $(\bU^m,\bA^m,\bOmega^m)$ satisfies
\begin{equation*}
	\int_{\Omega}\mathbb{G}\partial_{t}{\bU^m}
	\cdot{\bPsi}_i
	\, \d\by
	+ \left\langle \mathbb{G}\mathcal{F}(\bU^m),\bPsi_i\right\rangle 
	- G(\bA^m)\cdot\bPsi_i^{\ba}
	- H(\bOmega^m)\cdot\bPsi_i^{\bomega}
	= 0,
\end{equation*}
for $i=1,...m$.
This is a consequence of $\mathbb{G}=\I$ on $\overline{S_0}$ and equality
\begin{align*}
\int_{\Omega_F}\mathbb{G}\partial_{t}{\bU^m}
\cdot\bPsi_i
\, \d\by
+ \frac{\d}{\dt}\bA^m\cdot{\bPsi}_i^{\ba}
+ \frac{\d}{\dt}(\mathcal{J}{\bOmega^m})\cdot{\bPsi}_i^{\bomega}
&= \int_{\Omega_F}\mathbb{G}\partial_{t}{\bU^m}
\cdot\bPsi_i
\, \d\by 
+ \int_{S_0}\partial_{t}{\bU^m}
\cdot\bPsi_i
\, \d\by
\\
&= \int_{\Omega}\mathbb{G}\partial_{t}{\bU^m}
\cdot\bPsi_i
\, \d\by.
\end{align*}
We differentiate the equation in time
\begin{multline}
\int_{\Omega}\partial_t(\mathbb{G}\partial_{t}{\bU^m})
\cdot\bPsi_i
\, \d\by
+\left\langle \mathbb{\mathbb{G}}\mathcal{F}(\partial_t\bU^m),\bPsi_i\right\rangle 
+
\left\langle \partial_t(\mathbb{G}\mathcal{F})(\bU^m),\bPsi_i\right\rangle 
\\
- G\left(\frac{\d}{\dt}\bA^m\right)\cdot\bPsi_i^{\ba}
- H\left(\frac{\d}{\dt}\bOmega^m\right)\cdot\bPsi_i^{\bomega}
- 
\left(
G_1\left(\bA^m\right)\cdot\bPsi_i^{\ba}
+ H_1\left(\bOmega^m\right)\cdot\bPsi_i^{\bomega}
\right)
=0.
\end{multline}
We recall that  operators $G, H$ are defined by \eqref{rihgtGH}, and $G_1$ and $H_1$  by \eqref{G1H1}.
We multiply the above equation by $\frac{\d}{\dt}c_{im}$ and sum over $i$ form $1$ to $m$ to obtain
\begin{multline}	\label{equation_t1}
\int_{\Omega}\partial_t(\mathbb{G}\partial_{t}{\bU^m})
\cdot\partial_{t}{\bU^m}
\, \d\by
+\left\langle \mathbb{\mathbb{G}}\mathcal{F}(\partial_t\bU^m),\partial_{t}{\bU^m}\right\rangle 
+
\left\langle \partial_t(\mathbb{G}\mathcal{F})(\bU^m),\partial_{t}{\bU^m}\right\rangle 
\\
- G\left(\frac{\d}{\dt}\bA^m\right)\cdot\frac{\d}{\dt}\bA^m
- H\left(\frac{\d}{\dt}\bOmega^m\right)\cdot\frac{\d}{\dt}\bOmega^m
- 
\left(
G_1\left(\bA^m\right)\cdot\frac{\d}{\dt}\bA^m
+ H_1\left(\bOmega^m\right)\cdot\frac{\d}{\dt}\bOmega^m
\right)
=0.
\end{multline}
Then we integrate the equation on $[t_1,t_2]\subset(0,T]$ and estimate each term. For the first term we have
\begin{align*}
\int_{\Omega}\partial_t(\mathbb{G}\partial_{t}{\bU^m})
\cdot\partial_{t}{\bU^m}
\, \d\by
&=
\frac{1}{2}\frac{\d}{\dt}\int_{\Omega}\mathbb{G}\partial_{t}{\bU^m}
\cdot\partial_{t}{\bU^m}
\, \d\by
+\frac{1}{2}\int_{\Omega}\partial_{t}\mathbb{G}\,\partial_{t}{\bU^m}
\cdot\partial_{t}{\bU^m}
\, \d\by,
\end{align*}
\begin{align*}
\left| 
\int_{t_1}^{t_2}\int_{\Omega}\partial_{t}\mathbb{G}\,\partial_{t}{\bU^m}
\cdot\partial_{t}{\bU^m}
\, \d\by\d\tau
\right| 
&\leq\|\partial_t\mathbb{G}\|_{L^{\infty}(0,T;L^{\infty}(\Omega))}
\|\partial_t\bU^m\|_{L^2(t_1,t_2;L^2(\Omega))}^2,
\end{align*}
which implies
\begin{align*}
\int_{\Omega}\partial_t(\mathbb{G}\partial_{t}{\bU^m})
\cdot\partial_{t}{\bU^m}
\, \d\by
&\geq
\frac{1}{2}\frac{\d}{\dt}\|\nabla\bX\partial_{t}{\bU^m}\|_{L^2(\Omega)}^2 
- \frac{1}{2}\|\partial_t\mathbb{G}\|_{L^{\infty}(0,T;L^{\infty}(\Omega))}
\|\partial_t\bU^m\|_{L^2(t_1,t_2;L^2(\Omega_F))}^2.
\end{align*}
By the definition of $\left\langle \mathbb{G}\mathcal{F}(\bU),\bpsi\right\rangle $ and Lemma \ref{estimates_1} we get
\begin{multline*}
\left|\left\langle \mathbb{G}\mathcal{L}(\partial_t\bU^m),\partial_t\bU^m\right\rangle
-2\|\D(\partial_t\bU^m)\|_{L^2(\Omega)}^2
\right|\\
\leq 
\left(\|(g_{ik}g^{jl}-\delta_{ik}\delta_{jl})\|_{L^{\infty}(\Omega)} +\mu\right)\|\nabla\partial_t\bU^m\|_{L^2(\Omega_{F})}^2
+C\|\partial_t\bU^m\|_{L^2(\Omega)}^2
\end{multline*}
\begin{align*}
\left|
\int_{\Omega_F} \mathbb{G}\mathcal{M}(\partial_t\bU^m)\cdot\partial_t\bU^m\, \d\by
\right|
&\leq
\mu\|\nabla\partial_t\bU^m\|_{L^{2}(\Omega_{F})}^2
+C\|\partial_t\bU^m\|_{L^{2}(\Omega)}^2
\end{align*}
\begin{align*}
\left|\int_{\Omega_F} \mathbb{G}\mathcal{N}(\partial_t\bU^m)\cdot\partial_t\bU^m\, \d\by
\right|
&= \left|\int_{\Omega_F(t)} 
{\widetilde{\bu}}\cdot \nabla\bu_t^m \cdot \bu_t^m\, \d\by\right|
=
\left|\int_{\partial S(t)} 
\frac{1}{2}|\bu_t^m|^2\,{\widetilde{\bu}}\cdot \bn\, \d\gamma(\bx)
\right|
\\
&\leq
C\left|\int_{\partial S_0} 
\frac{1}{2}|\partial_t\bU^m|^2\,{\widetilde{\bU}}\cdot \bN\, \d\gamma(\by)
\right|
\\
&\leq C (|\widetilde{\bA}|+|\widetilde{\bOmega}|)
\left(\left|\frac{\d}{\dt}\bA^m\right|^2
+\left|\frac{\d}{\dt}\bOmega^m\right|^2\right)
\\
&\leq C \|\widetilde{\bU}\|_{L^{2}(\Omega)}
\|\partial_t\bU^m\|_{L^{2}(\Omega)}^2
\end{align*}
for arbitrary $\mu>0$ and $\bu_t^m=\nabla\bX\partial_t\bU^m$. Hence, for the second term in \eqref{equation_t1} we obtain
\begin{align*}
&\left|
\int_{t_1}^{t_2}\left\langle \mathbb{G}\mathcal{F}(\partial_t\bU^m),\partial_{t}{\bU^m}\right\rangle \d\tau
- \int_{t_1}^{t_2}2\|\D\partial_t^l\bU^m\|_{L^2(\Omega_F)}^2\,\d\tau
\right|
\\
&\qquad\leq 
\left(\|(g_{ik}g^{jl}-\delta_{ik}\delta_{jl})\|_{L^{\infty}(t_1,t_2;L^{\infty}(\Omega))}+2\mu\right)\int_{t_1}^{t_2}
\|\nabla\partial_t\bU^m\|_{L^2(\Omega_{F})}^2
\,\d\tau
+ C\int_{t_1}^{t_2}\|\partial_t\bU^m\|_{L^2(\Omega)}^2\,\d\tau.
\end{align*}
where constant $C>0$ depends on $T$ non-decreasingly.

Now, for the third term we compute
\begin{align*}
\left|\left\langle \partial_t(\mathbb{G}\mathcal{L})(\bU^m),\partial_t\bU^m\right\rangle
\right|
&\leq
\mu\|\nabla\partial_t\bU^m\|_{L^2(\Omega_{F})}^2
+C\|\partial_t\bU^m\|_{L^2(\Omega)}^2
+C\|\bU^m\|_{H^1(\Omega_{F})}^2
\end{align*}
\begin{align*}
\left|
\int_{\Omega_F} \partial_t(\mathbb{G}\mathcal{M})(\bU^m)\cdot\partial_t\bU^m\, \d\by
\right|
&\leq
C\|\partial_t\bU^m\|_{L^{2}(\Omega)}^2
+
\left(\|\widetilde{\bA}\|_{W^{1,\infty}(t_1,t_2)}
+ \|\widetilde{\bOmega}\|_{W^{1,\infty}(t_1,t_2)} \right)^2
\|\bU^m\|_{H^{1}(\Omega_{F})}^2
\end{align*}
\begin{equation*}
\left|\int_{\Omega_F} \mathbb{G}(\partial_t\mathcal{N})(\bU^m)\cdot\partial_t\bU^m\, \d\by
\right|
\leq \|\mathbb{G}\|_{L^{\infty}(\Omega)} 
\left(
\left|\int_{\Omega_F} 
\partial_t{\widetilde{\bU}}\cdot \nabla\bU^m \cdot \partial_t\bU^m\, \d\by\right|
+\right.
\end{equation*}
$$\left.\left|\int_{\Omega_F} \sum_{i,j,k=1}^n \partial_t(\Gamma ^i_{jk}\widetilde{\bU}_j)\bU_k^m\partial_t\bU_i^m\, \d\by
\right|
\right)
$$
\begin{align*}
\left|\int_{\Omega_F} 
\partial_t{\widetilde{\bU}}\cdot \nabla\bU^m \cdot \partial_t\bU^m\, \d\by\right|
&\leq \|\partial_t\widetilde{\bU}\|_{L^{4}(\Omega_F)}
\|\nabla\bU^m\|_{L^{2}(\Omega_F)}
\|\partial_t\bU^m\|_{L^{4}(\Omega_F)}
\\
&\leq C\|\partial_t\widetilde{\bU}\|_{L^{4}(\Omega_F)}^2
\|\partial_t\bU^m\|_{L^{4}(\Omega_F)}^2
+ \|\nabla\bU^m\|_{L^{2}(\Omega_F)}^2
\\
&\leq C\|\partial_t\widetilde{\bU}\|_{L^{4}(\Omega_F)}^2
\|\partial_t\bU^m\|_{L^{2}(\Omega_F)}^{\frac{3}{2}}
\|\nabla\partial_t\bU^m\|_{L^{2}(\Omega_F)}^{\frac{1}{2}}
+ \|\nabla\bU^m\|_{L^{2}(\Omega_F)}^2
\\
&\leq C\|\partial_t\widetilde{\bU}\|_{L^{4}(\Omega_F)}^8
\|\partial_t\bU^m\|_{L^{2}(\Omega_F)}^{2}
+ \mu\|\nabla\partial_t\bU^m\|_{L^{2}(\Omega_F)}^{2}
+ \|\nabla\bU^m\|_{L^{2}(\Omega_F)}^2.
\end{align*}
Hence,
\begin{multline*}
\left|
\int_{t_1}^{t_2}\left\langle \partial_{t}(\mathbb{G}\mathcal{F})(\bU^m),\partial_{t}{\bU^m}\right\rangle \d\tau
\right|
\\
\leq 
\mu\int_{t_1}^{t_2}
\|\nabla\partial_t\bU^m\|_{L^2(\Omega_{F})}^2
\,\d\tau
+ C\int_{t_1}^{t_2}\|\partial_t\bU^m\|_{L^2(\Omega)}^2\,\d\tau
+ C
\|\bU^m\|_{L^2(t_1,t_2;H^1(\Omega_{F}))}^2,
\end{multline*}
for arbitrary $\mu>0$, where constant $C>0$ depends on $T$ non-decreasingly, since $\partial_t\widetilde{\bU}\in L^8(\varepsilon,T;L^4(\Omega)),$ $\forall \varepsilon>0$ by \eqref{strong_p} and since
$\frac{\d}{\dt}\widetilde{\bA}, \frac{\d}{\dt}\widetilde{\bOmega}\in L^{\infty}(0,T)$ by the assumption of Proposition \ref{time_derivatives}. Here we emhasize that this assumption was necessary for bounding term involving $\partial_t(\mathbb{G}\mathcal{M})(\bU^m)$.

Note that 
$$
\|\nabla\bX\bpsi\|_{L^2(\Omega)}\geq \frac{1}{\|\nabla\bY\|_{L^{\infty}(0,T;L^{\infty}(\Omega))}}\|\bpsi\|_{L^2(\Omega)},
$$
since $\nabla\bY\nabla\bX=\I$, so all together, we get
\begin{align*}
&\|\partial_t\bU^m\|_{L^{2}(\Omega)}^2(t_2)
+ 4\|\D\partial_t\bU^m\|_{L^2(t_1,t_2;L^2(\Omega_F))}^2
\\
&\quad\leq
\|\partial_t\bU^m\|_{L^{2}(\Omega)}^2(t_1)
+ C\int_{t_1}^{t_2}\|\partial_t\bU^m\|_{L^{2}(\Omega)}^2\,\d\tau
\\
&\quad\qquad
+ C\left(\|(g_{ik}g^{jl}-\delta_{ik}\delta_{jl})\|_{L^{\infty}(0,T;L^{\infty}(\Omega))}+\mu\right)\int_{t_1}^{t_2}\|\nabla\partial_t^l\bU^m\|_{L^2(\Omega_F)}^2\,\d\tau
+ C
\|\bU^m\|_{L^2(t_1,t_2; H^1(\Omega_F))}^2
\end{align*}
where $\|(g_{ik}g^{jl}-\delta_{ik}\delta_{jl})\|_{L^{\infty}(0,T;L^{\infty}(\Omega))}$ is small for small $T$, by Lemma \ref{estimates_1}.

Now,
we take sufficiently small $T$, integrate the inequality on $t_1\in(\varepsilon,t_2)$
and by Gronwall's Lemma we find
\begin{align*}
&\|\partial_t\bU^m(t)\|_{L^2(\Omega)}^2
+ C\int_{\varepsilon}^{t}\|\D\partial_t\bU^m\|_{L^2(\Omega_F)}^2\,\d\tau
\leq M,
\end{align*}
for all $t\in(\varepsilon,t)$,
where constant $M>0$ depends on the norms 
$\|\partial_t\widetilde{\bU}\|_{L^{8}(\varepsilon,T;L^4(\Omega))}$,
$\|\widetilde{\bU}\|_{L^{2}(\varepsilon,T;H^1(\Omega))}$, 
$\|\bU^m\|_{L^{2}(0,T;H^1(\Omega))}$, 
$\|\widetilde{\bA}\|_{W^{1,\infty}(0,T)}$
and 
$\|\widetilde{\bOmega}\|_{W^{1,\infty}(0,T)}$.
Finally, since
$\|\bU^m\|_{L^{2}(0,T;H^1(\Omega))}$ is bounded,
on the limit we get that $\partial_t\bU\in L^2(\varepsilon,T;H^1(\Omega))\cap L^{\infty}(\varepsilon,T;L^2(\Omega))$.

\subsubsection{Uniqueness for time-derivative system}
Now we are able to show that $\bU_1^{*}=t\partial_t\bU$.
$(\bU, P)$ satisfies the following weak formulation
\begin{equation}	\label{equation_1}
\begin{split}
&\int_{\Omega_F}\partial_{t}\mathbb{G}\, \partial_{t}\bU
\cdot\bpsi
\, \d\by
-\int_{\Omega_F} \partial_{t}\mathbb{G}\,
\mathcal{F}(\bU)
\cdot \bpsi
\, \d\by
+ \int_{\Omega_F}\partial_{t}\mathbb{G}\,\mathcal{G}(P)
\cdot\bpsi
\, \d\by
= 0
\end{split}
\end{equation}
for all $\bpsi\in V(0)$, 
and $(\bU^{*}, P^{*},\bA^{*},\bOmega^{*})$ satisfies
\begin{equation}	\label{equation_2}
\begin{split}
&\int_{\Omega}\mathbb{G}\partial_{t}{\bU_1^{*}}
\cdot\bpsi
\, \d\by
+ \left\langle \mathbb{G}\mathcal{F}(\bU_1^{*}),\bpsi\right\rangle  
+ t\left(
\left\langle \mathbb{G}\mathcal{F}_1 (\bU),\bpsi\right\rangle 
+\int_{\Omega_F}\mathbb{G}\mathcal{G}_1(P)
\cdot\bpsi
\, \d\by
\right)
\\
&\qquad
- G(\bA_1^{*})\cdot\bpsi_{\ba}
- H(\bOmega_1^{*})\cdot\bpsi_{\bomega}
- t\left(
G_1\left(\bA\right)\cdot\bpsi_{\ba}
+ H_1\left(\bOmega\right)\cdot\bpsi_{\bomega}
\right)\\
&\qquad
-\int_{\Omega_F}\mathbb{G}\partial_{t}{\bU}
\cdot\bpsi\,\d\by
- \frac{\d}{\dt}\bA\cdot\bpsi_{\ba}
- \mathcal{J}\left(\frac{\d}{\dt}\bOmega\right)\cdot\bpsi_{\bomega}
=0,
\end{split}
\end{equation}
where
\begin{equation*}
\left\langle \mathbb{G}\mathcal{F}_1 (\bU),\bpsi\right\rangle 
= -\int_{\Omega_F}\mathbb{G}\mathcal{F}_1(\bU)\cdot\bpsi\,\d\by
\end{equation*}
with
$
\mathcal{F}_1 = \mathcal{L}_1 - \mathcal{M}_1 - \mathcal{N}_1
$
and operators $\mathcal{L}_1$, $\mathcal{M}_1$, $\mathcal{N}_1$, $\mathcal{G}_1$ are defined by \eqref{f2}-\eqref{f6}.
For $j\in\N$ and $(\bpsi,\bpsi_{\ba},\bpsi_{\bomega})=h(t)({\bPsi}_j,{\bPsi}_j^{\ba},{\bPsi}_j^{\bomega})$
we have
\begin{equation}	\label{equation_3}
\begin{split}
&\int_{\Omega}\mathbb{G}\partial_{t}{\partial_t\bU^m}
\cdot\bpsi
\, \d\by
+ \int_{\Omega_F}\partial_t\mathbb{G}\,{\partial_t\bU^m}
\cdot\bpsi
\, \d\by
\\
&\qquad
+ \left\langle \mathbb{\mathbb{G}}\mathcal{F}(\partial_t\bU^m),\bpsi\right\rangle 
+ \left\langle \mathbb{G}\mathcal{F}_1(\bU^m),\bpsi\right\rangle 
+ \left\langle \partial_t \mathbb{G}\,\mathcal{F}(\bU^m),\bpsi\right\rangle 
\\
&\qquad
- G\left(\frac{\d}{\dt}\bA^m\right)\cdot\bpsi_{\ba}
- H\left(\frac{\d}{\dt}\bOmega^m\right)\cdot\bpsi_{\bomega}
- 
\left(
G_1\left(\bA^m\right)\cdot\bpsi_{\ba}
+ H_1\left(\bOmega^m\right)\cdot\bpsi_{\bomega}
\right)
=0,
\end{split}
\end{equation}
where
\begin{equation*}
\left\langle \partial_t \mathbb{G}\,\mathcal{F}(\bU),\bpsi\right\rangle = -\int_{\Omega_F}\partial_t\mathbb{G}\mathcal{F}(\bU)\cdot\bpsi\,\d\by.
\end{equation*}
Note that the second row in \eqref{equation_3} is the time derivative of $\left\langle \mathbb{G}\mathcal{F}(\bU^m),\bpsi\right\rangle$, for time independent $\bpsi$, which holds from \eqref{weakF_2} since $\mathbb{G}=\I$ on $\overline{S_0}$ and $\bU^m$ is regular enough.
We multiply the above equation by $t$ and subtract from the previous one with
\begin{equation*}
	(\widehat{\bU}_1^m,\widehat{\bA}_1^m,\widehat{\bOmega}_1^m)
	= 
	\left(
	\bU_1^{*}-t\partial_t\bU^m,\,
	\bA_1^{*}-t\frac{\d}{\dt}\bA^m,\,
	\bOmega_1^{*}-t\frac{\d}{\dt}\bOmega^m
	\right).
\end{equation*}
Then, by using \eqref{equation_1}, we obtain
\begin{equation*}
\begin{split}
&\int_{\Omega}\mathbb{G}\partial_{t}\widehat{\bU}_1^m
\cdot\bpsi
\, \d\by
+ t\int_{\Omega_F}\partial_t\mathbb{G}\,
{\partial_t({\bU}-{\bU^m})}
\cdot\bpsi
\, \d\by
\\
&\qquad
- \left\langle \mathbb{G}\mathcal{F}(\widehat{\bU}_1^m),\bpsi\right\rangle  
- t\left(
	\left\langle \mathbb{G}\mathcal{F}_1 ({\bU}-{\bU^m}),\bpsi\right\rangle
	+ \left\langle \partial_t \mathbb{G}\,\mathcal{F}({\bU}-{\bU^m}),\bpsi\right\rangle  
	\right)
\\
&\qquad
{
	- t
	\int_{\Omega_F}\mathbb{G}\mathcal{G}_1(P)
	\cdot\bpsi
	\, \d\by}
-t
{
	\int_{\Omega_F}\partial_{t}\mathbb{G}\,\mathcal{G}(P)
	\cdot\bpsi
	\, \d\by}
\\
&\qquad
- G(\widehat{\bA}^m)\cdot\bpsi_{\ba}
- H(\widehat{\bOmega}^m)\cdot\bpsi_{\bomega}
\\
&\qquad
- 
{t\left(
	G_1\left(\bA-\bA^m\right)\cdot\bpsi_{\ba}
	+ H_1\left(\bOmega-\bOmega^m\right)\cdot\bpsi_{\bomega}
	\right)}
\\
&\qquad
-\int_{\Omega_F}\mathbb{G}\partial_{t}({\bU}-{\bU^m})
\cdot\bpsi
\,\d\by
- \left(\frac{\d}{\dt}(\bA-\bA^m)\right)\cdot\bpsi_{\ba}
- \left(\frac{\d}{\dt}(\bOmega-\bOmega^m)\right)\cdot\bpsi_{\bomega}.
=0
\end{split}
\end{equation*}
Since
\begin{align*}
\mathbb{G}\mathcal{G}_1(P)
+
\partial_{t}\mathbb{G}\,\mathcal{G}(P)
&= \mathbb{G}\partial_t\mathbb{G}^{-1}\nabla P
+
\partial_{t}\mathbb{G}\mathbb{G}^{-1}\nabla P
=\partial_t(\mathbb{G}\mathbb{G}^{-1})\nabla P
=0,
\end{align*}
terms with the pressure cancels,
and after integrating above equation on $(0,t)$, we get
\begin{equation*}
\begin{split}
&
\int_{\Omega}\mathbb{G}(t)\widehat{\bU}_1^m(t)
\cdot\bpsi(t)
\, \d\by
-\int_{0}^{t}\int_{\Omega}\mathbb{G}\widehat{\bU}_1^m
\cdot\partial_{t}\bpsi
\, \d\by\d\tau
-\int_{0}^{t}\int_{\Omega_F}\partial_t\mathbb{G}\widehat{\bU}_1^m
\cdot\partial_{t}\bpsi
\, \d\by\d\tau
\\
&\qquad
+ t
\int_{0}^{t}\int_{\Omega_F}\partial_t\mathbb{G}\,{\partial_t(\bU-\bU^m)}
\cdot\bpsi
\, \d\by
\,\d\tau
\\
&\qquad
- \int_{0}^{t}
\left\langle \mathbb{G}\mathcal{F}(\widehat{\bU}_1^m),\bpsi\right\rangle  
\,\d\tau
- t\int_{0}^{t}\left(
	\left\langle \mathbb{G}\mathcal{F}_1 (\bU-\bU^m),\bpsi\right\rangle
	+ \left\langle \partial_t \mathbb{G}\,\mathcal{F}(\bU-\bU^m),\bpsi\right\rangle  
	\right)\,\d\tau
\\
&\qquad
- \int_{0}^{t}\left(
G(\widehat{\bA}_1^m)\cdot\bpsi_{\ba}
+ H(\widehat{\bOmega}_1^m)\cdot\bpsi_{\bomega}
\right)\,\d\tau
\\
&\qquad
- 
{t\int_{0}^{t}\left(
	G_1\left(\bA-\bA^m\right)\cdot\bpsi_{\ba}
	+ H_1\left(\bOmega-\bOmega^m\right)\cdot\bpsi_{\bomega}
	\right)\,\d\tau}
\\
&\qquad
-\int_{0}^{t}\int_{\Omega_F}\mathbb{G}\partial_{t}({\bU}-{\bU^m})
\cdot\bpsi
\,\d\by\,\d\tau
- \int_{0}^{t}\left(\frac{\d}{\dt}(\bA-\bA^m)\right)\cdot\bpsi_{\ba}\,\d\tau
- \int_{0}^{t}\left(\frac{\d}{\dt}(\bOmega-\bOmega^m)\right)\cdot\bpsi_{\bomega}\,\d\tau
=0.
\end{split}
\end{equation*}
We let $m\to\infty$ and obtain the equation
\begin{equation*}
\begin{split}
&
\int_{\Omega}\mathbb{G}(t)\widehat{\bU}_1(t)
\cdot\bpsi(t)
\, \d\by
-\int_{0}^{t}\int_{\Omega}\widehat{\bU}_1
\cdot\partial_{t}(\mathbb{G}\bpsi)
\, \d\by\d\tau
\\
&\qquad
- \int_{0}^{t}
\left\langle \mathbb{G}\mathcal{F}(\widehat{\bU}_1),\bpsi\right\rangle\,\d\tau
- \int_{0}^{t}\left(
G(\widehat{\bA}_1)\cdot\bpsi_{\ba}
+ H(\widehat{\bOmega}_1)\cdot\bpsi_{\bomega}
\right)\,\d\tau
=0,
\end{split}
\end{equation*}
where
$$
(\widehat{\bU}_1,\widehat{\bA}_1,\widehat{\bOmega}_1)
= \left(
\bU_1^{*}-t\partial\bU,\bA_1^{*}-\frac{\d}{\dt}\bA,\bOmega_1^{*}-t\frac{\d}{\dt}\bOmega
\right)
$$
By linearity and density, the above equality holds for all $(\bpsi,\bpsi_{\ba},\bpsi_{\bomega})\in\mathbb{V}$.
Then we substitute
$$
(\bpsi,\bpsi_{\ba},\bpsi_{\bomega})
=
(\widehat{\bU}_1,\widehat{\bA}_1,\widehat{\bOmega}_1)
$$
and get the equality
\begin{multline*}
\frac{1}{2}
\left\|
\nabla\bX\widehat{\bU}_1(t)
\right\|_{L^2(\Omega)}^2 
- \int_{0}^{t}\int_{\Omega}
\nabla\bX^T\partial_t\nabla\bX\widehat{\bU}_1\cdot\widehat{\bU}_1\,\d\bx\d\tau
\\
- \int_{0}^{t}
\left\langle \mathbb{G}\mathcal{F}(\widehat{\bU}_1),\widehat{\bU}_1\right\rangle\,\d\tau
- \int_{0}^{t}\left(
G(\widehat{\bA}_1)\cdot\widehat{\bA}_1
+ H(\widehat{\bOmega}_1)\cdot\widehat{\bOmega}_1
\right)\,\d\tau
=0.
\end{multline*}
Now, as before, we can get the estimate
\begin{equation*}
\begin{split}
&\left\|
\widehat{\bU}_1(t)
\right\|_{L^2(\Omega)}^2
+ C_1\int_{0}^{t}\int_{\Omega_F}
|\D\widehat{\bU}_1|^2
\, \d\by\d\tau
\leq 
\int_{0}^{t}C_2\left\|
\widehat{\bU}_1(\tau)
\right\|_{L^2(\Omega)}^2
\d\tau
+\mu\int_{0}^{t}\int_{\Omega_F}
|\D\widehat{\bU}_1|^2
\, \d\by\d\tau
\end{split}
\end{equation*}
for all $\mu>0$ and for sufficiently small $\mu$ we get
\begin{equation*}
\begin{split}
&\left\|
\widehat{\bU}_1(t)
\right\|_{L^2(\Omega)}^2
\leq 
\int_{0}^{t}C\left\|
\widehat{\bU}_1(\tau)
\right\|_{L^2(\Omega)}^2
\d\tau.
\end{split}
\end{equation*}	
Finally, Gronwall's Lemma implies $
\widehat{\bU}_1=0$ which means that $\bU_1^{*}=t\partial_t\bU$.
Then the equations $\eqref{equation1}_1$	and $\eqref{equation2}_1$ give $\nabla P_1^{*}=t\nabla \partial_t P$.
Moreover, since $\bU_1^{*}=\bA_1^{*}+ \bOmega_1^{*}\times\by$ and $\partial_t\bU=\frac{\d}{\dt}\bA+ \frac{\d}{\dt}\bOmega \times\by$ on $\overline{S_0}$, it follows that $\bA_1^{*}=t\frac{\d}{\dt}\bA$ and $\bOmega_1^{*}=t\frac{\d}{\dt}\bOmega$.

\section{Spatial derivatives estimates}
\label{Section:spatial_derivatives}
Let $(\widetilde{\bU},\widetilde{P},\widehat{\bA},\widetilde{\bOmega})$ be a weak solution satisfying the assumption of Lemma \ref{strong}. We want to show that
\begin{equation*}
\begin{split}
&\partial_t^l\widetilde{\bU}\in L^2(\varepsilon,T;H^{k}(\Omega_F)),
\\
&\partial_t^l\widetilde{P}\in L^2(\varepsilon,T;{\quotient{H^{k-1}(\Omega_F)}{\R}}),
\end{split}
\end{equation*}
for all $l\geq 0, k\geq 2$. The case $k=2$ is exactly the statement of Proposition \ref{time_derivatives}.

As in previous sections, we consider a linear problem on cylindrical domain \eqref{FixedDomain_2} and follow the proof for the Navier-Stokes case (see eg. \cite[Section 5]{GaldiNS00}).
Let $k\geq 2$ and let us assume that solution $(\bU,P,\bA,\bOmega)$ to the system \eqref{FixedDomain_2} satisfies
\begin{equation*}
\partial_t^l\bU\in L^{2}(\varepsilon,T;H^{k}(\Omega_{F})),
\quad
\partial_t^l P\in L^{2}(\varepsilon,T;{\quotient{H^{k-1}(\Omega_F)}{\R}}),
\quad
\frac{\d^{l}}{\dt^{l}}\bA,\frac{\d^{l}}{\dt^{l}}\bOmega\in L^{2}(\varepsilon,T),
\quad\forall l\geq 0,\,\forall\varepsilon>0,
\end{equation*}
and by uniqueness
\begin{equation*}
\partial_t^l\widetilde{\bU}\in L^{2}(\varepsilon,T;H^{k}(\Omega_{F})),
\quad
\partial_t^l \widetilde{P}\in L^{2}(\varepsilon,T;{\quotient{H^{k-1}(\Omega_F)}{\R}}
),
\quad
\frac{\d^{l}}{\dt^{l}}\widetilde{\bA},\frac{\d^{l}}{\dt^{l}}\widetilde{\bOmega}\in L^{2}(\varepsilon,T),
\quad\forall l\geq 0,\,\forall\varepsilon>0,
\end{equation*}
which by Sobolev embeddings means that
\begin{equation*}
\bU,\widetilde{\bU}\in C^{\infty}((0,T];H^{k}(\Omega_{F})),
\quad
P,\widetilde{P}\in C^{\infty}((0,T];{\quotient{H^{k-1}(\Omega_F)}{\R}}),
\quad
\bA,\bOmega, \widetilde{\bA},\widetilde{\bOmega}\in C^{\infty}(0,T],
\end{equation*}
We want to show that
\begin{equation*}
\partial_t^l\bU(t)\in H^{k+1}(\Omega_{F}),
\quad
\partial_t^l P(t)\in {\quotient{H^{k}(\Omega_F)}{\R}},
\quad\forall l\geq 0,\,\forall t\in(0,T].
\end{equation*}
The solution for \eqref{FixedDomain_2} satisfies the following system
\begin{equation}	\label{FixedDomain_5}
\begin{array}{l}
\left.
\begin{split}
\triangle {\partial_t^l\bU} &= 
\partial_{t}^{l+1}{\bU} - {F}(\partial_t^l\bU,\partial_t^l P)
- {F}_l(\bU, P)
+ \nabla {\partial_t^l P}, 
\\
\divg{\partial_t^l\bU} &= 0
\end{split}
\,\right\} \;\mathrm{in}\;(0,T]\times\Omega_{F},
\\
{\partial_t^l\bU} = {\frac{\d^{l}}{\dt^{l}}\bA} + {\frac{\d^{l}}{\dt^{l}}\bOmega}\times\by \qquad \mathrm{on}\;(0,T]\times\partial S_0,
\\
{\bU} = 0\qquad \mathrm{on}\;(0,T]\times\partial\Omega,
\end{array}
\end{equation}
for all $l\geq 0$, where operators $F$ and $F_l$ are defined by \eqref{rightF} and \eqref{f2}-\eqref{f6} respectively, and if we define $F_0(\bU,P)=0$.

The idea is to use the following well known result for the steady Stokes system (see \cite[Lemma 5.2]{GaldiNS00}).
\begin{lemma}	\label{Stokes}
	Let $\Omega$ be a bounded domain of $\R^n$, of class $C^{k+2}$. For any $F\in W^{k,q}(\Omega)$, there exists one and only one solution $(\bU,P)$ to the following Stokes problem
    \begin{equation}	\label{SteadyStokes}
	\begin{array}{ll}
	\left.
	\begin{array}{l}
	-\triangle \bU  = 
	F + \nabla P, \\
	\divg{\bU} = 0
	\end{array}
	\right\} &\mathrm{in}\;\Omega,
	\\
	\qquad
	\bU = 0 &\mathrm{on}\;\partial \Omega,
	\end{array}
    \end{equation}
    such that
    \begin{equation*}
	\bU\in W^{k+2,q}(\Omega),
	\qquad P\in W^{k+1,q}(\Omega)
    \end{equation*}
    and 
    $$
	\int_{\Omega} P\,\d\by = 0.
    $$
    This solution satisfies the estimate:
    \begin{equation}
	\|\bU\|_{W^{k+2,q}(\Omega)} + \|P\|_{W^{k+1,q}(\Omega)} \leq C\|F\|_{W^{k,q}(\Omega)}
    \end{equation}
\end{lemma}
Therefore, first for $l=0$ by Lemma \ref{Stokes} and fixed point argument we will obtain that
$$
({\bU},{P})(t)\in
H^{k+1}(\Omega_{F})
\times {(\quotient{H^{k}(\Omega_F)}{\R})}, 
\qquad\forall t\in(0,T],
$$
and by uniqueness
$$
(\widetilde{\bU},\widetilde{P})(t)\in
H^{k+1}(\Omega_{F})
\times {(\quotient{H^{k}(\Omega_F)}{\R})}, 
\qquad\forall t\in(0,T].
$$
Then for $l\geq 1$ if we assume that
$$
(\partial_r\bU,\partial_r {P})(t)\in
H^{k+1}(\Omega_{F})
\times {(\quotient{H^{k}(\Omega_F)}{\R})}, 
\qquad \forall t\in(0,T],\forall r\leq l-1
$$
we will conclude that
$$
(\partial_l\bU,\partial_l {P})(t)\in
H^{k+1}(\Omega_{F})
\times {(\quotient{H^{k}(\Omega_F)}{\R})}, 
\qquad \forall t\in(0,T].
$$

First we define a smooth
divergence-free extension of the rigid velocity 
\begin{equation*}
b_{\bA,\bOmega}(t,y) = \operatorname{Ext}(\bA(t)+\bOmega(t)\times\by).
\end{equation*}
Operator $\operatorname{Ext}(\cdot)$ extends a function from solid domain $S_0$ to the domain $\Omega$ such that it preserves regularity of function and the divergence-free property. The construction of the operator can be found in \cite[Appendix A.1]{BorisSarkaAna2020weak}. 
Since $\divg b_{\bA,\bOmega} = 0$, functions
$
\bar{\bU}_l = \partial_t^l\bU - \partial_t^{l}b_{\bA,\bOmega}
$
and
$
\bar{P}_l=\partial_t^l P
$ 
satisfy
\begin{equation}	\label{FixedDomain_6}
\begin{array}{ll}
\left.
\begin{split}
\triangle \bar{\bU}_l  &= 
- {F}(\bar{\bU}_l,\bar{P}_l)
- F_l(\bU,P)
+ \partial _{t}^{l+1}{\bU}
-\mathcal{F}(\partial_t^{l}b_{\bA,\bOmega})
+ \nabla \bar{P}_l
\\
\divg \bar{\bU}_l &= 0
\end{split}
\,\right\} \;&\mathrm{in}\;(0,T]\times\Omega_{F}
\\
\quad\,\,\,\,
\bar{\bU}_l = 0 &\mathrm{on}\;(0,T]\times\partial \Omega_{F},
\end{array}
\end{equation}
for all $l\geq 0$, where
\begin{equation}\label{OpF}
\mathcal{F}(\bU)=\mathcal{L}(\bU)-\mathcal{M}(\bU)-\mathcal{N}(\bU),
\end{equation}
and $\mathcal{L}$, $\mathcal{M}$, $\mathcal{N}$ are defined by \eqref{OpL}, \eqref{OpM}, \eqref{OpN} respectively.
Now, for $l\geq 0$, we use fixed point argument, and consider the following problem 
\begin{equation}	\label{FixedDomain_7}
\begin{array}{ll}
\left.
\begin{split}
\triangle \bar{\bU}_l  &= 
- {F}(\widehat{\bU},\widehat{P}) 
- F_l(\bU,P)
+ \partial _{t}^{l+1}{\bU}
-\mathcal{F}(\partial_t^{l}b_{\bA,\bOmega})
+ \nabla \bar{P}_l
\\
\divg \bar{\bU}_l &= 0
\end{split}
\,\right\} \;&\mathrm{in}\;(0,T]\times\Omega_{F},
\\
\quad\,\,\,\,
\bar{\bU}_l = 0 \qquad
&\mathrm{on}\;(0,T]\times\partial \Omega_{F},
\end{array}
\end{equation}
with
$$
(\widehat{\bU},\widehat{P})(t)\in H^{k+1}(\Omega_{F})
\times {(\quotient{H^{k}(\Omega_F)}{\R})},
\qquad\forall t\in(0,T].
$$
By Lemma \ref{Stokes}, it is enough to show that
$$
- {F}(\widehat{\bU},\widehat{P})(t) 
- F_l(\bU,P)(t)\in H^{k-1}(\Omega_{F})
\qquad\forall t\in(0,T].
$$
since $\partial_t^{l}b_{\bA,\bOmega}\in C^{\infty}((0,T)\times\Omega)$, and $\partial_{t}^{l+1}\bU\in C^{\infty}((0,T];H^{k}(\Omega_F))$ by assumption. The only critical terms are derivatives of the convective term.
For $k=2,l=1$ we have
\begin{align*}
\|\widetilde{\bU}\cdot\nabla\partial_t\widehat{\bU}\|_{L^2(\Omega_{F})}
&\leq
\|\widetilde{\bU}\|_{L^{\infty}(\Omega_{F})}
\|\nabla\partial_t\widehat{\bU}\|_{L^2(\Omega_{F})}
\leq
\|\widetilde{\bU}\|_{H^2(\Omega_{F})}
\|\partial_t\widehat{\bU}\|_{H^1(\Omega_{F})}
\\
\|\partial_i\widetilde{\bU}\cdot\nabla\partial_t\widehat{\bU}\|_{L^2(\Omega_{F})}
&\leq
\|\partial_i\widetilde{\bU}\|_{L^{4}(\Omega_{F})}
\|\nabla\partial_t\widehat{\bU}\|_{L^4(\Omega_{F})}
\leq
\|\widetilde{\bU}\|_{H^2(\Omega_{F})}
\|\partial_t\widehat{\bU}\|_{H^2(\Omega_{F})}
\\
\|\widetilde{\bU}\cdot\nabla\partial_i\partial_t\widehat{\bU}\|_{L^2(\Omega_{F})}
&\leq
\|\widetilde{\bU}\|_{L^{\infty}(\Omega_{F})}
\|\nabla\partial_i\partial_t\widehat{\bU}\|_{L^2(\Omega_{F})}
\leq
\|\widetilde{\bU}\|_{H^2(\Omega_{F})}
\|\partial_t\widehat{\bU}\|_{H^2(\Omega_{F})}
\\
\|\partial_t\widetilde{\bU}\cdot\nabla{\bU}\|_{L^2(\Omega_{F})}
&\leq
\|\partial_t\widetilde{\bU}\|_{L^4(\Omega_{F})}
\|\nabla{\bU}\|_{L^4(\Omega_{F})}
\leq
\|\partial_t\widetilde{\bU}\|_{H^1(\Omega_{F})}
\|{\bU}\|_{H^2(\Omega_{F})}
\\
\|\partial_t\partial_i\widetilde{\bU}\cdot\nabla{\bU}\|_{L^2(\Omega_{F})}
&\leq
\|\partial_t\partial_i\widetilde{\bU}\|_{L^4(\Omega_{F})}
\|\nabla{\bU}\|_{L^4(\Omega_{F})}
\leq
\|\partial_t\widetilde{\bU}\|_{H^2(\Omega_{F})}
\|{\bU}\|_{H^2(\Omega_{F})}
\\
\|\partial_t\widetilde{\bU}\cdot\nabla\partial_i{\bU}\|_{L^2(\Omega_{F})}
&\leq
\|\partial_t\widetilde{\bU}\|_{L^{\infty}(\Omega_{F})}
\|\nabla\partial_i{\bU}\|_{L^2(\Omega_{F})}
\leq
\|\partial_t\widetilde{\bU}\|_{H^2(\Omega_{F})}
\|{\bU}\|_{H^2(\Omega_{F})}
\end{align*}
and in the general case the estimates can be obtained in the same way.

\section{Appendix}

\subsection{Time derivatives - general case}
\label{Section:time_derivatives_induction}

In Section \ref{Section:time_derivatives} we have presented the proof of Proposition \ref{time_derivatives} for case $l=1$. Here we are going to present the induction step for general $l\in\N$. The proof in general case is conceptually the same, but with more complicated expressions in the equations. 

Let $l\geq 1$ and let us assume that
\begin{equation*}
\begin{split}
&\partial_t^{l-1}\widetilde{\bU},\partial_t^{l-1}{\bU}\in W^{1,p}(\varepsilon,T;L^p(\Omega_F))\cap L^{p}(\varepsilon,T;W^{2,p}(\Omega_F)),
\\
&\partial_t^{l-1}\widetilde{P},\partial_t^{l-1}P\in L^{p}(\varepsilon,T;{\quotient{W^{1,p}(\Omega_F)}{\R}}),
\\
&\frac{\d^{l-1}}{\dt^{l-1}}\widetilde{\bA},\frac{\d^{l-1}}{\dt^{l-1}}{\bA}, \frac{\d^{l-1}}{\dt^{l-1}}\widetilde{\bOmega},\frac{\d^{l-1}}{\dt^{l-1}}{\bOmega}\in W^{1,p}(\varepsilon,T),
\end{split}
\end{equation*}
for all $\varepsilon>0$ and $1\leq p <\infty$.
We consider the problem \eqref{FixedDomain_stokes} with right hand side
\begin{equation}	\label{right3}
\begin{split}
&F^{*} = F_l^{*} = F(\bU^{*},P^{*}) + tF_l(\bU,P) + \partial_t^l\bU,
\\
&G^{*} = G_l^{*} = G(\bA^{*}) + tG_l(\bA) + \frac{\d^l}{\dt^l}\bA,
\\
&H^{*} = H_l^{*} = H(\bOmega^{*}) + tH_l(\bOmega) 
+ \mathcal{J}\frac{\d^l}{\dt^l}\bOmega,
\end{split}
\end{equation}
where
\begin{equation}	\label{rightFl}
\begin{split}
F_l(\bU,P) 
&= \sum_{p=0}^{l-1}\binom{l}{p}
\left( \mathcal{F}_{l-p}\left(\partial_t^{p}\bU\right)
- \mathcal{G}_{l-p}\left(\partial_t^{p}P
\right)
\right)\\
&= \sum_{p=0}^{l-1}\binom{l}{p}
\left( \mathcal{L}_{l-p}\left(\partial_t^{p}\bU\right)
- \mathcal{M}_{l-p}\left(\partial_t^{p}\bU
\right)
- \mathcal{N}_{l-p}\left(\partial_t^{p}\bU
\right)
- \mathcal{G}_{l-p}\left(\partial_t^{p}P
\right)
\right)    
\end{split}
\end{equation}
\begin{equation}\label{GlHl}
G_l(\bA) =
-\sum_{p=0}^{l-1}\binom{l}{p}
\frac{\d^{l-p}}{\dt^{l-p}}\widetilde{\bOmega}\times \frac{\d^{p}}{\dt^{p}}\bA,
\quad
H_l(\bOmega) =
-\sum_{p=0}^{l-1}\binom{l}{p}
\frac{\d^{l-p}}{\dt^{l-p}}\widetilde{\bOmega}\times \mathcal{J}\left(\frac{\d^{p}}{\dt^{p}}\bOmega\right).
\end{equation}
Subscript $l-p$ in operators $\mathcal{L}_{l-p}$, $\mathcal{M}_{l-p}$, $\mathcal{N}_{l-p}$ and $\mathcal{G}_{l-p}$ denotes $(l-p)$\textsuperscript{th} order time derivative of the coefficients in operators $\mathcal{L}$, $\mathcal{M}$, $\mathcal{N}$ and $\mathcal{G}$.
As for $l=1$, to show that described problem has a unique solution $(\bU_l^{*},P_l^{*},\bA_l^{*},\bOmega_l^{*})$ such that
\begin{equation} \label{condition_2}
\begin{split}
&\bU_l^{*}\in H^1(0,T;L^2(\Omega_F))\cap L^2(0,T;H^{2}(\Omega_F)),
\\
&P_l^{*}\in L^2(0,T;{\quotient{H^{1}(\Omega_F)}{\R}}),
\\
&\bA_l^{*},\bOmega_l^{*}\in H^1(0,T)
\end{split}
\end{equation}
it is sufficient to show that
\begin{equation*}
\begin{split}
&\mathcal{R} = tF_l(\bU,P)+\partial_t^l\bU,
\quad\mathcal{R}_{\ba} = tG_l(\bA)+\frac{\d^{l}}{\dt^{l}}\bA,
\quad\mathcal{R}_{\bomega} = tH_l(\bOmega)+\frac{\d^{l}}{\dt^{l}}\bOmega.
\end{split}
\end{equation*}
satisfy
\begin{equation*}
\begin{split}
&\mathcal{R}\in L^2(0,T;L^2(\Omega_{F}))
\quad\mathcal{R}_{\ba}\in L^2(0,T),
\quad\mathcal{R}_{\bomega} \in L^2(0,T).
\end{split}
\end{equation*}
All the terms can be estimated as in Section \ref{Section:time_derivatives}, so by Proposition \ref{strong2} there exists a unique strong solution $(\bU_l^{*},P_l^{*},\bA_l^{*},\bOmega_l^{*})$ of \eqref{FixedDomain_stokes} with the right hand side \eqref{right3} satisfying \eqref{condition_2}. Again, we have to prove that the obtained solution equals 
$
\left(
t\partial_t^l\bU,t\partial_t^l P,t\frac{\d^l}{\dt^l}\bA,t\frac{\d^l}{\dt^l}\bOmega
\right)
$.
\begin{lemma}	\label{Uniquenaess2}
Let $(\bU,P,\bA,\bOmega)$
be a unique strong solution for \eqref{FixedDomain_2}, and 
$(\bU_l^{*},P_l^{*},\bA_l^{*},\bOmega_l^{*})$
be a unique strong solution for \eqref{FixedDomain_stokes} with the right hand side \eqref{right3} satisfying \eqref{condition_2}. Suppose that
\begin{equation}	\label{condition_p}
\begin{split}
&\bU, \widetilde{\bU}\in W^{l-1,p}(\varepsilon,T;W^{2,p}(\Omega_F))\cap
W^{l,p}(\varepsilon,T;L^{p}(\Omega_F)),
\\
& P, \widetilde{P}\in W^{l-1,p}(\varepsilon,T;{\quotient{W^{1,p}(\Omega_F)}{\R}}),
\\
&\bA,\bOmega, 
\widetilde{\bA}, \widetilde{\bOmega}\in
W^{l,p}(\varepsilon,T)\cap W^{1,\infty}(\varepsilon,T)
\end{split}
\end{equation}
hold for some $l\in\N$ and for all $\varepsilon>0$ and all $1\leq p<\infty$.
Then
\begin{equation}	\label{unique}
(\bU_{l}^{*},P_{l}^{*},\bA_{l}^{*},\bOmega_{l}^{*}) = \left(
t\partial_t^l\bU,t\partial_t^l P,t\frac{\d^l}{\dt^l}\bA,t\frac{\d^l}{\dt^l}\bOmega
\right).
\end{equation}

\end{lemma}

\subsection{Proof of Lemma \ref{Uniquenaess2}}
\label{Section:Uniqueness2}

Let $(\bU,P,\bA,\bOmega)$
be a unique strong solution for \eqref{FixedDomain_2} satisfying \eqref{condition_p} and let 
$(\bU_{l}^{*},P_{l}^{*},\bA_{l}^{*},\bOmega_{l}^{*})$ be
a unique strong solution for \eqref{FixedDomain_stokes} with the right hand side \eqref{right3}. We want to show \eqref{unique}. Since we have already shown that the statement is valid for $l=1$ in Section \ref{Section:Uniqueness}, we can suppose that $l\geq 2$.

We use Galerkin approximations $(\bU^m,\bA^m,\bOmega^m)$, as in Section \ref{Section:Uniqueness}, and assume that 
$$
\|\bU^m\|_{W^{l-1,\infty}(\varepsilon,T;L^2(\Omega))}
+
\|\bU^m\|_{H^{l-1}(\varepsilon,T;H^1(\Omega))}
<M,
$$
for some constant $M>0$. This assumption comes from the previous step of the induction.

We want to show that
$$
\|\partial_t^l\bU^m\|_{L^{\infty}(\varepsilon,T;L^2(\Omega))}
+
\|\partial_t^l\bU^m\|_{L^{2}(\varepsilon,T;H^1(\Omega))}
<M.
$$
for some constant $M>0$,
which implies that
$$
\partial_t^l\widetilde{\bU},\partial_t^l{\bU}\in L^{\infty}(\varepsilon,T;L^2(\Omega))\cap L^{2}(\varepsilon,T;H^1(\Omega)).
$$ 
By \eqref{approx} approximation  $(\bU^m,\bA^m,\bOmega^m)$ satisfies
\begin{equation*}
	\begin{split}
	&\int_{\Omega}\mathbb{G}\partial_{t}{\bU^m}
	\cdot{\bPsi}_i
	\, \d\by
	+ \left\langle \mathbb{G}\mathcal{F}\bU^m,\bPsi_i\right\rangle 
	- G(\bA^m)\cdot\bPsi_i^{\ba}
	- H(\bOmega^m)\cdot\bPsi_i^{\bomega}
	= 0
	\end{split}
\end{equation*}
for $i=1,...m$.
We differentiate the equation in time $l$ times
\begin{equation*}
\begin{split}
&\int_{\Omega}\partial_t^l(\mathbb{G}\partial_{t}{\bU^m})
\cdot\bPsi_i
\, \d\by
- \left\langle \mathbb{\mathbb{G}}\mathcal{F}(\partial_t^l\bU^m),\bPsi_i\right\rangle 
- \sum_{k=1}^{l}\binom{l}{k}
\left\langle \partial_t^k(\mathbb{G}\mathcal{F})(\partial_t^{l-k}\bU^m),\bPsi_i\right\rangle 
\\
&\qquad
- G\left(\frac{\d^l}{\dt^l}\bA^m\right)\cdot\bPsi_i^{\ba}
- H\left(\frac{\d^l}{\dt^l}\bOmega^m\right)\cdot\bPsi_i^{\bomega}
- \sum_{k=1}^{l}\binom{l}{k}
\left(
G_k\left(\frac{\d^{l-k}}{\dt^{l-k}}\bA^m\right)\cdot\bPsi_i^{\ba}
+ H_k\left(\frac{\d^{l-k}}{\dt^{l-k}}\bOmega^m\right)\cdot\bPsi_i^{\bomega}
\right)
\\
&\qquad
=0,
\end{split}
\end{equation*}
multiply by $\frac{\d^l}{\dt^l}c_{im}$ and sum over $i$ form $1$ to $m$ to obtain
\begin{equation*}
\begin{split}
&\int_{\Omega}\partial_t^l(\mathbb{G}\partial_{t}{\bU^m})
\cdot\partial_{t}^l{\bU^m}
\, \d\by
- \left\langle \mathbb{\mathbb{G}}\mathcal{F}(\partial_t^l\bU^m),\partial_t^l\bU^m\right\rangle 
- \sum_{k=1}^{l}\binom{l}{k}
\left\langle \partial_t^k(\mathbb{G}\mathcal{F})(\partial_t^{l-k}\bU^m),\partial_t^l\bU^m\right\rangle 
\\
&\qquad
- G\left(\frac{\d^l}{\dt^l}\bA^m\right)\cdot\frac{\d^l}{\dt^l}\bA^m
- H\left(\frac{\d^l}{\dt^l}\bOmega^m\right)\cdot\frac{\d^l}{\dt^l}\bOmega^m
\\
&\qquad
- \sum_{k=1}^{l}\binom{l}{k}
\left(
G_k\left(\frac{\d^{l-k}}{\dt^{l-k}}\bA^m\right)\cdot\frac{\d^l}{\dt^l}\bA^m
+ H_k\left(\frac{\d^{l-k}}{\dt^{l-k}}\bOmega^m\right)\cdot\frac{\d^l}{\dt^l}\bOmega^m
\right)
=0.
\end{split}
\end{equation*}
Then we integrate the equation on $[t_1,t_2]\subset(0,T]$ and, in the same way as in Section \ref{Section:TimeDerivatives_estimates}, estimate
\begin{align*}
\int_{\Omega}\partial_t^l(\mathbb{G}\partial_{t}{\bU^m})
\cdot\partial_{t}^l{\bU^m}
\, \d\by
&=
\int_{\Omega}\mathbb{G}\partial_{t}^{l+1}{\bU^m}
\cdot\partial_{t}^l{\bU^m}
\, \d\by
+
\sum_{k=1}^{l}\binom{l}{k}
\int_{\Omega}\partial_t^k\mathbb{G}\,\partial_{t}^{l-k+1}{\bU^m}
\cdot\partial_{t}^l{\bU^m}
\, \d\by
\\
&=
\frac{1}{2}\frac{\d}{\dt}\int_{\Omega}\mathbb{G}\partial_{t}^{l}{\bU^m}
\cdot\partial_{t}^l{\bU^m}
\, \d\by
-\int_{\Omega}\partial_{t}\mathbb{G}\,\partial_{t}^{l}{\bU^m}
\cdot\partial_{t}^l{\bU^m}
\, \d\by
\\
&\qquad+
\sum_{k=1}^{l}\binom{l}{k}
\int_{\Omega}\partial_t^k\mathbb{G}\,\partial_{t}^{l-k+1}{\bU^m}
\cdot\partial_{t}^l{\bU^m}
\, \d\by
\end{align*}
\begin{align*}
\left| 
-\int_{t_1}^{t_2}\int_{\Omega}\partial_{t}\mathbb{G}\,\partial_{t}^{l}{\bU^m}
\cdot\partial_{t}^l{\bU^m}
\, \d\by\d\tau
+
\sum_{k=1}^{l}\binom{l}{k}
\int_{t_1}^{t_2}\int_{\Omega}\partial_t^k\mathbb{G}\,\partial_{t}^{l-k+1}{\bU^m}
\cdot\partial_{t}^l{\bU^m}
\, \d\by\d\tau
\right| 
&\leq C\|\bU^m\|_{H^l(t_1,t_2;L^2(\Omega))}^2
\end{align*}
\begin{align*}
&\left|
\int_{t_1}^{t_2}\left\langle \mathbb{G}\mathcal{F}\left(\partial_t^l\bU^m\right),\partial_{t}^l{\bU^m}\right\rangle \d\tau
- \int_{t_1}^{t_2}\int_{\Omega_{F}}2|\D\partial_t^l\bU^m|^2\,\d\by\d\tau
\right|
\\
&\qquad\leq 
\left(
\|g_{ik}g^{il}-\delta_{ik}\delta_{il}\|_{L^{\infty}(0,T;L^{\infty}(\Omega))}
+\mu
\right)
\int_{t_1}^{t_2}\int_{\Omega_{F}}|\nabla\partial_t^l\bU^m|^2\,\d\by\d\tau
+ C\int_{t_1}^{t_2}\|\partial_t^l\bU^m\|_{L^2(\Omega)}^2\,\d\tau
\end{align*}
The only difference from Section \ref{Section:TimeDerivatives_estimates} is in the following estimate
\begin{align*}
&\left|
\int_{t_1}^{t_2}\sum_{k=1}^l\binom{l}{k}\int_{\Omega_F} \partial_t^k(\mathbb{G}\mathcal{M})(\bU^m),\partial_t^l\bU^m\,\d\by\d\tau
\right|
\\
&\qquad\leq {\int_{t_1}^{t_2}}\|\partial_t^l\bU^m\|_{L^2(\Omega)}^2\,\d\tau 
+ C\left(\|\widetilde{\bA}\|_{W^{l,4}(t_1,t_2)}
+ \|\widetilde{\bOmega}\|_{W^{l,4}(t_1,t_2)}
\right)^2
\|\bU^m\|_{W^{l-2,4}(t_1,t_2;H^{1}(\Omega_F))}^2
\\
&\qquad\quad
+ C\left(\|\widetilde{\bA}\|_{W^{l-1,\infty}(t_1,t_2)}
+ \|\widetilde{\bOmega}\|_{W^{l-1,\infty}(t_1,t_2)}
\right)^2
\|\bU^m\|_{H^{l-1}(t_1,t_2;H^{1}(\Omega_F))}^2
\\
&\qquad\leq 
{\int_{t_1}^{t_2}}\|\partial_t^l\bU^m\|_{L^2(\Omega)}^2\,\d\tau + C
\|\bU^m\|_{H^{l-1}(t_1,t_2;H^{1}(\Omega_F))}^2.
\end{align*}
The last inequality follows from the fact that 
$
\widetilde{\bA}, \widetilde{\bOmega}\in
W^{l,4}(\varepsilon,T)\cap W^{1,\infty}(\varepsilon,T)
$
and embedding $W^{l-2,4}(t_1,t_2;H^{1}(\Omega_F))\hookrightarrow H^{l-1}(t_1,t_2;H^{1}(\Omega_F))$ for $l\geq 2$. Therefore, we get
\begin{align*}
&\left|
\int_{t_1}^{t_2}\left\langle \partial_t^k(\mathbb{G}\mathcal{F})(\partial_t^{l-k}\bU^m),\partial_t^l\bU^m\right\rangle\,\d\by\d\tau
\right|
\\
&\qquad\leq 
\mu
\int_{t_1}^{t_2}\int_{\Omega_{F}}|\nabla\partial_t^l\bU^m|^2\,\d\by\d\tau
+ C\int_{t_1}^{t_2}\|\partial_t^l\bU^m\|_{L^2(\Omega)}^2\,\d\tau
+ C
\|\bU^m\|_{H^{l-1}(t_1,t_2;H^{1}(\Omega_F))}^2
\end{align*}

All together, we get
\begin{align*}
&\|\partial_t^l\bU^m\|_{L^2(\Omega)}(t_2)
+ \|\D\partial_t\bU^m\|_{L^2(t_1,t_2;L^2(\Omega_F))}^2
\\
&\quad\leq
C\|\partial_t^l\bU^m\|_{L^2(\Omega)}(t_1)
+ C\int_{t_1}^{t_2}\|\partial_t^l\bU^m\|_{L^2(\Omega)}^2\d t
\\
&\qquad+ C\left(
\|g_{ik}g^{il}-\delta_{ik}\delta_{il}\|_{L^{\infty}(0,T;L^{\infty}(\Omega))}
+\mu
\right)\int_{t_1}^{t_2}\|\nabla\partial_t^l\bU^m\|_{L^2(\Omega_F)}^2\,\d\tau
+ C \|\bU^m\|_{H^{l-1}(t_1,t_2; H^1(\Omega_F))}^2,
\end{align*}
for arbitrary $\mu>0$, where $\|g_{ik}g^{il}-\delta_{ik}\delta_{il}\|_{L^{\infty}(0,T;L^{\infty}(\Omega))}$ is small for small $T$.
Now,
we take sufficiently small $T$, integrate the inequality on $t_1\in(\varepsilon,t_2)$
and by Gronwall's Lemma we find
\begin{align*}
&\|\partial_t^l\bU^m(t)\|_{L^2(\Omega)}^2
+ \int_{\varepsilon}^{t}\|\D\partial_t^l\bU^m\|_{L^2(\Omega_F)}^2\,\d\tau
\leq M,
\end{align*}
where constant $M>0$ depends on the norms $\|\widetilde{\bU}\|_{H^{l-1}(\varepsilon,T;H^1(\Omega))}$, 
$\|\bU^m\|_{H^{l-1}(0,T;H^1(\Omega))}$, 
$\|\widetilde{\bA}\|_{W^{1,\infty}(\varepsilon,T)}$, $\|\widetilde{\bA}\|_{W^{l,4}(\varepsilon,T)}$,
$\|\widetilde{\bOmega}\|_{W^{1,\infty}(\varepsilon,T)}$, $\|\widetilde{\bOmega}\|_{W^{l,4}(\varepsilon,T)}$ and $T$.
Finally, in the limit we get that 
$\partial_t^l\bU\in L^2(\varepsilon,T;H^1(\Omega))\cap L^{\infty}(\varepsilon,T;L^2(\Omega))$, for all $\varepsilon>0$.

\subsubsection{Uniqueness}
In previous section we showed that
$\partial_t^l\bU\in L^2(\varepsilon,T;H^1(\Omega))\cap L^{\infty}(\varepsilon,T;L^2(\Omega))$, for all $\varepsilon>0$. Now we are able to show that $\bU_{l}^{*}=t\partial_t^l\bU$.
We know that $(\bU, P)$ satisfies
\begin{equation}	\label{eq_1}
\begin{split}
&\int_{\Omega_F}\partial_{t}^k\mathbb{G}\, \partial_{t}^{l-k+1}\bU
\cdot\bpsi
\, \d\by
-\int_{\Omega_F} \partial_{t}^k\mathbb{G}\,
\partial_{t}^{l-k}\left(\mathcal{F}(\bU)\right)
\cdot \bpsi
\, \d\by
+ \int_{\Omega_F}\partial_{t}^k\mathbb{G}\,\partial_{t}^{l-k}(\mathcal{G}(P))
\cdot\bpsi
\, \d\by
= 0,
\quad 1\leq k \leq l,
\end{split}
\end{equation}
for all $\bpsi\in V(0)$, 
and $(\bU_{l}^{*}, P_{l}^{*},\bA_{l}^{*},\bOmega_{l}^{*})$ satisfies
\begin{equation}	\label{eq_2}
\begin{split}
&\int_{\Omega_F}\mathbb{G}\partial_{t}{\bU_{l}^{*}}
\cdot\bpsi
\, \d\by
+ \frac{\d}{\dt}\bA_{l}^{*}\cdot\bpsi_{\ba}
+ \frac{\d}{\dt}(\mathcal{J}{\bOmega_{l}^{*}})\cdot\bpsi_{\bomega}
\\
&\qquad
+ \left\langle \mathbb{G}\mathcal{F}(\bU_{l}^{*}),\bpsi\right\rangle  
+ t\sum_{k=1}^{l}\binom{l}{k}
\left(
\left\langle \mathbb{G}\mathcal{F}_k (\partial_t^{l-k}\bU),\bpsi\right\rangle 
+\int_{\Omega_F}\mathbb{G}\mathcal{G}_l( \partial_t^{l-k}P)
\cdot\bpsi
\, \d\by
\right)
\\
&\qquad
- G(\bA_{l}^{*})\cdot\bpsi_{\ba}
- H(\bOmega_{l}^{*})\cdot\bpsi_{\bomega}
- t\sum_{k=1}^{l}\binom{l}{k}\left(
G_k\left(\frac{\d^{l-k}}{\dt^{l-k}}\bA\right)\cdot\bpsi_{\ba}
+ H_k\left(\frac{\d^{l-k}}{\dt^{l-k}}\bOmega\right)\cdot\bpsi_{\bomega}
\right)\\
&\qquad
-\int_{\Omega_F}\mathbb{G}\partial_{t}^l{\bU}
\cdot\bpsi
- \left(\frac{\d^{l}}{\dt^{l}}\bA\right)\cdot\bpsi_{\ba}
- \left(\frac{\d^{l}}{\dt^{l}}\bOmega\right)\cdot\bpsi_{\bomega}
=0,
\end{split}
\end{equation}
where 
\begin{equation*}
\left\langle \mathbb{G}\,\mathcal{F}_k(\partial_{t}^{l-k}\bU),\bpsi\right\rangle
=- \int_{\Omega_F}\mathbb{G}\,\mathcal{F}_k(\partial_{t}^{l-k}\bU)
\cdot\bpsi
\, \d\by,\quad 1\leq k\leq l.
\end{equation*}
For  $(\bpsi,\bpsi_{\ba},\bpsi_{\bomega})=h(t)({\bPsi}_j,{\bPsi}_j^{\ba},{\bPsi}_j^{\bomega})$
we have
\begin{equation}	\label{eq_3}
\begin{split}
&\int_{\Omega_F}\mathbb{G}\partial_{t}{\partial_t^l\bU^m}
\cdot\bpsi
\, \d\by
+ \sum_{k=1}^{l}\binom{l}{k}\int_{\Omega_F}\partial_t^k\mathbb{G}\,{\partial_t^{l-k+1}\bU^m}
\cdot\bpsi
\, \d\by
+ \frac{\d^l}{\dt^l}\mathcal{J}\left(\frac{\d}{\dt}\bOmega^m\right)\cdot\bpsi_{\bomega}
+ \frac{\d^{l+1}}{\dt^{l+1}}\bA^m\cdot\bpsi_{\ba}
\\
&\qquad+ \left\langle \mathbb{\mathbb{G}}\mathcal{F}(\partial_t^l\bU^m),\bpsi\right\rangle 
+ \sum_{k=1}^{l}\binom{l}{k}
\left\langle \mathbb{G}\mathcal{F}_k(\partial_t^{l-k}\bU^m),\bpsi\right\rangle 
+ \sum_{k=1}^{l}\binom{l}{k}
\left\langle \partial_t^k \mathbb{G}\,\partial_t^{l-k}\left(\mathcal{F}(\bU^m)\right),\bpsi\right\rangle 
\\
&\qquad
- G\left(\frac{\d^l}{\dt^l}\bA^m\right)\cdot\bpsi_{\ba}
- H\left(\frac{\d^l}{\dt^l}\bOmega^m\right)\cdot\bpsi_{\bomega}
- \sum_{k=1}^{l}\binom{l}{k}
\left(
G_k\left(\frac{\d^{l-k}}{\dt^{l-k}}\bA^m\right)\cdot\bpsi_{\ba}
+ H_k\left(\frac{\d^{l-k}}{\dt^{l-k}}\bOmega^m\right)\cdot\bpsi_{\bomega}
\right)
\\
&\qquad
=0,
\end{split}
\end{equation}
where
\begin{equation*}
\left\langle \partial_t^k\mathbb{G}\,\partial_t^{l-k}\left(\mathcal{F}(\bU)\right),\bpsi\right\rangle
=- \int_{\Omega_F}\partial_t^k\mathbb{G}\,\partial_t^{l-k}\left(\mathcal{F}(\bU)\right)
\cdot\bpsi
\, \d\by,\quad 1\leq k\leq l.
\end{equation*}
We multiply the above equation by $t$ and subtract from the previous one with
\begin{equation*}
	(\widehat{\bU}^m,\widehat{\bA}^m,\widehat{\bOmega}^m)
	= 
	\left(
	\bU_{l}^{*}-t\partial_t^l\bU^m,\,
	\bA_{l}^{*}-t\frac{\d^l}{\dt^l}\bA^m,\,
	\bOmega_{l}^{*}-t\frac{\d^l}{\dt^l}\bOmega^m
	\right)
\end{equation*}
Then, by using \eqref{eq_1}, we obtain
\begin{equation*}
\begin{split}
&\int_{\Omega_F}\mathbb{G}\partial_{t}\widehat{\bU}^m
\cdot\bpsi
\, \d\by
+ t\sum_{k=1}^{l}\binom{l}{k}\int_{\Omega_F}\partial_t^k\mathbb{G}\,{\partial_t^{l-k+1}({\bU}-{\bU^m})}
\cdot\bpsi
\, \d\by
+ \frac{\d}{\dt}\widehat{\bA}^m\cdot\bpsi_{\ba}
+ \frac{\d}{\dt}(\mathcal{J}{\widehat{\bOmega}^m})\cdot\bpsi_{\bomega}
\\
&\qquad
- \left\langle \mathbb{G}\mathcal{F}(\widehat{\bU}^m),\bpsi\right\rangle  
- t\sum_{k=1}^{l}\binom{l}{k}
	\left(
	\left\langle \mathbb{G}\mathcal{F}_k (\partial_t^{l-k}({\bU}-{\bU^m})),\bpsi\right\rangle
	+ \left\langle \partial_t^k \mathbb{G}\,\partial_t^{l-k}\left(\mathcal{F}({\bU}-{\bU^m})\right),\bpsi\right\rangle  
	\right)
\\
&\qquad
{
	- t\sum_{k=1}^{l}\binom{l}{k}
	\int_{\Omega_F}\mathbb{G}\mathcal{G}_l (\partial_t^{l-k}P)
	\cdot\bpsi
	\, \d\by}
-t\sum_{k=1}^{l}\binom{l}{k}
{
	\int_{\Omega_F}\partial_{t}^k\mathbb{G}\,\partial_{t}^{l-k}(\mathcal{G}(P))
	\cdot\bpsi
	\, \d\by}
\\
&\qquad
- G(\widehat{\bA}^m)\cdot\bpsi_{\ba}
- H(\widehat{\bOmega}^m)\cdot\bpsi_{\bomega}
\\
&\qquad
- 
{t\sum_{k=1}^{l}\binom{l}{k}\left(
	G_k\left(\frac{\d^{l-k}}{\dt^{l-k}}(\bA-\bA^m)\right)\cdot\bpsi_{\ba}
	+ H_k\left(\frac{\d^{l-k}}{\dt^{l-k}}(\bOmega-\bOmega^m)\right)\cdot\bpsi_{\bomega}
	\right)}
\\
&\qquad
-\int_{\Omega_F}\mathbb{G}\partial_{t}^l({\bU}-{\bU^m})
\cdot\bpsi
\,\d\by
- \left(\frac{\d^{l}}{\dt^{l}}(\bA-\bA^m)\right)\cdot\bpsi_{\ba}
- \left(\frac{\d^{l}}{\dt^{l}}(\bOmega-\bOmega^m)\right)\cdot\bpsi_{\bomega}
=0
\end{split}
\end{equation*}
It can be shown that the terms with the pressure cancels, i.e. it holds
\begin{equation}
\sum_{k=1}^{l}\binom{l}{k}
\mathbb{G}\mathcal{G}_l (\partial_t^{l-k}P)
+\sum_{k=1}^{l}\binom{l}{k}
\partial_{t}^k\mathbb{G}\,\partial_{t}^{l-k}(\mathcal{G}(P))
= 0,
\end{equation}
and after integrating above equation on $(0,t)$, we get
\begin{equation*}
\begin{split}
&
\int_{\Omega_F}\mathbb{G}(t)\widehat{\bU}^m(t)
\cdot\bpsi(t)
\, \d\by
-\int_{0}^{t}\int_{\Omega_F}\mathbb{G}\widehat{\bU}^m
\cdot\partial_{t}\bpsi
\, \d\by\d\tau
-\int_{0}^{t}\int_{\Omega_F}\partial_t\mathbb{G}\widehat{\bU}^m
\cdot\partial_{t}\bpsi
\, \d\by\d\tau
\\
&\qquad
+ t\sum_{k=1}^{l}\binom{l}{k}
\int_{0}^{t}\int_{\Omega_F}\partial_t^k\mathbb{G}\,{\partial_t^{l-k+1}(\bU-\bU^m)}
\cdot\bpsi
\, \d\by
\,\d\tau
\\
&\qquad
+\widehat{\bA}^m(t)\cdot\bpsi_{\ba}(t)
+ \mathcal{J}{\widehat{\bOmega}^m}(t)\cdot\bpsi_{\bomega}(t)
-
\int_{0}^{t}
\left(
\widehat{\bA}^m\cdot\frac{\d}{\dt}\bpsi_{\ba}
+ \mathcal{J}{\widehat{\bOmega}^m}\cdot\frac{\d}{\dt}\bpsi_{\bomega}
\right)\,\d\tau
\\
&\qquad
- \int_{0}^{t}
\left\langle \mathbb{G}\mathcal{F}(\widehat{\bU}^m),\bpsi\right\rangle  
\,\d\tau
- t\sum_{k=1}^{l}\binom{l}{k}
	\int_{0}^{t}\left(
	\left\langle \mathbb{G}\mathcal{F}_k (\partial_t^{l-k}(\bU-\bU^m)),\bpsi\right\rangle
	+ \left\langle \partial_t^l \mathbb{G}\,\partial_t^{l-k}\mathcal{F}(\bU-\bU^m),\bpsi\right\rangle  
	\right)\,\d\tau
\\
&\qquad
- \int_{0}^{t}\left(
G(\widehat{\bA}^m)\cdot\bpsi_{\ba}
+ H(\widehat{\bOmega}^m)\cdot\bpsi_{\bomega}
\right)\,\d\tau
\\
&\qquad
- 
{t\sum_{k=1}^{l}\binom{l}{k}
	\int_{0}^{t}\left(
	G_k\left(\frac{\d^{l-k}}{\dt^{l-k}}(\bA-\bA^m)\right)\cdot\bpsi_{\ba}
	+ H_k\left(\frac{\d^{l-k}}{\dt^{l-k}}(\bOmega-\bOmega^m)\right)\cdot\bpsi_{\bomega}
	\right)\,\d\tau}
\\
&\qquad
-\int_{0}^{t}\int_{\Omega_F}\mathbb{G}\partial_{t}^l({\bU}-{\bU^m})
\cdot\bpsi
\,\d\by\,\d\tau
- \int_{0}^{t}\left(\frac{\d^{l}}{\dt^{l}}(\bA-\bA^m)\right)\cdot\bpsi_{\ba}\,\d\tau
- \int_{0}^{t}\left(\frac{\d^{l}}{\dt^{l}}(\bOmega-\bOmega^m)\right)\cdot\bpsi_{\bomega}\,\d\tau
=0.
\end{split}
\end{equation*}
We let $m\to\infty$ and obtain the equation
\begin{equation*}
\begin{split}
&
\int_{\Omega_F}\mathbb{G}(t)\widehat{\bU}(t)
\cdot\bpsi(t)
\, \d\by
-\int_{0}^{t}\int_{\Omega_F}\widehat{\bU}
\cdot\partial_{t}(\mathbb{G}\bpsi)
\, \d\by\d\tau
\\
&\qquad
+\widehat{\bA}(t)\cdot\bpsi_{\ba}(t)
+ \mathcal{J}{\widehat{\bOmega}}(t)\cdot\bpsi_{\bomega}(t)
-
\int_{0}^{t}
\left(
\widehat{\bA}\cdot\frac{\d}{\dt}\bpsi_{\ba}
+ \mathcal{J}{\widehat{\bOmega}}\cdot\frac{\d}{\dt}\bpsi_{\bomega}
\right)\,\d\tau
\\
&\qquad
- \int_{0}^{t}
\left\langle \mathbb{G}\mathcal{F}(\widehat{\bU}),\bpsi\right\rangle\,\d\tau
- \int_{0}^{t}\left(
G(\widehat{\bA})\cdot\bpsi_{\ba}
+ H(\widehat{\bOmega})\cdot\bpsi_{\bomega}
\right)\,\d\tau
=0
\end{split}
\end{equation*}
where
$$
(\widehat{\bU},\widehat{\bA},\widehat{\bOmega})
= \left(
\bU_{l}^{*}-t\partial^l\bU,\bA_{l}^{*}-\frac{\d^l}{\dt^l}\bA,\bOmega_{l}^{*}-t\frac{\d^l}{\dt^l}\bOmega
\right).
$$
By the linearity and the density, the above equality holds for all $(\bpsi,\bpsi_{\ba},\bpsi_{\bomega})\in\mathbb{V}$.
Then we can substitute
$$
(\bpsi,\bpsi_{\ba},\bpsi_{\bomega})
=
(\widehat{\bU},\widehat{\bA},\widehat{\bOmega})
$$
and get the equality
\begin{multline*}
\frac{1}{2}
\left\|(\nabla\bX
\widehat{\bU})(t)
\right\|_{L^2(\Omega)}^2 
-\int_{0}^t\int_{\Omega}\nabla\bX^T\partial_t\nabla\bX\widehat{\bU}\cdot\widehat{\bU}\,\d\by\d\tau
\\
- \int_{0}^{t}
\left\langle \mathbb{G}\mathcal{F}(\widehat{\bU}),\widehat{\bU}\right\rangle\,\d\tau
- \int_{0}^{t}\left(
G(\widehat{\bA})\cdot\widehat{\bA}
+ H(\widehat{\bOmega})\cdot\widehat{\bOmega}
\right)\,\d\tau
=0.
\end{multline*}
Now, as in Section \ref{Section:time_derivatives}, we can get the estimate
\begin{equation*}
\begin{split}
&\left\|
\widehat{\bU}(t)
\right\|_{L^2(\Omega)}^2
+ C_1\int_{0}^{t}\int_{\Omega_F}
|\D\widehat{\bU}|^2
\, \d\by\d\tau
\leq 
\int_{0}^{t}C\left\|
\widehat{\bU}(\tau)
\right\|_{L^2(\Omega)}^2
\d\tau
+\mu\int_{0}^{t}\int_{\Omega_F}
|\D\widehat{\bU}|^2
\, \d\by\d\tau
\end{split}
\end{equation*}
for all $\mu>0$ and for sufficiently small $\mu$ we get
\begin{equation*}
\begin{split}
&\left\|
\widehat{\bU}(t)
\right\|_{L^2(\Omega)}^2
\leq 
\int_{0}^{t}C\left\|
\widehat{\bU}(\tau)
\right\|_{L^2(\Omega)}^2
\d\tau.
\end{split}
\end{equation*}	
Finally, Gronwall's Lemma implies $\widehat{\bU}=0$ which means that $\bU_{l}^{*}=t\partial_t^l\bU$.
Then the equations for $\bU$ and $\bU_l^{*}$ give $\nabla P_l^{*}=t\nabla \partial_t P$, and since $\bU_l^{*}=\bA_l^{*}+ \bOmega_l^{*}\times\by$ and $\partial_t\bU=\frac{\d^l}{\dt^l}\bA+ \frac{\d^l}{\dt^l}\bOmega\times\by$ on $\overline{S_0}$, it follows that $\bA_l^{*}=t\frac{\d^l}{\dt^l}\bA$ and $\bOmega_l^{*}=t\frac{\d^l}{\dt^l}\bOmega$.

\section{Notation}
\renewcommand{\arraystretch}{1.5}
\begin{longtable}	
	{| p{.20\textwidth} | p{.50\textwidth} | p{.20\textwidth} |} 
	\hline
	\textbf{Label} & \textbf{Description} &
	\textbf{definition/1st appearance} 
	\\ \hline\hline 
	$(\widetilde{\bu}, \widetilde{p},\widetilde{\bomega},\widetilde{\ba})$ 
	& solution for the original nonlinear problem \eqref{FSINoslip} on physical domain 
	&  Section \ref{Section:Linear} 
	\\ \hline
	$(\bu, p,\bomega,\ba)$ 
	& solution for the linear problem \eqref{Linear} on the physical domain 
	& Section \ref{Section:Linear} 
	\\ \hline
	$(\widetilde{\bU}, \widetilde{P},\widetilde{\bOmega},\widetilde{\bA})$ 
	& solution for the nonlinear problem on the cylindrical domain 
	& Section \ref{Section:Transformed}, equation \eqref{ChangeOfVariables} 
	\\ \hline
	$(\bU, P,\bOmega,\bA)$ 
	& solution for the linear problem \eqref{FixedDomain_2} on the cylindrical domain 
	& Section \ref{Section:Transformed}, equation \eqref{ChangeOfVariables} 
	\\ \hline
	$(\bphi, \bphi_{\bomega}, \bphi_{\ba})$ 
	& test function on the physical domain 
	&  Definition \ref{definition}
	\\ \hline
	$(\bpsi, \bpsi_{\bomega}, \bpsi_{\ba})$ 
	& test function on the cylindrical domain 
	&  Section \ref{Section:Uniqueness}
	\\ \hline
	$\bX(t,\by)$ 
	& changes of variables 
	& Section \ref{Section:Transformed} 
	\\ \hline
	$\bY(t,\bx)$ 
	& changes of variables 
	& Section \ref{Section:Transformed} 
	\\ \hline
	$F(\bU,P)$, $G(\bA), H(\bOmega)$ 
	& right-hand side of the linear problem \eqref{FixedDomain_2} on the cylindrical domain, 
	& equations \eqref{rightF} and \eqref{rihgtGH}
	\\
	&
	$F(\bU,P) = (\mathcal{L}-\Delta)\bU-\mathcal{M}\bU-\mathcal{N}\bU-(\mathcal{G}-\nabla)P$
	&  
	\\ \hline
	$\mathcal{L}\bU$ 
	& the transformed Laplace operator
	& equation \eqref{OpL} 
	\\ \hline
	$\mathcal{M}\bU$ 
	& the transformation of time derivative and gradient
	& equation \eqref{OpM} 
	\\ \hline
	$\mathcal{N}\bU$ 
	& the transformation of the convection term 
	& equation \eqref{OpN} 
	\\ \hline
	$\mathcal{G}P$ 
	& $\mathcal{G}P=\nabla\bY\nabla\bY^T\nabla P$,  the transformation of the gradient of the pressure
	& equation \eqref{OpP}
	\\ \hline
	$\mathcal{F}\bU$ 
	& $\mathcal{F}\bU = \mathcal{L}\bU-\mathcal{M}\bU-\mathcal{N}\bU-\mathcal{G}P$ 
	& equation \eqref{rightF2} 
	\\ \hline
	$({\bU}^{*}, {P}^{*},{\bOmega}^{*},{\bA}^{*})$ 
	& the fixed point, the solution for the transformed problem
	& Section \ref{Section:StrongSolution} \\ \hline
	$(\widehat{\bU}, \widehat{P},\widehat{\bOmega},\widehat{\bA})$ 
	&  the fixed point, functions on the right-hand side
	& Section \ref{Section:StrongSolution} 
    \\ \hline
	$F^{*}$, $G^{*}$, $H^{*}$  
	& the right-hand side for the Stokes problem 
	& Section \ref{Section:StrongSolution} \\ \hline
	$X_{p,q}^T$, $Y_{p,q}^T$
	& $X_{p,q}^T :=W^{1,p}(0,T;L^q(\Omega_{F}))\cap L^p(0,T;W^{2,q}(\Omega_F))$ $Y_{p,q}^T := L^p(0,T; {W}^{1,q}(\Omega_{F}))$
	& Section \ref{Section:StrongSolution}, Theorem \ref{regularity_fixedPoint}\\ \hline
	$F_l(\bU,P)$ & $"F_l(\bU,P)=\partial_{t}^l(F({\bU},{P}))-F(\partial_{t}^l{\bU},\partial_{t}^l{P})"$ & \eqref{f2}-\eqref{f6} \\ \hline
	$G_l(\bA)$ 
    & $"G_l(\bA)=\frac{d^l}{{dt}^l}(G(\bA))-G\left(\frac{d^l}{{dt}^l}\bA\right)"$
 & \eqref{G1H1} ($l=1$) \& \eqref{GlHl} (general case) \\ \hline
	$H_l(\bOmega)$ 
     & $"H_l(\bOmega)=\frac{d^l}{{dt}^l}(H(\bOmega))-H\left(\frac{d^l}{{dt}^l}\bOmega\right)"$
    & \eqref{G1H1} ($l=1$) \eqref{GlHl} (general case) \\ \hline
	$\mathcal{G}_l(P)$ 
    &  the operator obtained by taking $l$\textsuperscript{th} order time derivative of the coefficients in operator $\mathcal{G}$, i.e. \mbox{$\mathcal{G}_l(P)=\partial_t^l(\nabla\bY\nabla\bY^T)\nabla P$ }
    & Section \ref{Section:Uniqueness} (l=1), Appendix \ref{Section:time_derivatives_induction} (general case)\\ \hline
    $\mathcal{F}_l(\bU)$
	&  $\mathcal{F}_l(\bU) = \mathcal{L}_l(\bU)-\mathcal{M}_l(\bU)-\mathcal{N}_l(\bU)-\mathcal{G}_l(P)$,
 $\mathcal{L}_l, \mathcal{M}_l, \mathcal{N}_l$ are operators obtained by taking $l$\textsuperscript{th} order time derivative of the coefficients in operators $\mathcal{L}, \mathcal{M}, \mathcal{N}$
	& Section \ref{Section:Uniqueness} (l=1), Appendix \ref{Section:time_derivatives_induction} (general case) 
	\\ \hline
	$\mathbb{G}$ & $\mathbb{G}=\nabla\bX^T\nabla\bX$ & 
	Section \ref{Section:Uniqueness} \\ \hline
\end{longtable}

{\bf Acknowledgments.} The authors express their sincere gratitude to the anonymous referee for providing a detailed and insightful review of the manuscript. The suggestions and comments provided by the referee helped us to improve the clarity and quality of the paper.

{\bf Conflict of interest:}

\v{S}\'arka Ne\v{c}asov\'a as 
 the corresponding author declares on behalf of all authors, that there is no conflict of interest.

{\bf Declarations:}
'Not applicable' for the whole manuscript.

{\bf Data availability statement:}
There are no associated data corresponding to the manuscript.

\end{document}